\documentclass[reqno,11pt]{amsart}

\usepackage[numbers,sort&compress]{natbib}
\usepackage{color}
\usepackage{graphicx}
\usepackage{tikz}
\usepackage{rotating}

\usepackage{geometry}
\usepackage{mathtools,amssymb,amsthm,mathrsfs,color,lineno,paralist,graphicx,float}
\usepackage[colorlinks,
linkcolor=blue,
anchorcolor=green,
citecolor=blue, 
]{hyperref}

\newtheorem{theorem}{Theorem}[section]
\newtheorem{lemma}[theorem]{Lemma}
\newtheorem{proposition}[theorem]{Proposition}
\newtheorem{corollary}[theorem]{Corollary}
\newtheorem{remark}[theorem]{Remark}
\newtheorem{example}{Example}
\newtheorem{definition}[theorem]{Definition}

\newcommand{\C}{{\mathbb C}}
\newcommand{\D}{{\mathbb D}}
\newcommand{\E}{{\mathbb E}}

\newcommand{\J}{{\mathbb J}}

\newcommand{\N}{{\mathbb N}}
\newcommand{\Q}{{\mathbb Q}}
\newcommand{\R}{{\mathbb R}}
\newcommand{\T}{{\mathbb T}}
\newcommand{\V}{{\mathbb V}}

\newcommand{\Z}{{\mathbb Z}}

\def\qz#1{\textcolor{magenta}{#1}}


\catcode`\@=12

\def\Empty{}
\newcommand\oplabel[1]{
	\def\OpArg{#1} \ifx \OpArg\Empty {} \else
	\label{#1}
	\fi}
\begin{document}
	\title[Mobility edge]{ Exact new mobility edges}
	\author{Yongjian Wang}
	\address{School of Mathematics and Statistics, Nanjing University of Science and Technology, Nanjing 210094, China}\email{wangyongjian@amss.ac.cn}
	\author{Qi Zhou}
\address{Chern Institute of Mathematics and LPMC, Nankai University, Tianjin 300071, China}
\email{qizhou@nankai.edu.cn}
	\date{\today}
	\begin{abstract}
  Mobility edges (ME), defined as critical energies that separate the extended states from the localized states, are a significant topic in quantum physics. In this paper, we demonstrate the existence of two  exact new mobility edges for two physically realistic models: the first, referred to as Type II ME, represents the critical energy that separates the critical states from localized states; the second, referred to as Type III ME, marks the critical energy that separate the critical states  from extended states. The proof is based on spectral analysis of singular Jacobi operator on the strip.
\end{abstract}

\setcounter{tocdepth}{1}

\maketitle

\section{Introduction}
Anderson localization is a fundamental quantum phenomenon characterized by  pure point spectrum, with eigenfunctions that exhibit exponential decay. In 1D and 2D disordered systems, all quantum states are localized, irrespective of the disorder strength. However, in 3D systems, localized and extended states can coexist across different energy levels \cite{Anderson}. A critical energy, known as the {\em mobility edge} (ME), as defined by Mott, separates localized states from extended states, illustrating the coexistence of these states within the system. ME is central to the Anderson metal-insulator transition, a phenomenon of great significance in condensed matter physics. Currently, ME remains one of the most prominent research topics in the field, serving as a fundamental aspect for numerous key problems \cite{EM}. Despite its critical importance, a rigorous proof of the ME remains an open question \cite{S1}. This question is so vital that Simon \cite{S1} included it as Problem 1 in his list of Schr\"odinger operator problems for the twenty-first century.

A breakthrough in this area was achieved through the manipulation of ultra-cold atoms, which provided a novel and well-controlled tool for directly observing the mobility edge \cite{B,Ro}. Consequently, there is a growing interest in exploring the mobility edge in 1D quasiperiodic models, particularly concerning its exact characterization to better understand the extended-localized transition and to advance the detailed study of fundamental mobility edge physics \cite{Roeck2016,Gao2019}.

Moreover, quasiperiodic systems can host a third type of state known as critical states, which correspond to  singular continuous spectrum in spectral theory. These critical states are fundamentally distinct from the localized and extended states observed in spectral statics \cite{Geisel}, multifractal properties \cite{Dubertrand}, and dynamical evolution \cite{Ketzmerick}. However, the underlying mechanism of critical states remains enigmatic, especially in light of their possible coexistence with localized and extended states. To clarify this, we refer to the traditional mobility edge as Type I ME and introduce the following:

\begin{itemize} 
\item Type II ME: the critical energy separating critical states (singular continuous spectrum) from localized states (Anderson localization). 
\item Type III ME: the critical energy separating critical states (singular continuous spectrum) from extended states (absolutely continuous spectrum). 
\end{itemize}

Once we introduce this, the question is whether there are some (physical) quasiperiodic models that exhibit exact Type II and Type III ME? Furthermore, is there a  universal mechanism to produce these MEs? In this paper, we will investigate this issue. 

\subsection{Main results}
\subsubsection{Type I ME}

As mentioned earlier,  in the random case, the proof of the Type I ME itself remains  unresolved. This is despite the significant progress made in establishing the existence of localization near the edges of the spectrum \cite{BK,DS,LZ}.
In contrast, the quasiperiodic case has seen extensive investigation into the coexistence of absolutely continuous spectrum and Anderson localization.  Here, the quasiperiodic models are defined as quasiperiodic Jacobi operator on $\ell^2(\Z)$:
\begin{equation*}
[\mathcal{H}_{c,v,\theta}u](n) := c(\theta+n\alpha)u(n+1) + c(\theta+(n-1)\alpha)u(n-1) + v(\theta+n\alpha)u(n),
\end{equation*}
where $c(\cdot)\in C^{\omega}(\T^d,\R)$ is the hopping, $v(\cdot)\in C^{\omega}(\T^d,\R)$ is the potential, and $\alpha\in \T^d$ is rationally independent.  Bourgain \cite{bourgain} constructed a quasiperiodic model with two frequencies that exhibits this coexistence of absolutely continuous spectrum and Anderson localization. Bjerkl\"{o}v \cite{Bj} provided a specific example with the potential $$
v(\theta) = \exp(Kf(\theta + \alpha)) + \exp(-Kf(\theta)),
$$ which results in the coexistence of both positive and zero Lyapunov exponents in the spectrum for arbitrarily large $K$. Furthermore, Zhang \cite{zhang} proved the coexistence of absolutely continuous spectrum and Anderson localization for certain parameter values. Recently, Bjerkl\"{o}v and Krikorian \cite{BK2021} constructed a class of ``peaky" potentials, demonstrating that the corresponding Schr\"{o}dinger operator exhibits both absolutely continuous spectrum and Anderson localization. Additionally, Wang et al. \cite{wang} further validated that this model manifests exact Type I ME.

Indeed, mixed spectra are expected to be general phenomenon of  one-frequency operators \cite{A4}. However, constructing explicit examples, particularly those that represent physically realizable models exhibiting exact Type I ME, poses significant challenges. To date, only two models with Type I MEs have been experimentally realized \cite{GPD,WXYZ,Gao2024,luschen2018}. The first model is known as the Generalized Aubry-Andr\'e (GAA) model \cite{GPD}, while the second is referred to as the Mosaic model \cite{WXYZ}:
\begin{equation}\label{mosaic}
		[\mathcal{H}_{1,v,\theta}u](n)=u(n+1)+u(n-1)+v_{\theta}(n)u(n),
	\end{equation}
	where
	\begin{equation*}\quad
		v_{\theta}(n)=\left\{\begin{matrix}2\lambda \cos2\pi(\theta+n\alpha),&n\in \mathbb{Z}_2,\\ 0,&else,\end{matrix}\right.\quad\lambda>0.
	\end{equation*}
Recently, we \cite{wang} rigorously proved the existence of these MEs and provided precise locations for them in these two physical models.

\subsubsection{Type II ME}
Type II ME can be traced back to Simon's proposal of the random decaying model \cite{Kiselev1997}, which is defined by Schr\"{o}dinger operators with a potential given by $ v_\omega(n) = \lambda n^{-\tau} a_n(\omega) $, where $ a_n $ are independent and identically distributed (i.i.d.) random variables. If $\tau=\frac{1}{2}$, the corresponding Schr\"odinger operator exhibits exact Type II ME. 

In the quasiperiodic case, the Maryland model, defined by $ v(\theta) = \lambda \tan(\pi \theta) $, is the first quasiperiodic model for which the spectral decomposition has been established for all parameters \cite{HJY,jiliu2017, JLiu}. More precisely, for the index $ \delta(\alpha,\theta) \in [-\infty, \infty] $ defined by
$$
\delta(\alpha, \theta) = \limsup_{n \to \infty} \frac{\ln\|q_n(\theta - 1/2)\|_{\mathbb{T}} + \ln q_{n+1}}{q_n},
$$
arithmetic spectral transitions or Type II MEs occur at
$$
E_c = \pm \coth \delta(\alpha,\theta) \sqrt{4\sinh^2 \delta(\alpha,\theta) - \lambda^2}.
$$
Here, $\frac{p_n}{q_n}$ is the best rational approximations of the frequency $\alpha$. Type II ME only occur if $\alpha$ is Liouvillean, as if $\alpha$ is Diophantine, 
 the index $ \delta(\alpha,\theta) $ vanishes for any $ \theta $, resulting in no phase transition. Recall that $ \alpha $ is  Diophantine (denoted by $ DC(\gamma, \sigma) $) if there exist constants $ \gamma, \sigma > 0 $ such that
$$
\|k\alpha\|_{\mathbb{T}} \geq \frac{\gamma}{|k|^{\sigma}} \quad \forall k \neq 0.
$$
We denote $ DC = \bigcup_{\gamma > 0, \sigma > 0} DC(\gamma, \sigma) $, a set of full Lebesgue measure. Recently, an unbounded GAA model at Diophantine frequencies exhibiting possible exact Type II MEs was proposed in \cite{Liuxia}. However, unbounded potentials or non-Diophantine frequencies pose significant challenges for physical experiments \cite{zhouwang}.

To further elucidate the mechanism underlying Type II MEs and motivated by the mosaic model \cite{wang}, Zhou et al. \cite{zhouwang} introduced the following mosaic-type model, which also features exact Type II MEs:
\begin{equation}\label{IRS}
[\mathcal{H}_{c,v,\theta}u](n) := c(\theta,n)u(n+1) + c(\theta,n-1)u(n-1) + v(\theta,n)u(n), \quad \theta \in \mathbb{T}_0, \quad \lambda>0,
\end{equation}
where
$
\mathbb{T}_0 = \{\theta : \cos(2\pi(n\alpha + \theta)) \neq 0,\forall n\in\Z\},$
$$
v(\theta,n) =
\begin{cases}
\cos(2\pi((n-1)\alpha + \theta)) & \text{if } n = 1 \mod 2,\\
\cos(2\pi(n\alpha + \theta)) & \text{if } n = 0 \mod 2,
\end{cases}
$$
and
$$
c(\theta,n) =
\begin{cases}
\lambda & \text{if } n = 1 \mod 2,\\
\cos(2\pi(n\alpha + \theta)) & \text{if } n = 0 \mod 2.
\end{cases}
$$
In contrast to the Maryland model \cite{HJY,jiliu2017,JLiu} and the unbounded GAA model \cite{Liuxia}, \cite{zhouwang} offers a novel scheme with experimental feasibility for realizing and detecting the Type II MEs of \eqref{IRS} using  incommensurate Raman superarray of Rydberg atoms.
In this paper, we aim to provide a rigorous proof demonstrating that  \eqref{IRS} possesses exact Type II MEs. Denote $\sigma^{\mathcal{H}}(c,v)$ the spectrum of $ \mathcal{H}_{c,v,\theta}$, which is independent of $\theta$, since the base dynamics is minimal. Then our precise result is as follows:
 
\begin{theorem}\label{theorem1.1.}
	Let $\alpha\in DC$ and $\lambda>0$. For almost every $\theta\in\T$, we have the following:
    \begin{itemize}
        \item $ \mathcal{H}_{c,v,\theta}$ has purely singular continuous spectrum in $\sigma^{\mathcal{H}}(c,v)\cap(-\lambda,\lambda)$.
        \item $ \mathcal{H}_{c,v,\theta}$ exhibits Anderson localization in $\sigma^{\mathcal{H}}(c,v)\backslash[-\lambda,\lambda]$.
    \end{itemize}
\end{theorem}

\subsubsection{Type III ME}
In contrast to the extensive studies on Type I and Type II  MEs, examples of the coexistence of extended and critical states are relatively rare, even in the physical literature. In \cite{yucheng2022}, Wang et al. introduced a model that demonstrates the coexistence of extended, critical, and localized states; however, this finding remains limited to numerical analysis.

For quasiperiodic models, when the frequency is Liouvillean, it becomes relatively easier to achieve  coexistence of absolutely continuous and singular continuous spectra. Zhang showed that Bjerkl\"ov's model exhibits both absolutely continuous and singular continuous spectra for certain parameter sets \cite{zhang}. Moreover, Ge and Jitomirskaya illustrated that a typical Type I  one-frequency operator can manifest all three types of spectra, featuring exact Type II and Type III MEs \cite{gehidden2024}. Nevertheless, to date, no models with exact Type III ME have been identified for Diophantine frequencies. It is noteworthy that for the operators analyzed in \cite{Bj, gehidden2024,zhang}\footnote{One should note Avila's global theory insures for a (measure-theoretically) typical potential, the operator is acritical \cite{A4}.}, if the frequency is Diophantine, there is no singular continuous spectrum at all, as all the proof is based on Gordon's argument \cite{AYZ,Gordon1976}.

In this paper, we propose the following model, which possesses exact Type III MEs:
\begin{equation}\label{ac-sc} [\mathcal{S}_{J,V,\omega}\boldsymbol{u}](n) = J \boldsymbol{u}(n+1) + J \boldsymbol{u}(n-1) + V(\omega + 2n\alpha) \boldsymbol{u}(n) \end{equation}
where $J=\begin{pmatrix}
1&0\\ 0&0\end{pmatrix}$, $V(\omega)=\lambda\begin{pmatrix}\cos2\pi\omega&i\sin2\pi\omega\\-i\sin2\pi\omega&-\cos2\pi\omega\end{pmatrix}$ and $\lambda>0$.
The precise results are as follows:
\begin{theorem}\label{11-28-theorem1.2}
Let $\alpha \in DC$ and  $\lambda>0$. Then we have the following conclusions:
\begin{itemize} 
\item $\mathcal{S}_{J,V,\omega}$ possesses  purely absolutely continuous spectrum in  $\sigma^{\mathcal{S}}(J,V)\backslash[-\lambda, \lambda]$ for every $\omega\in\T$.
\item $\mathcal{S}_{J,V,\omega}$ exhibits purely singular continuous spectrum in  $\sigma^{\mathcal{S}}(J,V) \cap (-\lambda, \lambda)$, if $\omega$ is not rational w.r.t $2\alpha$.
\end{itemize} \end{theorem}

\subsection{Novelty} Now we introduce the novelty of the paper.

\subsubsection{Singular Jacobi operator on the strip.}
The Type I ME model presented in \cite{yucheng2022} and the Type II ME model described in \cite{zhouwang} both belong to the mosaic type. While these discoveries are grounded in Avila's global theory, which is mathematically quite natural, they remain mysterious from a physical perspective. Usually in physical literature, it is a general philosophy that critical states is  self-dual based, and duality breaking induce ME. 
Since the revelations made in \cite{yucheng2022} and \cite{zhouwang}, physicists have been striving to understand the underlying mechanisms from various angles \cite{DRL, Gon2023, HLG,Long20242}. In this paper, we propose a quite general approach to comprehending the mosaic-type models.
If we define
$$
\boldsymbol{u}( n) = (u( 2n), u( 2n+1))^{\mathsf{T}},
$$
both \eqref{mosaic} and \eqref{IRS} can be expressed in the following form:
\begin{equation}\label{sjs}
[\mathcal{S}_{C,V,\omega} \boldsymbol{u}](n) = C \boldsymbol{u}(n + 1) + C^{*} \boldsymbol{u}(n - 1) + V(\omega + 2n\alpha) \boldsymbol{u}(n),
\end{equation}
where $ C $ is a singular matrix, i.e., $ \det C = 0 $. This means that \eqref{sjs} represents a singular Jacobi operator on the strip.
The spectral theory of non-singular Jacobi operators on the strip, characterized by ($\det C \neq 0$), is well-developed (see e.g. \cite{damanik2022,Marx2017}). This theory exhibits properties similar to  Schr\"odinger operator ($C = id$). Many physical models such as Hofstadter model, Chern instulator, Dirac semimetal, Lieb lattices, can be reduced to singular Jacobi operators on the strip \cite{Dwivedi,Solis2024,Bodyfelt2014,Lieb}, which have already garnered significant attention. However, to the best of the authors' knowledge, the comprehensive and basic spectral theory for singular operators  on the strip remains absent.
The following elements are new and significant for our discussion:
\begin{itemize} 
\item Aubry duality (Section \ref{duality})
\item Thouless formula (Section \ref{thou})
\item Relationship between the fibered rotation number and the integrated density of states (IDS) (Section \ref{ids-rota})
\item Jitomirskaya-Last inequality and subordinacy theory (Section \ref{jito-last})

\end{itemize}

Our work is also inspired by recent findings from \cite{jitomirskaya2021on,jitokra}, which demonstrate that the singularity of an operator not only helps to exclude the point spectrum \cite{jitomirskaya2021on}, but also facilitates improved continuity of the spectrum with respect to the Hausdorff topology \cite{jitokra}. In this context, we show that singularity can lead to transitions in the spectrum. The underlying idea is as follows: consider the case where $ C = \begin{pmatrix} 1 & 0 \\ 0 & 0 \end{pmatrix}$, then \eqref{sjs} can be decoupled and reduced to the following operator:
\begin{equation}\label{sjs-1} [Hu]_n = u_{n+1} + u_{n-1} + \left( V_{11}(\theta + n\alpha) - \frac{|V_{12}(\theta + n\alpha)|^2}{E - V_{22}(\theta + n\alpha)} \right) u_n = E u_n. \end{equation}
Consequently, the energy  $E$ is incorporated into the potential, resulting in the emergence of  phase transitions. We emphasize that our analysis can be easily generalized to singular Jacobi operators of the form \eqref{sjs} defined on  $\ell^2(\mathbb{Z}, \mathbb{C}^d)$, where $C$ is a rank-one matrix. We present the case  $d=2$ primarily for its simplicity and physical relevance.

\subsubsection{Aubry duality.} Aubry duality is a powerful tool in the study of the spectral theory of quasiperiodic Schr\"odinger operators with smooth potentials \cite{BJ02,Puig2004}. It plays a crucial role in addressing several open problems related to the almost Mathieu operator \cite{AJten,AJ1,AYZ,AYZ2}, particularly in light of the recent developments in quantitative Avila's  global theory \cite{Ge2023}. However, when the potential is discontinuous, as in the case of mosaic-type models \cite{wang, zhouwang}, it is believed that Aubry duality (essentially Fourier transformation) can't be applied.

Our simple but useful observation is that if we consider equations \eqref{mosaic} and \eqref{IRS} as singular Jacobi operators on the strip, we can indeed apply Aubry duality. Notably, it is straightforward to verify that  \eqref{ac-sc} can be interpreted as the Aubry duality of \eqref{IRS}, which clarifies why \eqref{ac-sc} exhibits Type III mixed spectral properties. Furthermore, we emphasize that this approach is more accessible to physicists, as the existence of singular continuous spectrum (or critical states) is traditionally framed in terms of duality in the literature.
Nonetheless, in rigorous mathematical proofs, one must be cautious regarding whether the singular continuous spectrum is a self-dual phenomenon, since  one can view \eqref{sjs} as representing two coupled chains\footnote{We would like to thank S. Jitomirskaya for insightful discussions on this issue.}.

\subsection{Difficulty, ideas of the proof}

From the spectral perspective, our primary focus is to generalize the spectral theory of the Jacobi operator on $\ell^2(\mathbb{Z})$:
  \begin{equation}\label{jacobi}
[\mathcal{H}_{c,v,\theta}u]_n := c(\theta+n\alpha)u_{n+1} + c(\theta+(n-1)\alpha)u_{n-1} + v(\theta+n\alpha)u_n.
\end{equation}
We aim to extend this to the singular Jacobi operator $\mathcal{S}_{C,V,\theta}$ as defined in \eqref{sjs}.  
In the case $C=\begin{pmatrix}1&0\\ 0&0\end{pmatrix}$, which was reduced to \eqref{sjs-1}. However,  \eqref{sjs-1} cannot be interpreted as either a Sch\"odinger operator with unbounded potential or as a singular Jacobi operator, typically one can only consider it as a nonlinear eigenvalue problem; thus, classical methods fail in this context, which is the root of the difficulties we encounter. In the following, we will begin with classical methods and results related to the  Jacobi operator \eqref{jacobi} to explain how to overcome the challenges presented by \eqref{sjs-1}.

\subsubsection{Singular Continuous spectrum.}
If $c(\theta) \neq 0$, then  \eqref{jacobi} is non-singular, then results mainly parallel the results of Schr\"odinger operator $\mathcal{H}_{1,v,\theta}$.  In this case, in the positive Lyapunov exponent regime, singular continuous spectrum mainly appears if $\alpha$ is Liouvillean \cite{AYZ,Gordon1976} or a generic set of $\theta$ \cite{jiliu2017,JitoSimon1994}. 

If $c(\theta)$ has real roots, trace-class perturbation can decouple the Jacobi matrix into finite-dimensional blocks, resulting $\mathcal{H}_{c,v,\theta}$ absence of absolutely continuous spectrum \cite{jit2013}. This argument was former used by  Simon-Spencer \cite{SS} to exclude absolutely continuous spectrum of Schr\"odinger operator with  unbounded potential. However, this trace-class argument fails if the potential depends on $E$ as in \eqref{sjs-1}. 

Excluding the point spectrum is also a challenge. Based on duality, Avila-Jitomirskaya-Marx introduced a method to establish the absence of point spectrum for extended Harper model \cite{AJM}, which is improved in \cite{hanrui2017,jitomirskaya2021on}. However, for Jacobi operators on the strip, no precedent exists.

Our proof strategy relies on the observation that \eqref{IRS} serves as the dual operator to \eqref{ac-sc}. As a result, the trace-class argument can be used to exclude the point spectrum of \eqref{IRS} 
 \cite{jit2013}. By combining this with Kotani theory, we conclude that the singular continuous spectrum $\sigma^\mathcal{H}_{sc}(c,v)=\sigma^\mathcal{S}(J,V) \cap [-\lambda,\lambda]$  has zero Lebesgue measure. This leads to the absence of absolutely continuous spectrum for $\mathcal{S}_{J,V,\omega}$ within the interval $[- \lambda, \lambda]$.
Furthermore, taking advantage of the singularity of the operator, one can further develop the method outlined in \cite{AJM,hanrui2017} to rule out the existence of point spectrum for $\mathcal{S}_{J,V,\omega}$ in the interval $(- \lambda, \lambda)$.
Thus, it becomes clear that this proof fundamentally depends on the duality of \eqref{ac-sc} is \eqref{IRS}. However, for more general potentials $V$ that not as in \eqref{ac-sc}, the question of whether the corresponding operator $\mathcal{S}_{J,V,\omega}$ exhibits purely singular continuous spectrum remains unresolved.

\subsubsection{Absolutely continuous spectrum.} 
	In the seminar paper of  Dinaburg and Sinai \cite{DS}, they proved that  that if $v$ is analytic, and $\alpha \in DC$, then in the perturbative small regime $\lambda < \lambda_0$, the operator has  $\mathcal{H}_{1,\lambda v, \theta}$ absolutely continuous spectrum.  Here, "perturbative" means that  $\lambda_0$  depends on  $\alpha$  through the Diophantine constants $\gamma$ and  $\sigma$. Under the same assumption, Eliasson \cite{E92} further demonstrated that the spectrum is purely absolutely continuous for any $\theta$.
    
    In the one-frequency case, the breakthrough can be traced back to Avila, who established profound connections between the existence of absolutely continuous spectrum of $\mathcal{H}_{1,\lambda v, \theta}$ and the vanishing of the Lyapunov exponent. Specifically, his \textit{Almost Reducibility Conjecture} (ARC) states that any subcritical cocycle is almost reducible \cite{A4,avila2010almost, avila-kam}, which, in turn, implies the purely absolutely continuous spectrum \cite{A01,arcl}.

Now,  once \eqref{sjs} is  decoupled and reduced to \eqref{sjs-1}, one would expect absolutely continuous spectrum of   $\mathcal{S}_{J,V,\omega}$ in the regime  $E\notin \operatorname{Ran(V_{22})}$ and the corresponding cocycle has zero Lyapunov exponents. Our proof mainly follows from Avila \cite{A01}, and 
relies on the solutions to ARC \cite{avila2010almost}, however, as  the energy  $E$ is incorporated into the potential, there are several additional difficulties need to overcome. 
\begin{itemize}
    \item For non-singular Jacobi operators on $\ell^2(\Z,\C^d)$, one typically defines a maximum spectral measure using $d$ cyclic vectors, followed by the introduction of the $m$-function \cite{kotanisimon1988}. However, in the singular case, finite cyclic vectors may not exist, and the dimension of the Jost solutions may be less than $d$.
    \item These challenges complicate the establishment of the Jitomirskaya-Last inequality and subordinacy theory, both of which are crucial for linking the spectral measure to the growth of the transfer matrix norm.
    \item The refined property of the IDS is derived by leveraging the Thouless formula (which is even absent for singular Jacobi operators on the strip) and the estimation of the Lyapunov exponent in the complex plane. Nevertheless, the classical harmonic analysis argument presented in \cite{DPSB} proves ineffective due to this modified Thouless formula. For detailed information, refer to Section \ref{har-lya}.
    \item The relationship between the fibered rotation number and the IDS relies on the Interlacing Lemma \cite{damanik2022}. The difficulty arises from the fact that its finite restriction $\mathcal{S}|_{[1,N]}$ may have repeated eigenvalues, making the direct application of the Interlacing Lemma non-trivial.
\end{itemize}

\subsubsection{Anderson localization.} 
The previously mentioned coexistence papers \cite{A4, BK2021, zhang} significantly depend on the findings of Bourgain and Goldstein \cite{BG}. They established that in the supercritical regime, the operator  $\mathcal{H}_{1,\lambda v, \theta}$ demonstrates Anderson localization (AL) for almost every Diophantine frequency at any fixed phase. For insights into the multi-frequency and multi-dimensional cases, additional references can be found in \cite{b1, b2, bgs2, jls}. However in physical applications, the frequency $\alpha$ is always fixed as a Diophantine frequency.

In this aspect, Fr\"{o}hlich, Spencer, and Wittwer \cite{fsw}, along with Sinai \cite{Sinai}, independently demonstrated that when the potential is a cosine-like function, the operator  $\mathcal{H}_{1,\lambda v, \theta}$ exhibits AL for almost every phase, provided that the coupling constant is sufficiently large. For more recent developments on the localization of $C^2$-cosine-like potentials, one may refer to \cite{GY,FV}. In the realm of analytic potentials, Eliasson \cite{Eli97} proved that $\mathcal{H}_{1,\lambda v,\theta}$ has  pure point spectrum for almost every $\theta$, given that $\lambda$ is sufficiently large.

While all the above-mentioned results \cite{Eli97, fsw, Sinai} pertain to a fixed Diophantine frequency, they are fundamentally perturbative. At this juncture, it is essential to highlight Jitomirskaya’s groundbreaking paper \cite{jitomirskaya1999metal}, which introduces a non-perturbative localization approach for demonstrating Anderson localization of the almost Mathieu operator. This methodology has since inspired a number of further non-perturbative results \cite{AJ1, HJY, jks, JLiu}. In this paper, we will further develop Jitomirskaya’s framework to a class of singular Jacobi operator, which prove Anderson localization of \eqref{IRS}.

	\section{Preliminaries}
	For a bounded analytic function $f$ defined on a strip $\{|\Im z|<h\}$, let $\mathop{\|f\|}_{h}=\sup_{|\Im\theta<h|}\|f(\theta)\|$ and denote by $C_{h}^{\omega}(\mathbb{T},\bullet)$ the set of all these $\bullet$-valued functions ($\bullet$ will usually denote $\mathbb{R}$, $SL(2,\mathbb{R})$, $M(2,\mathbb{C})$). When $\theta\in\mathbb{R}$, we also set $\|\theta\|_{\mathbb{T}}=\inf_{j\in\mathbb{Z}}|\theta-j|$.

	\subsection{Continued Fraction Expansion.}\quad Let $\alpha\in(0,1)$ be irrational, $x_{0}=0$ and $y_{0}=\alpha$. Inductively, for $k\ge1$, we define
	\begin{equation*}\quad
		x_{k}=\lfloor y_{k-1}^{-1}\rfloor,\ y_{k}=y_{k-1}^{-1}-x_{k},
	\end{equation*}
	Let $p_{0}=0$, $p_{1}=1$, $q_{0}=1$, $q_{1}=x_{1}$. We define inductively $p_{k}=x_{k}p_{k-1}+p_{k-1}$, $q_{k}=x_{k}q_{k-1}+q_{k-2}$. Then $(q_{n})_{n}$ is the sequence of denominators of the best rational approximations of $\alpha$, since we have $\|k\alpha\|_{\mathbb{T}}\ge\|q_{n-1}\alpha\|_{\mathbb{T}}$, $\forall\ 1\le k<q_{n}$, and 
	\begin{equation*}\quad
		\frac{1}{2q_{n+1}}\le\|q_{n}\alpha\|_{\mathbb{T}}\le\frac{1}{q_{n+1}}.
	\end{equation*}

	\begin{lemma}\cite{AJten}\label{ten}
		Let $\alpha\in\mathbb{R}\backslash\mathbb{Q}$, $x\in\mathbb{R}$ and $0\le l_0\le q_n-1$ be such that $$|\sin\pi(x+l_0\alpha)|=\inf_{0\le l\le q_n-1}|\sin\pi(x+l\alpha)|,$$ then for some absolute constant $C>0$,
		$$-C\ln q_n\le\sum_{0\le l\le q_n-1,l\neq l_0}\ln|\sin\pi(x+l\alpha)|+(q_n-1)\ln2\le C\ln q_n.$$
	\end{lemma}
	
	\subsection{Cocycle, Lyapunov exponent} \label{co-ly}
	Let $X$ be a compact metric space, $(X, \nu, T)$ be ergodic. Given $A:X \to M(2,\mathbb{C})$ measurable with $\ln_+\|A(\cdot)\|\in L^1(X)$, a cocycle $(T, A)$ is a linear skew product:
	\begin{eqnarray*}\label{cocycle}
		(T,A):&X \times \C^2 \to X \times \C^2\\
		\nonumber &(x,\phi) \mapsto ( T x,A(x) \cdot \phi).
	\end{eqnarray*}
The  cocycle $(T,A)$  is called singular if $\det A(x)=0$ for some $x\in X $, and non-singular otherwise.

	For $n\in\mathbb{Z}$, $A_n$ is defined by $(T,A)^n=(T^n,A_n)$. Thus $A_{0}(x)=id$,
	\begin{equation*}
		A_{n}(x)=\prod_{j=n-1}^{0}A(T^{j}x)=A(T^{n-1}x)\cdots A(Tx)A(x),\ for\ n\ge1,
	\end{equation*}
	and $A_{-n}(x)=A_{n}(T^{-n}x)^{-1}$. The Lyapunov exponent is defined as
	\begin{equation*}\quad
		L(T,A):=\lim_{n\rightarrow\infty}\frac{1}{n}\int_{X}\ln\|A_{n}(x)\|dx.
	\end{equation*}
	In this paper, we will consider the following two useful  cocycles.
	\begin{itemize}
		\item  $X=\mathbb{T}$ and $T=R_\alpha$, where  $R_{\alpha}\theta= \theta+\alpha$, then $(\alpha,A):=(R_{\alpha},A)$ is a  quasiperiodic cocycle.
		\item  $X=\mathbb{T}\times\mathbb{Z}_{2}$ and $T=T_{\alpha}$, where $T_{\alpha} (\theta,n)= (\theta+\alpha,n+1)$, then $(T_{\alpha},A)$ defines an  almost-periodic cocycle.
	\end{itemize}
	These dynamical  system $(X,T)$ is uniquely ergodic if  $\alpha$ is irrational \cite{mane2012ergodic}.

	We say an $SL(2,\mathbb{R})$ cocycle $(T,A)$  is uniformly hyperbolic if, for every $x \in X$, there exists a continuous splitting $\mathbb{R}^2=E_{s}(x)\oplus E_{u}(x)$ such that for every $n\ge0$,
	\begin{equation*}\quad
		\begin{split}
			|A_{n}(x)v(x)|&\le Ce^{-cn}|v(x)|,\ v(x)\in E_{s}(x),\\
			|A_{-n}(x)v(x)|&\le Ce^{-cn}|v(x)|,\ v(x)\in E_{u}(x),
		\end{split}
	\end{equation*}
	for some constans $C,c>0$. Clearly, it holds that $A(x)E_{s}(x)=E_{s}(Tx)$ and $A(x)E_{u}(x)=E_{u}(Tx)$ for every $x\in X$, and if $(T,A)$ is uniformly hyperbolic, then $L(T,A)>0$.
	
	\subsection{Dynamical  defined Jacobi operators.} 
	Let $X$ be a compact metric space, $(X,\nu,T)$ be ergodic, and $v,c: X \rightarrow \R$ is continuous. Then one can define the Jacobi operator on 
	$\ell^2(\Z)$:
	\begin{equation*}
		[\mathcal{H}_{c,v,x}u]_{n}=c(T^nx)u_{n+1}+c(T^{n-1}x)u_{n-1}+v(T^n x )u_{n},\ \ \forall x \in X. 
	\end{equation*}
	It is well known that the spectrum of $\mathcal{H}_{c,v,x}$ is a compact subset of $\mathbb{R}$, independent of $x$ if $(X,T)$ is minimal \cite{damanik2017schrodinger}, we shall denote it by $\sigma^\mathcal{H}(c,v)$.
	The integrated density of states (IDS) $N_{v,c}:\mathbb{R}\rightarrow[0,1]$ of $\mathcal{H}_{c,v,x}$ is defined as
	\begin{equation*}
		N_{v,c}(E)=\int_{X}\tilde{\mu}_{v,c,x}(-\infty,E]d\nu,
	\end{equation*}
	where $\tilde{\mu}_{v,c,x}$ is the canonical spectral measure of $\mathcal{H}_{c,v,x}$. Note  any formal solution $u=\{u_{n}\}_{n\in\mathbb{Z}}$ of $\mathcal{H}_{c,v,x}u=Eu$ can be rewritten as 
	\begin{equation*}
		\begin{pmatrix}u_{n+1}\\u_{n}\end{pmatrix}=S_{E}^{v,c}( T^n x)\begin{pmatrix}u_{n}\\u_{n-1}\end{pmatrix},
	\end{equation*}
	where \begin{equation*}
		S_{E}^{v,c}(\theta)=\frac{1}{c(\theta)}\begin{pmatrix}E-v(\theta)&-c(T^{-1}\theta)\\c(\theta)&0\end{pmatrix}=:\frac{1}{c(\theta)}\widetilde S_{E}^{v,c}(\theta),
	\end{equation*}
	and we  call $(T, S_{E}^{v,c})$  the Jacobi cocycle.  
	The Thouless formula \cite{JAC} links $N_{v,c}(E)$ and $L(T,S_E^{v,c})$:
	\begin{equation*}
	L(T,S_E^{v,c})=\int_{\mathbb{R}}\ln|E-E'|dN_{v,c}(E')-\mathbb{E}(\ln|c(x)|).
	\end{equation*}
	
	\subsection{Global theory of one frequency quasiperiodic cocycle.}\label{acceleration}
	Let us make a short review of  Avila's global theory of one-frequency quasiperiodic cocycles \cite{A4}.  Suppose that $D\in C^\omega(\T, M(2,\C))$ admits a holomorphic extension to $\{|\Im  \theta |<h\}$. Then for
	$|\epsilon|<h$, we define $D_\epsilon \in
	C^\omega(\T, M(2,\C))$ by $D_\epsilon(\cdot)=A(\cdot+i \epsilon)$, and define the acceleration of $(\alpha,D_{\varepsilon})$ as follows
	\begin{equation*}
		\omega(\alpha,D_{\varepsilon})=\frac{1}{2\pi}\lim_{h\to 0+}\frac{L(\alpha,D_{\varepsilon+h})-L(\alpha,D_{\varepsilon})}{h}.
	\end{equation*}
	
	The acceleration was first introduced by Avila  for analytic $SL(2,\C)$-cocycles \cite{A4}, and extended to analytic $M(2,\C)$ cocycles by Jitomirskaya-Marx \cite{jitomirskaya2012analytic}. It follows from the convexity and continuity of the Lyapunov exponent that the acceleration is an upper semicontinuous function in parameter $\varepsilon$. The key property of the acceleration is that it is quantized: 
	
	\begin{theorem}[Quantization of acceleration\cite{A4,jitomirskaya2012analytic,jit2013}]\label{theorem3.4}
		Suppose that  $(\alpha,D)\in(\mathbb{R}\backslash\mathbb{Q})\times C^{\omega}(\mathbb{T},M_{2}(\mathbb{C}))$ with $\det D(\theta)$ bound away from 0 on the strip $\{|\Im  \theta |<h\}$, then $\omega(\alpha,D_{\varepsilon})\in \frac{1}{2}\mathbb{Z}$ in the strip. Moreover,  if $D\in C^{\omega}(\mathbb{T}, SL(2,\C))$, then 
		$\omega(\alpha,D_{\varepsilon})\in\mathbb{Z}$. 
	\end{theorem}

Analytic one-frequency $SL(2,\R)$ cocycles that are not uniformly hyperbolic can be classified into several categories: 
it is subcritical if $L(\alpha,A)=0$ and $\omega(\alpha,A)=0$, it is supercritical if $L(\alpha,A)>0$, and critical if $L(\alpha,A)=0$ and $\omega(\alpha,A)>0$ \cite{A4}. 

    \begin{theorem}\cite{avila2010almost,avila-kam}\label{11-26-themrem2.4}
    Let $\alpha\in\R\backslash\Q$ and $A\in C^\omega(\T,SL(2,\R))$ be such that $(\alpha,A)$ is subcritical. Then $(\alpha,A)$ is almost reducible: There exist $\delta>0$, a constant $A_*\in SL(2,\R)$ and a sequence of analytic maps $B_n\in C^\omega(\T,PSL(2,\R))$ such that  $\sup_{|\Im x|<\delta}\|B_n(x+\alpha)A(x)B_n(x)^{-1}-A_*\|\to0$ as $n\to\infty$.
    \end{theorem}

\section{New exact mobility edges of critical and localized states}

In this section, we aim to prove $\mathcal{H}_{c,v,\theta}$ defined by \eqref{IRS} has exact  Type II ME. 

\subsection{Explicit formulas of LE of $\mathcal{H}$}
To do this, the first step is to determine its Lyapunov exponent.  
Note in \eqref{IRS}, $v$ and $c$ are defined on $\mathbb{T}\times\mathbb{Z}_{2}$, consequently \eqref{IRS} induces an almost-periodic Jacobi cocycle $(T_\alpha, S_E^{v,c})$ where $T_\alpha(\theta,n)=(\theta+\alpha,n+1)$. Although $(T_\alpha, S_E^{v,c})$ is not a quasiperiodic cocycle in the strict sense, its iterate
	\begin{equation*}
		\begin{split}
	(2\alpha,\tilde A^E)&:=(2\alpha,  S_E^{v,c}(\theta,1) S_E^{v,c}(\theta,0))
	\end{split}
\end{equation*}
	indeed defines a quasiperiodic cocycle. Furthermore, through straightforward calculation, we find
		\begin{equation}\label{be}
		\begin{split}
		\tilde A^E(\theta)=\frac{1}{\lambda\cos2\pi\theta}\begin{pmatrix}
		E^2-2E\cos2\pi\theta&-\lambda E+\lambda\cos2\pi\theta\\\lambda E-\lambda\cos2\pi\theta&-\lambda^2
	\end{pmatrix}=:\frac{1}{\lambda\cos2\pi\theta}\tilde B^E(\theta).
\end{split}
	\end{equation}

\begin{lemma}\label{lemma3.1.}
For any $\alpha\in \R\backslash \Q$,	$L(2\alpha,\tilde A^E_{\epsilon})$ is a continuous function with respect to the parameter $\epsilon$.
\end{lemma}

\begin{proof}
Obviously, $(2\alpha,\tilde B^E)$ is an analytic cocycle, which implies $L(2\alpha,\tilde B^E_\epsilon)$ is continuous in $\epsilon$ \cite{A4}.
Denote $I_\epsilon=\int_{\mathbb{T}}\ln|\lambda\cos2\pi(\theta+i\epsilon)|d\theta$, and we have
$$L(2\alpha,\tilde A^E_\epsilon)=L(2\alpha,\tilde B^E_\epsilon)-I_\epsilon.$$

Applying Poisson-Jensen formula yields$
	I_\epsilon=2\pi|\epsilon|+\ln|\lambda/2|\text{ for all }\epsilon,$
which implies the continuity of $L(2\alpha,\tilde A^E_\epsilon)$ in $\epsilon$.
\end{proof}

As $\tilde B^E_\epsilon$ ($\epsilon\neq0$) is non-singular, we can calculate the Lyapunov exponent $L(2\alpha,\tilde A^E_\epsilon)$ for $|\epsilon|>0$, and obtain the Lyapunov exponent for $\epsilon=0$ by Lemma \ref{lemma3.1.}.

\begin{lemma}\label{theorem3.2.}
Suppose $\lambda\neq0$ and $\alpha\in\mathbb{R\backslash Q}$. For $E\in\sigma^{\mathcal{H}}(c,v)$, the Lyapunov exponent can be expressed as
\begin{equation*}
L(T_\alpha, S_E^{v,c})=\frac{1}{2}\ln\left|\frac{|E|+\sqrt{E^2-\lambda^2}}{\lambda}\right|.
\end{equation*}
\end{lemma}

\begin{proof} Let us begin by considering the behavior of $\tilde A^E(\theta+i\epsilon)$ uniformly in $\theta\in\mathbb{T}$. As $\epsilon\to+\infty$, we have \begin{equation*}
\tilde A^E(\theta+i\epsilon)=\tilde A^E_{\infty}+o(1):= \frac{1}{\lambda}\begin{pmatrix}
-2E&\lambda\\-\lambda&0
\end{pmatrix} +o(1).
\end{equation*} 

The continuity of the Lyapunov exponent   implies that \begin{equation*}
L(2\alpha,\tilde A^E_\epsilon)=L(2\alpha,\tilde A^E_{\infty})+o(1),\quad  \text{as}  \quad  \epsilon\to+\infty
\end{equation*} By applying the quantization of acceleration  \cite{A4}, we deduce that 
$L(2\alpha,\tilde A^E_\epsilon)=L(2\alpha,\tilde A^E_{\infty})
$ for all $\epsilon>0$ sufficiently large. The convexity, continuity, and symmetry of $L(2\alpha,\tilde A^E_\epsilon)$ with respect to $\epsilon$ further imply that 
$L(2\alpha,\tilde A^E_\epsilon)=L(2\alpha,\tilde A^E_{\infty})
$ for all $\epsilon$. Consequently, we find that \begin{equation*}
L(2\alpha,\tilde B^E_\epsilon)=L(2\alpha,\tilde A^E_{\infty})+2\pi|\epsilon|+\ln|\lambda/2|\text{ for all }\epsilon.
\end{equation*} Note that the constant matrix $\tilde A^E_{\infty}$ possesses two eigenvalues $\frac{-E\pm\sqrt{E^2-\lambda^2}}{\lambda}$, the result follows. \end{proof}

Define \begin{equation*}
\sigma_{\text{sup}}^{\mathcal{H}}(c,v)=\sigma^{\mathcal{H}}(c,v)\backslash[-\lambda,\lambda],
\end{equation*}
and
\begin{equation*}
\sigma_{\text{cri}}^\mathcal{H}(c,v)=\sigma^{\mathcal{H}}(c,v)\cap(-\lambda,\lambda).
\end{equation*}
By Lemma \ref{theorem3.2.}, the Lyapunov exponent is positive on $\sigma_{\text{sup}}^\mathcal{H}(c,v)$ and zero on $\sigma_{\text{cri}}^\mathcal{H}(c,v)$. Then, we have the following elementary observations:
\begin{lemma}
For all irrational $\alpha$, $\sigma_{\text{sup}}^{\mathcal{H}}(c,v)\neq\emptyset$.
\end{lemma}

\begin{proof}
Otherwise, for any $E\in\sigma^{\mathcal{H}}(c,v)$, we have $|E|\le\lambda$.
Note that
\begin{equation*}
\begin{split}
\mathcal{H}_{c,v,\theta} \delta_n=&v(\theta, n)\delta_n+c(\theta, n-1)\delta_{n-1}+c(\theta, n)\delta_{n+1},\\ \mathcal{H}_{c,v,\theta}^2\delta_n=&v(\theta, n)(v(\theta, n)\delta_n+c(\theta, n-1)\delta_{n-1}+c(\theta,n)\delta_{n+1})\\&+c(\theta,n-1)(v(\theta, n-1)\delta_{n-1}+c(\theta, n-2)\delta_{n-2}+c(\theta, n-1)\delta_n)\\&+c(\theta, n)(v(\theta, n+1)\delta_{n+1}+c(\theta, n)\delta_n+c(\theta, n+1)\delta_{n+2}).
\end{split}
\end{equation*}
For any $\theta,$ select $n$ such that $v(\theta, n)\neq0$. By the spectral theorem, one has
\begin{equation*}
		\lambda^2<v^2(\theta,n)+c^2(\theta,n-1)+c^2(\theta,n)=\langle\delta_n,\mathcal{H}_{c,v,\theta}^2\delta_n\rangle=\int E^2d\tilde{\mu}_{\theta,\delta_n}\le\lambda^2,
\end{equation*}
which leads to a contradiction.
\end{proof}

\begin{lemma}\label{8.19-lemma5.5}
For all irrational $\alpha$,  $\theta\in\mathbb{T}_0$, $\pm\lambda$ are not eigenvalues of $\mathcal{H}_{c,v,\theta}$.
\end{lemma}

\begin{proof} Assume that $u\in \ell^2(\mathbb{Z})$ is a non-trivial solution of $\mathcal{H}_{c,v,\theta}u=Eu$. We have the following equations \begin{equation*}
\cos2\pi(\theta+2n\alpha)\ u_{2n}+\cos2\pi(\theta+2n\alpha)\ u_{2n+1}+\lambda u_{2n-1}=Eu_{2n},
\end{equation*} and \begin{equation*}
\cos2\pi(\theta+2n\alpha)\ u_{2n+1}+\lambda u_{2n+2}+\cos2\pi(\theta+2n\alpha)\ u_{2n}=Eu_{2n+1}.
\end{equation*}

	Let $t:=E/\lambda$, then $t=\pm1$. We can rewrite the above equations as
	\begin{equation}\label{8.19-eq27}
		\cos2\pi(\theta+ 2n\alpha)(u_{2n}+u_{2n+1})+\lambda u_{2n-1}=E u_{2n},
	\end{equation}
	and
	\begin{equation}\label{8.19-eq28}t(u_{2n+1}-u_{2n})=u_{2n+2}-u_{2n-1}.\end{equation}
	
	 If $t=-1$, summing equation \eqref{8.19-eq28} over $n$ from $m$ to $\infty$, we obtain
	\begin{equation}\label{eq7}
	u_{2m}+u_{2m-1}=0,\ \forall m\in\mathbb{Z}
	\end{equation}
	since $u\in \ell^2(\mathbb{Z})$. Substituting the above result into \eqref{8.19-eq27}, we have \begin{equation*}
\cos2\pi(\theta+2n\alpha)(u_{2n}+u_{2n+1})=-\lambda(u_{2n}+u_{2n-1})=0.
\end{equation*} Since $\theta\in\mathbb{T}_0$, it follows that \begin{equation}\label{eq8}
u_{2n}+u_{2n+1}=0,\ \forall n\in\mathbb{Z}.
\end{equation}
 Combining equations \eqref{eq7} and \eqref{eq8}, we observe that $|u_n|$ is constant for all $n\in\mathbb{Z}$, which leads to a contradiction.

If $t=1$, summing equation \eqref{8.19-eq28} over $n$ from $m$ to $\infty$, we have \begin{equation*}
u_{2m}-u_{2m-1}=-2\sum_{k=m+1}^\infty(u_{2k}-u_{2k-1})=:-2S_{m+1},
\end{equation*} Taking the limit as $m$ tends to $\pm\infty$, we find that $S_{\pm\infty}=0$ since $u\in \ell^2(\mathbb{Z})$. Note \begin{equation*}
u_{2m}-u_{2m-1}=S_m-S_{m+1},
\end{equation*} it follows that $|S_{m+1}|=|S_m|=0$ for all $m\in\mathbb{Z}$. Thus, we conclude $
u_{2m}=u_{2m-1},\ \forall m\in\mathbb{Z}.
$ Substituting the above equation into \eqref{8.19-eq27}, we obtain \begin{equation*}
\cos2\pi(\theta+2n\alpha)(u_{2n}+u_{2n+1})=0.
\end{equation*} Therefore, we have $
u_{2n}+u_{2n+1}=0,\ \forall n\in\mathbb{Z}.$
One has $|u_n|$ is constant for all $n\in\mathbb{Z}$, leading to a contradiction. \end{proof}

Once we have this, we obtain the following consequence:

\begin{corollary}
For all irrational $\alpha$,  $\pm\lambda\in\sigma^{\mathcal{H}}(c,v)$. Moreover,  $\sigma_{\text{cri}}^\mathcal{H}(c,v)\neq\emptyset$.
\end{corollary}

\begin{proof}
We just prove $\lambda\in\sigma^{\mathcal{H}}(c,v)$, it can be easily verified that the equation $\mathcal{H}_{c,v,\theta}u=\lambda u$ has a generalized eigenfunction $u$ with $u_{2n}=(-1)^n,\ u_{2n-1}=(-1)^{n}$.

Now assume $\sigma^\mathcal{H}_{\text{cri}}(c,v)= \emptyset$.  Let $\theta\in\T_0$ and $w=\frac{1}{\sqrt{2}}(\delta_{n+1}-\delta_{n})$ with $2|n$, then
\begin{align*}
\mathcal{H}_{c,v,\theta}w=\frac{\lambda}{\sqrt{2}}(\delta_{n+2}-\delta_{n-1})
\end{align*}
and
\begin{align*}
\mathcal{H}_{c,v,\theta}^2w=& \frac{\lambda}{\sqrt{2}}(c(\theta,n+2)\delta_{n+3}+\lambda\delta_{n+1}+v(\theta,n+2)\delta_{n+2}) \\
& + \frac{\lambda}{\sqrt{2}}( -\lambda\delta_n-c(\theta,n-2)\delta_{n-2}-v(\theta,n-1)\delta_{n-1}),
\end{align*}
which implies $\langle w,\mathcal{H}_{c,v,\theta}^2w\rangle=\lambda^2$. Therefore, we have
\begin{align*}
	\lambda^2=\langle w,\mathcal{H}_{c,v,\theta}^2w\rangle=\int E^2d\tilde{\mu}_{\theta, w}\ge\lambda^2.
	\end{align*}
Then, it follows that
\begin{align*}
E^2-\lambda^2=0\text{ for }\tilde{\mu}_{\theta, w}\text{-a.e. }E\in\sigma^{\mathcal{H}}(c,v),
\end{align*}
which implies
$\tilde{\mu}_{\theta, w}=t\delta_{\lambda}+(1-t)\delta_{-\lambda}$ for some $t\in[0,1]$. Thus,
\begin{align*}
\|(\mathcal{H}_{c,v,\theta}^2-\lambda^2)w\|^2=\int|E^2-\lambda^2|^2d\tilde{\mu}_{\theta,w}(E)=0,
\end{align*}
and so $(\mathcal{H}_{c,v,\theta}+\lambda)(\mathcal{H}_{c,v,\theta}-\lambda)w=0$, i.e., $\lambda$ or $-\lambda$ is an eigenvalue of $\mathcal{H}_{c,v,\theta}$, leading to a contradiction with Lemma \ref{8.19-lemma5.5}.
\end{proof}

	\subsection{Singular continuous spectrum of $\mathcal{H}$}
	Having established the coexistence of the supercritical regime ($\sigma^\mathcal{H}_{\text{sup}}(c,v)$) and critical regime $(\sigma^\mathcal{H}_{\text{cri}}(c,v))$, we now focus on demonstrating that the critical regime exhibits purely singular continuous spectrum for $a.e.\ \theta\in\T$. To achieve this, we first rule out the possibility of an absolutely continuous spectrum, as detailed in the following proposition:
	\begin{proposition}[Proposition 7.1 in \cite{jitomirskaya2012analytic}]\label{proposition3.1}Let $\alpha\in \R\backslash \Q$, consider the Jacobi operator 
    \begin{equation*}
		[\mathcal{H}_{c,v,\theta}u]_n=c(\theta+n\alpha)u_{n+1}+\bar{c}(\theta+ (n-1)\alpha)u_{n-1}+v(\theta+n\alpha)u_n.\ 
	\end{equation*}
		If $c(\theta)$ has a real root, then the absolutely continuous spectrum is empty.
	\end{proposition}

    Then, we have the following direct consequence: 
	\begin{corollary}\label{corollary6.2}
	We have $\left|\sigma^{\mathcal{H}}(c,v)\cap[-\lambda,\lambda]\right|=0$.
	\end{corollary}

	\begin{proof}
    By Lemma \ref{theorem3.2.}, for any $E\in \sigma^{\mathcal{H}}(c,v)$ with $|E|\leq \lambda$, we have $L(T_\alpha, S_E^{v,c})=0$.
	By Proposition \ref{proposition3.1}, there is no absolutely continuous component in $[-\lambda,\lambda]$. Therefore, applying Kotani theory for the singular Jacobi operator (the proof of Theorem 5.17 of \cite{JAC} shows that it holds for a.e. $\theta$ under $\ln(|c(\cdot)|)\in L^1$), we obtain $\left|\sigma( \mathcal{H}_{c,v,\theta})\cap[-\lambda,\lambda]\right|=0$. Note that $\sigma( \mathcal{H}_{c,v,\theta})$ is independent of $\theta$ and coincides with $\sigma^{\mathcal{H}}(c,v)$.
	\end{proof}

Next, we  demonstrate the absence of point spectrum in the critical regime  $\sigma^\mathcal{H}_{\text{cri}}(c,v)$. To this end, we introduce the following definition:
	\begin{definition}
	A point $\theta\in\T$ is said to be $\alpha$-rational if $(\mathbb{Z}\alpha+2\theta)\cap\mathbb{Z}\neq\emptyset$; otherwise, it is called non-$\alpha$-rational.
	\end{definition}
    
Under this definition, we establish the following result:	
	\begin{theorem}\label{theorem2.2}
	For any irrational $\alpha$ and any non-$2\alpha$-rational $\theta\in\T_0$, the operator $ \mathcal{H}_{c,v,\theta}$
 exhibits purely singular continuous spectrum in $\sigma^\mathcal{H}_{\text{cri}}(c,v)$.
	\end{theorem}
	
\begin{proof}
	By Proposition \ref{proposition3.1}, it suffices to show that there is no point spectrum component for all non-$2\alpha$-rational $\theta\in\T_0$. Assume that $u\in \ell^2(\mathbb{Z})$ is a non-trivial solution of $ \mathcal{H}_{c,v,\theta}u=Eu$ for some non-$2\alpha$-rational $\theta\in\T_0$. Then, the following equations hold:
	\begin{equation*}
	\cos2\pi(\theta+2n\alpha)\ u_{2n}+\cos2\pi(\theta+2n\alpha)\ u_{2n+1}+\lambda u_{2n-1}=Eu_{2n},\end{equation*}
	and
	\begin{equation*}
	\cos2\pi(\theta+2n\alpha)\ u_{2n+1}+\lambda u_{2n+2}+\cos2\pi(\theta+2n\alpha)\ u_{2n}=Eu_{2n+1}.\end{equation*}
	As in Lemma \ref{8.19-lemma5.5}, let $t:=E/\lambda$, where $|t|<1$ for $E\in \sigma^\mathcal{H}_{\text{cri}}(c,v)$. The equations can then be rewritten as:
	\begin{equation}\label{equation9}
		\cos2\pi(\theta+ 2n\alpha)(u_{2n}+u_{2n+1})+\lambda u_{2n-1}=\lambda t u_{2n},
	\end{equation}
	and
	\begin{equation}\label{equation10}t(u_{2n+1}-u_{2n})=u_{2n+2}-u_{2n-1}.\end{equation}
	Define $\varphi(x)=\sum_{n\in\mathbb{Z}}u_{2n}e^{2\pi i nx}$ and $\psi(x)=\sum_{n\in\mathbb{Z}}u_{2n+1}e^{2\pi i nx}$, where $\varphi,\psi\in L^2(\mathbb{T})$. Multiplying equations \eqref{equation9} and \eqref{equation10} by $e^{2\pi inx}$ and summing over $n$, we obtain
    \small
	\begin{equation*}
	\begin{split}
	e^{2\pi i\theta}(\varphi(x+2\alpha)+\psi(x+2\alpha))+e^{-2\pi i\theta}(\varphi(x-2\alpha)+\psi(x-2\alpha))+2\lambda e^{2\pi i x}\psi(x)=2E\varphi(x),
	\end{split}
	\end{equation*}
    \normalsize
	and
	\begin{equation*}
	t(\psi(x)-\varphi(x))=e^{-2\pi ix}\varphi(x)-e^{2\pi ix}\psi(x).
	\end{equation*}
	By solving these equations, we have
	$$\psi(x)=\frac{t+e^{-2\pi ix}}{t+e^{2\pi ix}}\varphi(x),$$
	and
    \small
	\begin{equation}\label{eq11}
	\begin{split}
	e^{2\pi i\theta}\varphi(x+2\alpha)\frac{t+\cos2\pi(x+2\alpha)}{t+e^{2\pi i(x+2\alpha)}}+e^{-2\pi i\theta}\varphi(x-2\alpha)\frac{t+\cos2\pi(x-2\alpha)}{t+e^{2\pi i(x-2\alpha)}}=\lambda\varphi(x)\frac{t^2-1}{t+e^{2\pi ix}}.
	\end{split}
	\end{equation}
    \normalsize
To derive the next step, we	multiply \eqref{equation9} and \eqref{equation10} by $e^{-2\pi inx}$ and sum over $n$, yielding
	\footnotesize
	\begin{equation}\label{eq12}
	\begin{split}
	e^{2\pi i\theta}\varphi(-x+2\alpha)\frac{t+\cos2\pi(-x+2\alpha)}{t+e^{2\pi i(-x+2\alpha)}}+e^{-2\pi i\theta}\varphi(-x-2\alpha)\frac{t+\cos2\pi(-x-2\alpha)}{t+e^{2\pi i(-x-2\alpha)}}=\lambda\varphi(-x)\frac{t^2-1}{t+e^{-2\pi ix}}.
	\end{split}
	\end{equation}
	\normalsize
	 
	Define $g(x):=\frac{\varphi(x)}{t+e^{2\pi ix}}\not\equiv0$ (otherwise, $\varphi(x)=\psi(x)\equiv0$, leading to $u=0$), where $g(x)\in L^2(\mathbb{T})$ since $|t+e^{2\pi ix}|\ge1-t$. We combine \eqref{eq11} and \eqref{eq12} to derive the following matrix equation
	\begin{equation*}
	\begin{split}
	&\begin{pmatrix}
	\frac{\lambda(t^2-1)}{t+\cos2\pi(x+2\alpha)}&-\frac{t+\cos2\pi(x-2\alpha)}{t+\cos2\pi(x+2\alpha)}\\1&0
	\end{pmatrix}\begin{pmatrix}g(x)&g(-x)\\e^{-2\pi i\theta}g(x-2\alpha)&e^{2\pi i\theta}g(-x+2\alpha)\end{pmatrix}\\&=\begin{pmatrix}g(x+2\alpha)&g(-x-2\alpha)\\e^{-2\pi i\theta}g(x)&e^{2\pi i\theta}g(-x)\end{pmatrix}\begin{pmatrix}e^{2\pi i\theta}&0\\0&e^{-2\pi i\theta}\end{pmatrix}.
	\end{split}
	\end{equation*}
	To simplify the expressions further, we define
	\begin{equation*}
	\begin{split}
	&D(x):=\begin{pmatrix}
	\frac{\lambda(t^2-1)}{t+\cos2\pi(x+2\alpha)}&-\frac{t+\cos2\pi(x-2\alpha)}{t+\cos2\pi(x+2\alpha)}\\1&0
	\end{pmatrix},\\
	&M_\theta(x):=\begin{pmatrix}g(x)&g(-x)\\e^{-2\pi i\theta}g(x-2\alpha)&e^{2\pi i\theta}g(-x+2\alpha)\end{pmatrix},\\
	&d(x):=t+\cos2\pi x,\quad R_\theta=\begin{pmatrix}e^{2\pi i\theta}&0\\0&e^{-2\pi i\theta}\end{pmatrix}.
	\end{split}
	\end{equation*}
	Thus, $M_\theta(x)\in L^2(\mathbb{T},M(2,\mathbb{C})),$ and the equation can be rewritten as
\begin{equation}\label{equation13}
D(x)M_\theta(x)=M_\theta(x+2\alpha)R_\theta.
\end{equation}
Consider the set $\mathbb{T}_t:=\left\{x\in\mathbb{T}:t+\cos2\pi(n\alpha+x)\neq0,\forall n\in\Z\right\}$. Clearly, we have $|\mathbb{T}_t|=1$.
		For any $x\in\mathbb{T}_t$, it follows from \eqref{equation13} that
		$$|\det M_\theta(x)|\cdot |d(x-2\alpha)|=|\det M_\theta(x+2\alpha)|\cdot| d(x+2\alpha)|.$$
		Using the ergodicity of irrational rotations, we deduce that
		\begin{equation*}
		|\det M_{\theta}(x)|\cdot|d(x-2\alpha)|\cdot|d(x)|\overset{\text{a.e. x}}= b,\text{ for some }b\ge0.
		\end{equation*} 
		Since $d(x)\neq0$ on $\mathbb{T}_t$, it follows that $b>0$ if and only if $\det M_\theta(x)\neq0$ for almost every $x$. Moreover, note that the set $\{x\in\mathbb{T}_t:\det M_\theta(x)=0\}$ is invariant under rotations. Consequently, its measure must be either zero or one.
		
	Assume $\det M_\theta(x)=0$ for almost every $x$. Then, there exists a function $\phi(x)$ such that, for almost every $x$,
		\begin{equation}\label{meq16}
		\begin{pmatrix}
			g(x)\\e^{-2\pi i\theta}g(x-2\alpha)
		\end{pmatrix}=\phi(x)\begin{pmatrix}
		g(-x)\\e^{2\pi i\theta}g(-x+2\alpha)
	\end{pmatrix}.
\end{equation}
		In particular, $\phi(x)=\frac{g(x)}{g(-x)}\in\overline{\C}$ is a non-identically vanishing, measurable function on $\mathbb{T}$. Equation \eqref{meq16} further implies
		$$\phi(x+2\alpha)=e^{-4\pi i\theta}\phi(x),\ a.e.\ x.$$
		By ergodicity, $\phi(x)=b'$ for some $b'\neq0$, and $\phi(x)\in L^1(\mathbb{T})$. Writing $\phi(x)=\sum_{n\in\mathbb{Z}}\hat{\phi}_ne^{2\pi inx}$, it follows that
		$\hat{\phi}_n(e^{4\pi in\alpha+4\pi i\theta}-1)=0,\ \forall n\in\mathbb{Z}.$ By assumption, we have excluded all $\theta$ that are $2\alpha$-rational. Consequently, $\phi\equiv0$, which leads to a contradiction. 
Thus, for all non-$2\alpha$-rational $\theta\in\T_0$ and almost every $x\in\mathbb{T}$, it holds that
\begin{equation}\label{equation15}
|\det M_\theta(x)|=\frac{b}{|d(x-2\alpha)d(x)|},
\end{equation}
for some constant $b>0$.

	Since $|t|<1$, there exists $x_0\in\mathbb{T}$ such that $d(x_0)=0$. The presence of singularity implies $\frac{1}{d(x-2\alpha)d(x)}\notin L^1(\T)$. Consequently, by \eqref{equation15}, we have
    $$|\det M_{\theta}(x)|=\frac{b}{|d(x-2\alpha)d(x)|}\notin L^1(\mathbb{T}).$$
    This contradicts $M_\theta(x)\in L^2(\mathbb{T},M(2,\mathbb{C}))$, thereby completing the proof.
	\end{proof}
	
	\subsection{Anderson Localization of $\mathcal{H}$}
	In this section, we aim to prove that  $ \mathcal{H}_{c,v,\theta}$ exhibits Anderson localization in the supercritical regime for almost every $\theta$:
	
\begin{theorem}\label{al-thm}
	Suppose that $\alpha\in DC(\gamma,\sigma)$. Then  $ \mathcal{H}_{c,v,\theta}$ exhibits Anderson localization in $\sigma^\mathcal{H}_{\text{sup}}(c,v)$
	 for every $\theta\in \Theta$, where 
	$$ \Theta= \cup_{\eta>0} \Theta(\eta) := \cup_{\eta>0}  \left\{\theta \in \T_0:\|2\theta- k\alpha\| \geq \frac{\eta}{|k|^{\sigma}}, \forall k\neq 0 \right\}.$$
\end{theorem}

We follow the basic setup presented in \cite{jitomirskaya1999metal}. Throughout the proof, we denote by $G_{[n_{1},n_{2}]}(n,m)$ the Green's function $( \mathcal{H}_{c,v,\theta}-E)^{-1}(n,m)$ of the operator $ \mathcal{H}_{c,v,\theta}$ when restricted to the interval $[n_{1},n_{2}]$, with zero boundary conditions at $n_{1}-1$ and $n_{2}+1$. For brevity, we will omit the explicit mention of the dependencies on $v$, $c$, and $\alpha$ whenever possible.

By Lemma 1.8 in \cite{JAC}, we have the inclusion
	$$\sigma^{\mathcal{H}}(c,v)\subseteq\left[-\lambda-2,\lambda+2\right].$$ 
    Let $M_k(\theta)$ denote the $k$-step transfer matrix for the equation $ \mathcal{H}_{c,v,\theta}u=Eu$, and define
	$$P_k(\theta):=\det\left[( \mathcal{H}_{c,v,\theta}-E)|_{[0,k-1]}\right],\quad Q_k(\theta):=\det\left[( \mathcal{H}_{c,v,\theta}-E)|_{[1,k]}\right].$$	
	Then, the $k$-step transfer matrix can be written as
    \footnotesize
	\begin{align*}
		M_k(\theta)=\frac{(-1)^k}{\prod_{j=0}^{k-1} c(\theta,j)}\begin{pmatrix}
		P_k(\theta)&c(\theta,-1)Q_{k-1}(\theta)\\-c(\theta,k-1)P_{k-1}(\theta)&-c(\theta,k-1)c(\theta,-1)Q_{k-2}(\theta)
	\end{pmatrix}=:\frac{(-1)^k}{\prod_{j=0}^{k-1} c(\theta,j)}\widetilde M_k(\theta),
\end{align*}
\normalsize
where the second equality defines $\widetilde M_k(\theta)$. This expression can be derived by expanding the determinant of $( \mathcal{H}_{c,v,\theta}-E)|_{[0,k-1]}$ along its final row and then using induction on $k$.

	By Kingman's subadditive ergodic theorem, the Lyapunov exponent satisfies
	\begin{equation*}
		L(E)=\inf_{n\ge1}\frac{1}{n}\int_{\mathbb{T}_0}\ln\|M_n(\theta)\|d\theta=\lim_{n\to\infty}\frac{1}{n}\ln\|(M_n(\theta)\|,
	\end{equation*}
	for almost every $\theta\in\mathbb{T}_0$. 
	
	Observe that if $E\in \sigma^\mathcal{H}_{\text{sup}}(c,v)$, we have
	\begin{equation*}
		\begin{split}
		0<L(E)&=\lim_{k\to\infty}\frac{1}{k}\int_{\T_0}\ln\|\widetilde M_k(\theta)\|d\theta-\lim_{k\to\infty}\frac{1}{k}\int_{\T_0}\ln|c(\theta,0)\cdots c(\theta,k-1)|d\theta\\&=:\widetilde{L}(E)-\frac{1}{2}\ln\left|\frac{\lambda}{2}\right|,
		\end{split}
	\end{equation*}
	where both limits exist by the subadditive ergodic theorem. Furthermore, we obtain the relation
	\begin{equation}\label{equation19}
		\widetilde{L}(E)=L(E)+\frac{1}{2}\ln\left|\frac{\lambda}{2}\right|=\inf_{k\ge1}\frac{1}{k}\int_{\T_0}\ln\|\widetilde M_k(\theta)\|d\theta=\lim_{k\to\infty}\frac{1}{k}\ln\|\widetilde M_k(\theta)\|,
	\end{equation}
	for any fixed $E$ and almost every $\theta\in\T_0$.
    
    Moreover, we recall the following result:
	\begin{proposition}[\cite{AJ1}]\label{theorem3.1.}
		Let $\alpha\in\R\backslash\Q$ and $A\in C^0(\T,M_2(\C))$ be measurable with $\ln\|A(x)\|\in L^1(\T)$. If $(\alpha,A)$ is a continuous cocycle, then for any $\delta>0$, there exists $C_\delta>0$ such that for all $n\in\N$ and $\theta\in\T$, we have
		$$\|A_n(\theta)\|\le C_\delta e^{(L(\alpha,A)+\delta)n}.$$
	\end{proposition}
	Then we can establish the uniform growth of the matrix $\widetilde M_{k}(\theta)$:
	\begin{lemma}\label{lemma3.2.}
		For every $E\in\sigma^{\mathcal{H}}(c,v)$ and $\epsilon>0$, there exists $k_1(E,\epsilon)$ such that
		$$|P_{k}(\theta)|\le\|\widetilde M_{k}(\theta)\|<e^{(L(E)+\frac{1}{2}\ln|\lambda/2|+\epsilon)k}$$
		for every $k>k_1(E,\epsilon)$ and every $\theta\in\mathbb{T}_0$.
	\end{lemma}
    
	\begin{proof}
    Note that $(2k\alpha, \widetilde M_{k}(\theta))$ is simply the $k$-th iteration of the cocycle $(2\alpha,\tilde B^E(\theta))$, where $\tilde B^E$ is defined in \eqref{be}.
	 By applying Proposition \ref{theorem3.1.}, we can obtain uniform growth 
     of $\widetilde M_{2k}(\theta)$, and subsequently for $\widetilde M_{2k+1}(\theta)$, since $\tilde B^E(\theta)$ is uniformly bounded.
	\end{proof}
		
Now, consider the matrix $\widetilde{M}_{2k}(\theta)$, given by
	\begin{align}\label{11-23-eq42}
    \widetilde{M}_{2k}(\theta)=\begin{pmatrix}
		P_{2k}(\theta)&c(\theta,-1)Q_{2k-1}(\theta)\\-c(\theta,2k-1)P_{2k-1}(\theta)&-c(\theta,-1)c(\theta,2k-1)Q_{2k-2}(\theta)
	\end{pmatrix}.
    \end{align}
	The key observation here is that elements of $\widetilde{M}_{2k}(\theta)$ can be controlled by $P_{2k}(\theta)$ (possible with different $k$ and different $\theta$) in a majority set:
	\begin{lemma}\label{lemma3.3.}
	For any $\theta\in\T_0$ and $k\in\Z$, the following equations hold: 
	\begin{align*}
	(v(\theta,2k-2)-E)P_{2k-1}(\theta)=P_{2k}(\theta)+c^2(\theta,2k-2)P_{2k-2}(\theta),\\
	(v(\theta,0)-E)Q_{2k-1}(\theta)=P_{2k}(\theta)+c^2(\theta,0)P_{2k-2}(\theta+2\alpha),
	\end{align*}
	and
	\begin{equation*}
	\begin{split}
	&(v(\theta,0)-E)(v(\theta,2k-2)-E)Q_{2k-2}(\theta)\\&\quad=P_{2k}(\theta)+c^2(\theta,0)P_{2k-2}(\theta+2\alpha)+c^2(\theta,2k-2)P_{2k-2}(\theta)\\&\quad\quad+c^2(\theta,2k-2)c^2(\theta,0)P_{2k-4}(\theta+2\alpha).
	\end{split}
	\end{equation*}
	\end{lemma}
	
	\begin{proof}
		We begin by noting the recurrence relations for $v$ and $c$, which state:
        $$v(\theta,2n)=v(\theta,2n+1)=c(\theta,2n),\quad  c(\theta,2n+1)=\lambda,\quad v(\theta,n+2)=v(\theta+2\alpha,n).$$ Expanding the determinant of $( \mathcal{H}_{c,v,\theta}-E)|_{[0,2k-1]}$ and  $( \mathcal{H}_{c,v,\theta}-E)|_{[1,2k-1]}$ along the last column, we obtain:
		\begin{equation*}
			\begin{split}
			&(v(\theta,2k-2)-E)P_{2k-1}(\theta)=P_{2k}(\theta)+c^2(\theta,2k-2)P_{2k-2}(\theta),\\
			&Q_{2k-1}(\theta)=(v(\theta,2k-2)-E)Q_{2k-2}(\theta)-c^2(\theta,2k-2)Q_{2k-3}(\theta).
			\end{split}
		\end{equation*}
		Similarly, expanding the determinant of $( \mathcal{H}_{c,v,\theta}-E)|_{[0,2k-1]}$ along its first column yields:
		\begin{equation*}
			P_{2k}(\theta)=(v(\theta,0)-E)Q_{2k-1}(\theta)-c^2(\theta,0)P_{2k-2}(\theta+2\alpha).
		\end{equation*}
		From this, we derive:
        \begin{align*}
            (v(\theta,0)-E)Q_{2k-1}(\theta)=P_{2k}(\theta)+c^2(\theta,0)P_{2k-2}(\theta+2\alpha).
        \end{align*}
        Substituting the expression for $Q_{2k-1}(\theta)$ into the determinant expansions, we derive:
        \begin{align*}
            (v&(\theta,0)-E)(v(\theta,2k-2)-E)Q_{2k-2}(\theta)\\&=P_{2k}(\theta)+c^2(\theta,0)P_{2k-2}(\theta+2\alpha)+c^2(\theta,2k-2)P_{2k-2}(\theta)\\&\quad+c^2(\theta,2k-2)c^2(\theta,0)P_{2k-4}(\theta+2\alpha).
        \end{align*}
        We thus finish the whole proof.
		\end{proof}

	\begin{lemma}\label{lemma3.4.}
	For every $k\in 2\N$,	there exists a polynomial $R_{\frac{k}{2}}$ of degree $\frac{k}{2}$ such that
		$$P_{k}(\theta)=R_{\frac{k}{2}}\left(\cos2\pi(\theta+(\frac{k}{2}-1)\alpha)\right).$$
	\end{lemma}
	\begin{proof}
The eigenfunction equation $ \mathcal{H}_{c,v,\theta}u=Eu$
can be written as:
$$
C\begin{pmatrix}
		u(2j+2)\\ u(2j+3) 
		\end{pmatrix} + C^* \begin{pmatrix}
		u(2j-2)\\ u(2j-1) 
		\end{pmatrix} + V_{j}(\theta) \begin{pmatrix}
		u(2j)\\ u(2j+1) 
		\end{pmatrix}=E\begin{pmatrix}
		u(2j)\\ u(2j+1) 
		\end{pmatrix},
$$
where $$C=\begin{pmatrix}
			0&0\\\lambda&0
		\end{pmatrix},\quad V_{j}(\theta)=\cos2\pi(2j\alpha+\theta)\begin{pmatrix}
			1&1\\1&1
		\end{pmatrix}.$$
For $k=2l$, the restriction of $ \mathcal{H}_{c,v,\theta}$ to $[0,k-1]$ is given by
		$$ \mathcal{H}_{c,v,\theta}|_{[0,k-1]}=\begin{pmatrix}
			V_{0}(\theta)&C&O_2&\cdots&O_2\\C^*&V_{1}(\theta)&C&\cdots&O_2\\O_2&C^*&V_{2}(\theta)&\cdots&O_2\\\vdots&\vdots&\vdots&\ddots&\vdots\\O_2&O_2&O_2&\cdots&V_{l-1}(\theta)
		\end{pmatrix}_{2l\times 2l},$$
		where $O_2$ is the $2\times2$ zero matrix.
        
        Define the matrix $Q$ as
$$Q=\begin{pmatrix}&&1\\&\begin{sideways}$\ddots$\end{sideways}& \\1& & 
		\end{pmatrix}_{2l\times2l}.$$	
	Applying the transformation $Q$ to $ \mathcal{H}_{c,v,\theta}|_{[0,k-1]}$, we find
    \small
	\begin{equation*}
	\begin{split}
	&QH_{c,v,\theta-(l-1)\alpha}|_{[0,k-1]}Q^{-1}\\&=\begin{pmatrix}
			V_{l-1}(\theta-(l-1)\alpha)&C&O_2&\cdots&O_2\\C^*&V_{l-2}(\theta-(l-1)\alpha)&C&\cdots&O_2\\O_2&C^*&V_{l-3}(\theta-(l-1)\alpha)&\cdots&O_2\\\vdots&\vdots&\vdots&\ddots&\vdots\\O_2&O_2&O_2&\cdots&V_{0}(\theta-(l-1)\alpha)
		\end{pmatrix}\\&=H_{c,v,-\theta-(l-1)\alpha}|_{[0,k-1]},
		\end{split}
		\end{equation*}
	\normalsize
    which implies
		\begin{equation}\label{meq20}
		P_{k}(\theta-(l-1)\alpha)=P_{k}(-\theta-(l-1)\alpha).
		\end{equation}
		Thus, the Fourier expansion of $\theta\to P_{k}(\theta-(l-1)\alpha)$ reads
		$$P_{k}(\theta)=\sum_{j=0}^{l}a_j\cos2\pi j(\theta+(l-1)\alpha),$$
		since all the $\sin$ terms are absent due to \eqref{meq20} and the degree obviously does not exceed $l$ due to
	$$UV_j(\theta)U^{-1}=\cos2\pi(2j\alpha+\theta)\begin{pmatrix}2&0\\0&0\end{pmatrix},$$
		where $U=\begin{pmatrix}1&1\\-1&1\end{pmatrix}$. From this representation the lemma follows since one can see inductively that the linear span of $\{1,\cos2\pi x,\dots,\cos2\pi lx\}$ is equal to that of $\{1,\cos2\pi x,\dots,\cos^{l}2\pi x\}$.
		\end{proof}

When the Lyapunov exponent is positive, Lemma \ref{lemma3.2.} implies that some of the entries must grow exponentially. These entries in turn appear in a description of Green's function of the operator restricted to a finite interval. Namely, by Cramer's Rule, if we denote
$$\Delta_{m,n}(\theta)=\det[( \mathcal{H}_{c,v,\theta}-E)|_{[m,n]}],$$
then for $n_1,n_2=n_1+k-1$, and $n\in[n_1,n_2]$, the following relations hold:
\begin{equation}\label{equation21}
	\begin{split}
		&\left|G_{[n_1,n_2]}(n_1,n)\right|=\left|\frac{\Delta_{n+1,n_2}(\theta)}{\Delta_{n_1,n_2}(\theta)}\prod_{l=n_1}^{n-1}c(\theta,l)\right|,\\
		&\left|G_{[n_1,n_2]}(n,n_2)\right|=\left|\frac{\Delta_{n_1,n-1}(\theta)}{\Delta_{n_1,n_2}(\theta)}\prod_{l=n+1}^{n_2}c(\theta,l)\right|.
	\end{split}
\end{equation}
	
	A useful definition related to Green's function is as follows:
	\begin{definition}
		Fix $E\in\mathbb{R}$ and $\xi\in\mathbb{R}$. A point $n\in\mathbb{Z}$ is called $(\xi,k)$-regular if there exists an interval $[n_1,n_2]$ , where $n_2=n_1+k-1$, containing $n$, such that
		$$|G_{[n_1,n_2]}(n,n_i)|<e^{-\xi|n-n_i|},\quad\text{and}\quad|n-n_i|\ge\frac{1}{7}k,\quad i=1,2.$$
Otherwise, $n$ is called $(\xi,k)$-singular.
		\end{definition}
	
	It is well known that any formal solution $u$ of the equation $ \mathcal{H}_{c,v,\theta}=Eu$ at a point $n\in[n_1,n_2]$ can be reconstructed from the boundary values via
\begin{equation}\label{equation24.}
	u(n)=-c(\theta,n_1-1)u(n_1-1)G_{[n_1,n_2]}(n,n_1)-c(\theta,n_2)u(n_2+1)G_{[n_1,n_2]}(n,n_2).
\end{equation}
	This implies that if $u$ is a generalized eigenfunction corresponding to $E$, then every point $n\in\mathbb{Z}$ where $u(n)\neq0$ is $(\xi,k)$-singular for sufficiently large $k$. In the following, we assume $u_0\neq0$ (if not, we can replace $u_0$ by $u_1$). 

\begin{proposition}\cite{jitomirskaya1999metal}\label{11-28-proposition7.11}
    Suppose $\alpha\in DC(\gamma,\sigma)$. If $f(x)$ is an analytic function, $z\in[\min f,\max f]$, then for all $\epsilon>0$ there exists a $N(\epsilon)$ such that for $n>N(\epsilon)$ and, if desired, any $l_0\in[0,\dots,n-1]$
    \begin{align*}
        \sum_{j=0,j\ne l_0}^{n-1}\ln|z-f(\theta+j\alpha)|\le n\left(\int_0^1\ln|z-f(\theta)|d\theta+\epsilon\right).
    \end{align*}
\end{proposition}

	Denote the set $W_{k,r}=\left\{x\in\mathbb{T}:|R_{\frac{k}{2}}(\cos2\pi x)|\le e^{(k+1)(r+\frac{1}{2}\ln\lambda)}\right\},\ k\in2\N.$
	Then we have the following result:
	\begin{lemma}\label{lemma3.6.}
		Suppose $k\in2\N$, and $y\in\mathbb{Z}$ is $(L(E)-\epsilon,k)$-singular. Then, for every $j\in\mathbb{Z}$ satisfying $y-\frac{5}{6}k\le2j\le y-\frac{1}{6}k$, we have that $\theta+\left(2j+\frac{k}{2}-1\right)\alpha$ belongs to $W_{k,L(E)-\frac{\epsilon}{8}}$ for $k>k_2(E,\epsilon)$.
	\end{lemma}
	\begin{proof}
		Since $y$ is $(L(E)-\epsilon,k)$-singular, without loss of generality, assume that for every interval $[2j,2j+k-1]$ of length $k$ containing $y$ with $y-\frac{5}{6}k\le 2j\le y-\frac{1}{6}k$, we have
		$$|G_{[2j,2j+k-1]}(y,2j+k-1)|\ge e^{-(L(E)-\epsilon)|y-2j-k+1|}.$$
		
		Using the second equation in \eqref{equation21}, it follows that
		\begin{equation}
			\begin{split}|G_{[2j,2j+k-1]}(y,2j+k-1)|&=\left|\frac{\Delta_{2j,y-1}(\theta)}{\Delta_{2j,2j+k-1}(\theta)}\prod_{l=y+1}^{2j+k-1}c(\theta,l)\right|=\left|\frac{P_{y-2j}(\theta+2j\alpha)}{P_{k}(\theta+2j\alpha)}\prod_{l=y+1}^{2j+k}c(\theta,l)\right|\\&\ge e^{-(L(E)-\epsilon)|y-2j-k+1|}.
			\end{split}
		\end{equation}
	
	By Lemma \ref{lemma3.2.}, we obtain
	$$|P_{y-2j}(\theta+2j\alpha)|\le e^{|y-2j|(L(E)+\frac{1}{2}\ln|\lambda/2|+\frac{\epsilon}{90})},\text{ for }k>6k_1(E,\frac{\epsilon}{90}),$$
	which implies that
	$$|P_k(\theta+2j\alpha)|\le e^{k(L(E)-\frac{2\epsilon}{15})}\left|\prod_{l=y+1}^{2j+k}c(\theta,l)\right|.$$

	On the other hand, by Proposition \ref{11-28-proposition7.11}, there exists $k_2(E,\epsilon)>6k_1(E,\frac{\epsilon}{90})$ such that if $k>k_2(E,\epsilon)$, then
	$$\ln\prod_{l=y+1}^{2j+k}|c(\theta,l)|\le(2j+k-y)\left(\frac{1}{2}\ln|\frac{\lambda}{2}|+\frac{\epsilon}{120}\right)$$
    for any $\theta\in\T_0$.   Consequently, by Lemma \ref{lemma3.4.}, we have
	\begin{equation*}
		|P_{k}(\theta+2j\alpha)|=\left|R_{\frac{k}{2}}(\cos2\pi(\theta+(2j+\frac{k}{2}-1)\alpha))\right|\le e^{(k+1)(L(E)+\frac{1}{2}\ln|\lambda/2|-\frac{\epsilon}{8})}.
	\end{equation*}
Thus, $\theta+(2j+\frac{k}{2}-1)\alpha\in W_{k,L(E)-\frac{\epsilon}{8}}$ for $k>k_2(E,\epsilon)$.
\end{proof}

	We can express the polynomial $R_k(x)$ in its Lagrange interpolation form as follows:
	\begin{equation}\label{equation24}
	|R_k(x)|=\left|\sum_{j=0}^kR_k(\cos2\pi\theta_j)\frac{\prod_{l\neq j}(x-\cos2\pi\theta_l)}{\prod_{l\neq j}(\cos2\pi\theta_j-\cos2\pi\theta_l)}\right|,
	\end{equation}
where $|\theta_j|\neq|\theta_l|\text{ for }j\neq l.$ 

We now introduce the following useful definition:
	\begin{definition}
		A set $\left\{\theta_0,\cdots,\theta_k\right\}$ is called $\epsilon$-uniform if it satisfies the condition:
$$\max_{x\in[-1,1]}\max_{i=0,\cdots,k}\prod_{j=0,j\neq i}^k\frac{|x-\cos2\pi\theta_j|}{|\cos2\pi\theta_i-\cos2\pi\theta_j|}<e^{k\epsilon}.$$
	\end{definition}

	\begin{lemma}\label{lemma3.8.}
		Let $0<\epsilon'<\epsilon$, $k\in2\mathbb{N}$, and $L(E)>0$. If $\theta_0,\cdots,\theta_{\frac{k}{2}}\in W_{k,L(E)-\epsilon}$, then the set $\left\{\theta_0,\cdots,\theta_{\frac{k}{2}}\right\}$ is not $\epsilon'$-uniform for $k>k_3(\epsilon,\epsilon')$.
	\end{lemma}
	\begin{proof}
		Assuming, for contradiction, that $\left\{\theta_0,\cdots,\theta_{\frac{k}{2}}\right\}$ is $\epsilon'-$uniform, we can utilize equation \eqref{equation24} to obtain the following inequality
		$$\left|R_{\frac{k}{2}}(x)\right|\le \left(\frac{k}{2}+1\right)e^{(k+1)(L(E)-\epsilon+\frac{1}{2}\ln|\lambda/2|+\epsilon')},$$
		for any $x\in[-1,1]$. This implies that
$$|P_k(\theta)|<e^{k(L(E)+\frac{1}{2}\ln|\lambda/2|-\frac{\epsilon-\epsilon'}{3})},\ \forall\theta\in\mathbb{R},$$
for sufficiently large $k$.

Recalling the notation $\T_t:=\{x:t+\cos2\pi(n\alpha+x)\neq0,\forall n\in\Z\}$, and employing equation \eqref{11-23-eq42} along with Lemma \ref{lemma3.3.}, we find
\footnotesize
\begin{align*}
\|\widetilde{M}_k(\theta)\|&\le(1+\lambda^2)\max\left\{|P_{k}(\theta)|,|Q_{k-1}(\theta)|,|P_{k-1}(\theta)|,|Q_{k-2}|\right\}\\&\le\frac{C}{|(v(\theta,0)-E)(v(\theta,k-2)-E)|}\max\{|P_{k}(\theta)|,|P_{k-2}(\theta+2\alpha)|,|P_{k-2}(\theta)|,|P_{k-4}(\theta+2\alpha)|\},
\end{align*}
\normalsize
for every $\theta\in\T_0\cap\T_{-E}$ and some constant $C=C(\lambda)$, since $|v(\theta,0)-E|,|v(\theta,k-2)-E|\le 3+\lambda$.

Thus, we obtain
\footnotesize
\begin{align*}
 \int_{\T_0\cap\T_{-E}}\ln\|\widetilde{M}_k(\theta)\|d\theta\le \ln C&-\int_{\T_0\cap\T_{-E}}\ln|E-\cos2\pi\theta|d\theta-\int_{\T_0\cap\T_{-E}}\ln|E-\cos2\pi(\theta+(k-2)\alpha)|d\theta\\&+\int_{\T_0\cap\T_{-E}}\ln\max\{|P_{k}(\theta)|,|P_{k-2}(\theta+2\alpha)|,|P_{k-2}(\theta)|,|P_{k-4}(\theta+2\alpha)|\}d\theta.
\end{align*}
\normalsize
Applying the Poisson-Jensen formula and noting that $|\T_0\cap\T_{-E}|=1$, we have
\begin{align*}
I_E:=\int_{\T_0\cap\T_{-E}}\ln|E-\cos2\pi\theta|d\theta&=\int_{\T_0\cap\T_{-E}}\ln|E-\cos2\pi(\theta+(k-2)\alpha)|d\theta\\&=\left\{\begin{array}{ll}
    \ln\frac{E^2+\sqrt{E^2-1}}{2} &\text{if }|E|\ge1,  \\
     -\ln2&\text{if } |E|\le1.\end{array}\right.
\end{align*}
Consequently, we obtain
\begin{align*}
    \int_{\T_0\cap\T_{-E}}\frac{1}{k}\ln\|\widetilde{M}_k(\theta)\|d\theta\le\frac{\ln C+I_E}{k}+L(E)+\frac{1}{2}\ln|\lambda/2|-\frac{\epsilon-\epsilon'}{3},
\end{align*}
which contradicts equation \eqref{equation19}  for sufficiently large $k$.
	\end{proof}

	Assume that $\{q_n\}_n$ is the sequence of denominators of the best rational approximations of $2\alpha$. Select $n$ such that $q_n\le\frac{y}{8}<q_{n+1}$, and let $s$ be the largest positive integer satisfying $sq_n\le\frac{y}{8}$. Set $I_1,I_2\subseteq\mathbb{Z}$ as follows
	$$I_1=[0,2sq_n]\text{ and }I_2=\left[\lfloor\frac{y}{2}\rfloor+1-2sq_n,\lfloor\frac{y}{2}\rfloor+2sq_n\right],$$
and we get $6sq_n+1$ points.
	
	\begin{lemma}\label{lemma3.9.}
	Let $\theta_j=\theta+(2j-1)\alpha$, then for any $\epsilon>0$, the set $\{\theta_j\}_{j\in I_1\cup I_2}$ is $\epsilon$-uniform if $y>y(\alpha,\theta,\epsilon)$.
	\end{lemma}

	\begin{proof}
		Take $x=\cos2\pi a$. Now it suffices to estimate
		$$\sum_{j\in I_1\cup I_2,j\neq i}\left(\ln|\cos2\pi a-\cos2\pi\theta_j|-\ln|\cos2\pi\theta_i-\cos2\pi\theta_j|\right)=\sum_1-\sum_2.$$
		Then Lemma \ref{ten} reduces this problem to estimating the minimal terms.
		
		First we estimate $\sum_1$:
		\begin{equation*}
			\begin{split}
				\sum_1 &=\sum_{j\in I_1\cup I_2,j\neq i}\ln|\sin\pi(a+\theta_j)|+\sum_{j\in I_1\cup I_2,j\neq i}\ln|\sin\pi(a-\theta_j)|+6sq_n\ln2\\
				&=\sum_{1,+}+\sum_{1,-}+6sq_n\ln2,
			\end{split}
		\end{equation*}
		we cut $\sum_{1,+}$ or $\sum_{1,-}$ into $6s$ sums and then apply Lemma \ref{ten}, we get that for some absolute constant $C_1$:
		$$\sum_1\le-6sq_n\ln2+C_1s\ln q_n.$$
		
		Next, we estimate $\sum_2$ as follows
		\begin{equation*}
			\begin{split}
				\sum_2&=\sum_{j\in I_1\cup I_2,j\neq i}\ln|\sin\pi(2\theta+(i+j-1)2\alpha)|\quad+\sum_{j\in I_1\cup I_2,j\neq i}\ln|\sin\pi(i-j)2\alpha|+6sq_n\ln2\\&=\sum_{2,+}+\sum_{2,-}+6sq_n\ln2.
			\end{split}
		\end{equation*}
		For any $0<|j|<q_{n+1}$, since $\alpha\in DC(\gamma,\tau)$  we have $$\|j2\alpha\|_{\T}\ge\|q_n2\alpha\|_{\T}\ge \frac{\gamma}{ (2q_n)^{\sigma}}.$$ 
        Therefore, we obtain 
		$$\max\left\{\ln|\sin x|,\ln|\sin(x+\pi j2\alpha)|\right\}\ge2\ln\gamma-2\sigma\ln2q_n\quad\text{for}\ y>y_1(\alpha).$$
		This means in any interval of length $sq_n$, there can be at most one term which is less than $2\ln\gamma-2\sigma\ln2q_n$. Then there can be at most 6 such terms in total.
		
		For the part $\sum_{2,-}$, since $$\|(i-j)2\alpha\|_{\T} \ge \frac{\gamma}{2^{\sigma} |i-j|^{\sigma}}\ge \frac{\gamma}{(18sq_n)^{\sigma}},$$ these 6 smallest terms must be bounded by $\ln\gamma-\sigma\ln(18sq_n)$ from below. Hence by Lemma \ref{ten}, we have 
		\begin{equation}\label{s2}\sum_{2,-}\ge-6sq_n\ln2+6\ln\gamma-6\sigma\ln(18sq_n)-C_2s\ln q_n,\end{equation}
		for $y>y_2(\alpha)$ and some absolute constant $C_2$.
		For the part $\sum_{2,+}$, since $\theta\in \Theta$, then $$\|2\theta+(i+j-1)2\alpha\|_{\T}\ge \frac{\eta}{|2(i+j)|^{\sigma}}\ge \frac{\eta}{(36sq_n)^{\sigma}},$$ these 6 smallest terms must be greater than $\ln\eta-\sigma\ln(36sq_n)$.  Therefore combining with \eqref{s2}, we have $$\sum_2\ge-6sq_n\ln2+6\ln\gamma-6\sigma\ln(18sq_n)+6\ln\eta-6\sigma\ln(36sq_n)-(C_2+C_3)s\ln q_n,$$
		consequently,  for any $\epsilon>0$ if $y>y(\alpha,\theta,\epsilon)$,
		$\sum_1-\sum_2\leq 12 \epsilon sq_n,$
		i.e.  the set $\{\theta_j\}_{j\in I_1\cup I_2}$ is $\epsilon-$uniform.
	\end{proof}

	\begin{proposition}\label{proposition3.10.}
		Assume that $\alpha\in DC,\ \theta\in\Theta, L(E)>0$. Then for every $\epsilon>0$, for any $|y|>y(\alpha,\theta,E,\epsilon)$ sufficiently large, there exists $k>\frac{3}{4}|y|$ such that $y$ is $(L(E)-\epsilon,k)-regular$.
	\end{proposition}
	
	\begin{proof}
		Combining Lemma \ref{lemma3.8.} and Lemma \ref{lemma3.9.}, we know that when $y$ is sufficiently large, $\{\theta_j\}_{j\in I_1\cup I_2}$ can not be inside the set $W_{12sq_n,L(E)-\frac{\epsilon}{8}}$ at the same time. Therefore, 0 and $y$ can not be $(L(E)-\epsilon,12sq_n)$-singular at the same time by Lemma \ref{lemma3.6.}. However, 0 is $(L(E)-\epsilon,12sq_n)-$singular given $y$ large enough. Therefore,
		$$\{\theta_j\}_{j\in I_1}\subset W_{12sq_n,L(E)-\frac{\epsilon}{8}},$$
		and $y$ must be $(L(E)-\epsilon,12sq_n)-$regular for $y>y(\alpha,\theta,E,\epsilon)$. Notice that $12sq_n>\frac{3}{4}y$, thus we complete the proof.
	\end{proof}

	\textbf{Proof of Theorem \ref{al-thm}:}
	It is well known  that if every generalized eigenfunction  of the ergodic Jacobi operator $ \mathcal{H}_{c,v,\theta}$ decays exponentially, then the operator $ \mathcal{H}_{c,v,\theta}$ displays Anderson localization. Let $E\in \sigma^\mathcal{H}_{\text{sup}}(c,v)$ be a generalized eigenvalue of $ \mathcal{H}_{c,v,\theta}$, and denote the corresponding generalized eigenfunction by $u$. Let $\epsilon$ small enough,
	by  \eqref{equation24.} and Proposition \ref{proposition3.10.}, if $|y|>y(\alpha,\theta,E,\epsilon)$
	the point $y$ is $(L(E)-\epsilon,k)$-regular for some $k>\frac{3}{4}|y|$. Thus, there exists an interval $[n_{1},n_{2}]$ of length $k$ containing $y$ such that 
	$\frac{1}{7}k\le |y-n_{i}|\le\frac{6}{7}k,$
	and
	$$\left|G_{[n_{1},n_{2}]}(y,n_{i})\right|<e^{-(L(E)-\epsilon)|y-n_{i}|}, \quad i=1,2.$$
	Using \eqref{equation24.}, we obtain that
	\begin{equation*}
		\begin{split}
			|u_y|\le2C(E)(2|y|+1)e^{-\frac{L(E)-\epsilon}{7}k} \le e^{-\frac{L(E)-\epsilon}{28}|y|}.
		\end{split}
	\end{equation*}
	This implies exponential decay of the eigenfunction if $\epsilon$  is chosen small enough.
	\qed

\bigskip
\textbf{Proof of Theorem \ref{theorem1.1.}:} It follows from Theorem \ref{theorem2.2} and Theorem \ref{al-thm}.

\section{Aubry duality}\label{duality}

In this section, we aim to establish Aubry duality for singular Jacobi operators on the strip.  As we mentioned, for the mosaic-type models, one can consider it as 
singular Jacobi operators on the strip:
\begin{equation}
[\mathcal{S}_{C,V,\omega} \boldsymbol{u}](n) = C \boldsymbol{u}(n + 1) + C^* \boldsymbol{u}(n - 1) + V(\omega + 2n\alpha) \boldsymbol{u}(n),
\end{equation}
where $C$ is a singular matrix. Indeed, direct computations show that 
\begin{example}(Mosaic model \eqref{mosaic})  In this case, one can take
    $$ C_1= \begin{pmatrix}
            0&0\\1&0
        \end{pmatrix}, V_1(\theta)=\begin{pmatrix}
            2\lambda\cos2\pi \theta&1\\1&0
        \end{pmatrix} $$
\end{example}

\begin{example}(Mosaic-type model  \eqref{IRS})   In this case, one can take
    $$ C_2= \begin{pmatrix}
            0&0\\\lambda &0
        \end{pmatrix}, V_2(\theta)=\cos2\pi \theta\begin{pmatrix}
            1&1\\1&1
        \end{pmatrix} $$
\end{example}

\subsection{Dynamical Aubry duality}
Assume that
    $[\mathcal{S}_{C,V,\theta} \boldsymbol{u}](n)=E\boldsymbol{u}(n),$
has an analytic quasiperiodic Bloch wave,
    $\boldsymbol{u}(n)=e^{2\pi in\omega}\widetilde{\boldsymbol{\psi}}(2n\alpha+\theta),$
where $\widetilde{\boldsymbol{\psi}}:\T\to\C^2$ being analytic and $\omega\in\T$ the Floquet exponent. If $\{\boldsymbol{\psi}(n)\}_{n\in\Z}$ are the Fourier coefficients of $\widetilde{\boldsymbol{\psi}}$, direct computation shows that they satisfy the following difference equation
\begin{align*}
    \sum_{k\in\Z}V^{(k)}\boldsymbol{\psi}(n-k)+(Ce^{2\pi i(2n\alpha+\omega)}+C^*e^{-2\pi i(2n\alpha+\omega)})\boldsymbol{\psi}(n)=E\boldsymbol{\psi}(n),\quad n\in\Z,
\end{align*}
where $\{V^{(k)}\}_{k\in\Z}$ are the Fourier coefficients of $V$, i.e.  $V(\theta)=\sum_{k\in\Z}V^{(k)}e^{2\pi ik\theta}.$ In the following, we call $$[\widehat{S}_{C,V,\omega}\boldsymbol{\psi}](n)=\sum_{k\in\Z}V^{(k)}\boldsymbol{\psi}(n-k)+(Ce^{2\pi i(2n\alpha+\omega)}+C^*e^{-2\pi i(2n\alpha+\omega)})\boldsymbol{\psi}(n)$$
the dual operator of $\mathcal{S}_{C,V,\theta}$.  Direct computations show the following:
\begin{example}(Dual operator of Mosaic model \eqref{mosaic}) 
 \small
    \begin{align*}
    &[\widehat{\mathcal{S}}_{C_1,V_1,\omega}\boldsymbol{u}](n)\\&=\begin{pmatrix}
            \lambda&0\\0&0
        \end{pmatrix}\boldsymbol{u}(n+1)+\begin{pmatrix}
            \lambda&0\\0&0
        \end{pmatrix}\boldsymbol{u}(n-1)+\begin{pmatrix}
            0&1+e^{-2\pi i(2n\alpha+\omega)}\\1+e^{2\pi i(2n\alpha+\omega)}&0
        \end{pmatrix}\boldsymbol{u}(n),
    \end{align*}
    \normalsize
\end{example}

\begin{example}(Dual operator of Mosaic-type model  \eqref{IRS}) 
\small
    \begin{align*}
        &[\widehat{\mathcal{S}}_{C_2,V_2,\omega}\boldsymbol{u}](n)\\&=\frac{1}{2}\begin{pmatrix}
            1&1\\1&1
        \end{pmatrix}\boldsymbol{u}(n+1)+\frac{1}{2}\begin{pmatrix}
            1&1\\1&1
        \end{pmatrix}\boldsymbol{u}(n-1)+\lambda\begin{pmatrix}
            0&e^{-2\pi i(2n\alpha+\omega)}\\ e^{2\pi i(2n\alpha+\omega)}&0
        \end{pmatrix}\boldsymbol{u}(n).
    \end{align*}
    \normalsize
\end{example}
Moreover, let $U_2$ be the following operator on $\ell^2(\Z,\C^2)$,
\begin{align*}
    [U_2\boldsymbol{u}](n)=\frac{1}{\sqrt{2}}\begin{pmatrix}1&1\\1&-1\end{pmatrix}\boldsymbol{u}(n).
\end{align*} Then, $U_2^*\widehat{\mathcal{S}}_{C_2,V_2,\omega}U_2$ and $\widehat{\mathcal{S}}_{C_2,V_2,\omega}$ are unitarily equivalent, 
thus share the  same spectral types. Indeed, one can easily check that  $U_2^*\widehat{\mathcal{S}}_{C_2,V_2,\omega}U_2$ is just Type III ME model \eqref{ac-sc}:
\footnotesize
\begin{align}
\label{2025-1-16-eq49}&[U_2^*\widehat{\mathcal{S}}_{C_2,V_2,\omega}U_2\boldsymbol{u}](n)\\&=\begin{pmatrix}
            1&0\\0&0
        \end{pmatrix}\boldsymbol{u}(n+1)+\begin{pmatrix}
            1&0\\0&0
        \end{pmatrix}\boldsymbol{u}(n-1)+\lambda\begin{pmatrix}
            \cos2\pi(2n\alpha+\omega)&i\sin2\pi(2n\alpha+\omega)\\ -i\sin2\pi(2n\alpha+\omega)&-\cos2\pi(2n\alpha+\omega)
        \end{pmatrix}\boldsymbol{u}(n).\nonumber
\end{align}
\normalsize
Since the base dynamic is minimal, the spectrum of the long-range operators $\widehat{\mathcal{S}}_{C,V,\omega}$ does not depend on the chosen $\omega$, so that one can write $
    \sigma^{\widehat{\mathcal{S}}}(C,V)=\sigma(\widehat{\mathcal{S}}_{C,V,\omega}).
    $

\subsection{Spectral Aubry duality} To rigorously establish the Aubry duality, let us consider the following Hilbert space,
\begin{align*}
    \mathbb{H}=L^2(\T\times\Z,\C^2),
\end{align*}
which consists of functions $\boldsymbol{\Psi}=\boldsymbol{\Psi}(\theta,n)$ satisfying
\begin{align*}
\sum_{n\in\Z}\int_{\T}\|\boldsymbol{\Psi}(\theta,n)\|^2d\theta<\infty.
\end{align*}

The extensions of the singular operators $\mathcal{S}$ and their long-range duals $\widehat{\mathcal{S}}$ to $\mathbb{H}$ are given in terms of their {\em direct integrals}, which we now define. The direct integral of the singular operator $\mathcal{S}_{C,V,\theta}$ is the operator $\mathbf{S}_{C,V}$, defined as
\begin{align*}
\mathbf{S}_{C,V}=\int_\T^{\oplus}\mathcal{S}_{C,V,\omega} d\omega:\int_\T^{\oplus}\ell^2(\Z,\C^2)d\omega\to\int_\T^{\oplus}\ell^2(\Z,\C^2)d\omega,
\end{align*}
and
$$[\mathbf{S}_{C,V}\boldsymbol{u}](\theta,n)=C\boldsymbol{u}(\theta,n+1)+C^{*}\boldsymbol{u}(\theta,n-1)+V(2n\alpha+\theta)\boldsymbol{u}(\theta,n).$$
Similarly, the direct integral of $\widehat{S}_{C,V,\omega}$ denoted as $\widehat{\mathbf{S}}_{C,V}$ is
\begin{align*}
[\widehat{\mathbf{S}}_{C,V}\boldsymbol{u}](\omega,n)=\sum_{k\in\Z}V^{(k)}\boldsymbol{\psi}(\omega,n-k)+(Ce^{2\pi i(2n\alpha+\omega)}+C^*e^{-2\pi i(2n\alpha+\omega)})\boldsymbol{\psi}(\omega,n).
\end{align*}
These two operators are bounded and self-adjoint in $\mathbb{H}$. 
Define the Aubry duality $\mathbf{A}_{2\alpha}$ on $\mathbb{H}$ as
\begin{align*}
[\mathbf{A}_{2\alpha}\boldsymbol{u}](\theta,n)=\sum_{k\in\Z}\int_\T e^{-2\pi i(\theta+2n\alpha)k}e^{-2\pi in\eta}\boldsymbol{u}(\eta,k)d\eta.
\end{align*}
Then, for any fixed real analytic $V$ and non-resonant frequency $2\alpha$, the direct integrals $\mathbf{S}_{C,V}$ and $\widehat{\mathbf{S}}_{C,V}$ are unitarily equivalent in the sense that the conjugation $
\mathbf{S}_{C,V}\mathbf{A}_{2\alpha}=\mathbf{A}_{2\alpha}\widehat{\mathbf{S}}_{C,V} $
holds.
Since $\mathbf{A}_{2\alpha}$ is an unitary operator, we have
\begin{align*}
    \sigma(\int^\oplus_\T \mathcal{S}_{C,V,\theta}d\theta)=\sigma(\int^\oplus_\T \widehat{\mathcal{S}}_{C,V,\theta}d\theta).
\end{align*}

\section{Integrated density of states and Thouless formula}\label{section5}	

Throughout the paper, we assume that $(\Omega, \mathcal{F}, \nu, T)$ is ergodic, and $\mathfrak{f}: \Omega \to M(2, \mathbb{C})$ is a bounded, measurable, and self-adjoint matrix. We define the potentials $V_\omega(n) = \mathfrak{f}(T^n \omega)$ for $\omega \in \Omega$ and $n \in \mathbb{Z}$, and consider the associated singular Jacobi operator on $\ell^2(\mathbb{Z}, \mathbb{C}^2)$:
\begin{equation}\label{sjo}
\mathcal{S}_{J,V,\omega} \boldsymbol{u} := J \boldsymbol{u}(n+1) + J \boldsymbol{u}(n-1) + V_\omega(n) \boldsymbol{u}(n) = z \boldsymbol{u}(n),
\end{equation}
where $J = \begin{pmatrix} 1 & 0 \\ 0 & 0 \end{pmatrix}$, $\boldsymbol{u}(n) = (a(n), b(n))^\mathsf{T}$, and $V_\omega(n) = \begin{pmatrix} V_{11,\omega}(n) & V_{12,\omega}(n) \\ V_{21,\omega}(n) & V_{22,\omega}(n) \end{pmatrix}$. 
Additionally, let $\Omega_0$ be a full-measure subset of $\Omega$ such that for all $\omega \in \Omega_0$ and $i \neq j$, the following conditions hold: $V_{12,\omega}(i) \neq 0$ and $V_{22,\omega}(i) \neq V_{22,\omega}(j)$.

\subsection{Singular Jacobi cocycle on the strip}
Note in the non-singular case $\det J\neq 0$, then \eqref{sjo} induces an symplectic cocycle \cite{puig}, however in our singular case, this doesn't work. 
Nevertheless, one can observe 
 $z\notin\{V_{22,\omega}(j):j\in\Z\}$, then \eqref{sjo} can be written as 
\begin{align}\label{eq6}
\begin{pmatrix}a(n+1)\\ a(n)\end{pmatrix}&=\begin{pmatrix}z-V_{11,\omega}(n)-\frac{|V_{12,\omega}(n)|^2}{z-V_{22,\omega}(n)}&-1\\1&0\end{pmatrix}\begin{pmatrix}a(n)\\ a(n-1)\end{pmatrix}\\&:=A^z(T^n\omega)\begin{pmatrix}a(n)\\a(n-1)\end{pmatrix}\nonumber
\end{align}
Meanwhile, we have
\begin{align}\label{eq5}
b(n)=\frac{V_{21,\omega}(n)}{z-V_{22,\omega}(n)}a(n)\quad\forall n\in\Z,
\end{align} thus to study $\mathcal{S}_\omega \boldsymbol{u}=z\boldsymbol{u}$, the main emphasize is study the cocyle $(T,A^z(\cdot))$, and we call it singular Jacobi cocycle on the strip induced by \eqref{sjo}.

For $z\notin\{V_{22,\omega}(j):j\in\Z\}$, denote $\boldsymbol{u}_1(\cdot,z)$ and	$\boldsymbol{u}_2(\cdot,z)$ the solutions of \eqref{sjo} which obey the initial conditions
\begin{align*}
\begin{pmatrix}a_1(1,z)&a_2(1,z)\\a_1(0,z)&a_2(0,z)\end{pmatrix}=\begin{pmatrix}1&0\\0&1\end{pmatrix}
\end{align*}
and
\begin{align*}
\begin{pmatrix}b_1(1,z)&b_2(1,z)\\b_1(0,z)&b_2(0,z)\end{pmatrix}=\begin{pmatrix}\frac{V_{21,\omega}(1)}{z-V_{22,\omega}(1)}&0\\0&\frac{V_{21,\omega}(0)}{z-V_{22,\omega}(0)}\end{pmatrix}.
\end{align*}

If $z=V_{22,\omega}(m)$ for some $m\in\Z$ and $\omega\in\Omega_0$, then it follows that $a(m)=0$ and
\begin{align}\label{11-25-eq7}
\begin{pmatrix}a(n+1)\\ a(n)\end{pmatrix}=\left\{\begin{array}{ll}A^z_{n-m}(T^m\omega)\begin{pmatrix}a(m+1)\\ 0\end{pmatrix}&\forall n\ge m,\\&\\(A^z_{m-n-1}(T^n\omega))^{-1}\begin{pmatrix}0\\ a(m-1)\end{pmatrix}&\forall n<m,\end{array}\right.
\end{align}
and
\begin{align}\label{11-25-eq8}
b(n)=&\left\{\begin{array}{ll}\frac{V_{21,\omega}(n)}{z-V_{22,\omega}(n)}a(n)&\forall n\neq m,\\&\\-\frac{a(m+1)+a(m-1)}{V_{12,\omega}(m)}&\forall n=m.\end{array}\right.
\end{align}

Then, one obtains the following:
\begin{proposition}
For any $z\notin\{V_{22,\omega}(j):j\in\Z\}$, we have
\begin{align*}
a_1(n+1,z)&=\frac{\det(z-\mathcal{S}_{\omega}|_{[1,n]})}{\prod_{j=1}^n(z-V_{22,\omega}(j))}\text{ for all }n\ge1,\\
a_2(n+1,z)&=-\frac{\det(z-\mathcal{S}_{\omega}|_{[2,n]})}{\prod_{j=2}^n(z-V_{22,\omega}(j))}\text{ for all }n\ge2,
\end{align*}
and
\begin{align}\label{8.5-eq10}
	A^z_n(\omega)=\begin{pmatrix}\frac{\det(z-\mathcal{S}_{\omega}|_{[1,n]})}{\prod_{j=1}^n(z-V_{22,\omega}(j))}&-\frac{\det(z-\mathcal{S}_{\omega}|_{[2,n]})}{\prod_{j=2}^{n}(z-V_{22,\omega}(j))}\\ & \\\frac{\det(z-\mathcal{S}_{\omega}|_{[1,n-1]})}{\prod_{j=1}^{n-1}(z-V_{22,\omega}(j))}&-\frac{\det(z-\mathcal{S}_{\omega}|_{[2,n-1]})}{\prod_{j=2}^{n-1}(z-V_{22,\omega}(j))}\end{pmatrix}\quad \text{for }n\ge3.
\end{align}
\end{proposition}

An energy $z$ of $\mathcal{S}_{\omega}$ is called an generalized eigenvalue if there is a formal solution to \eqref{sjo} with $\|\boldsymbol{u}(n)\|\le C(1+|n|)^{1/2+\epsilon}$ for some constants $C,\epsilon>0$. Then we have that the spectrum of $\mathcal{S}_{\omega}$ is given by the closure of the set of generalized eigenvalues of $\mathcal{S}_{\omega}$ \cite{Kirsch2007}. Denote $\mathcal{UH}=\{E\in\R:(T,A^E)\text{ is uniformly hyperbolic}\}.$ Then we have the following result:
\begin{proposition}\label{11-29-proposition3.6}
One has $
\cup_{\omega\in\Omega}\sigma(\mathcal{S}_{\omega})\subseteq\R\backslash\mathcal{UH}.$
Moreover, we have
\begin{align*}
\cup_{\omega\in\Omega}\sigma(\mathcal{S}_{\omega})\cap\mathcal{T}_\delta=(\R\backslash\mathcal{UH})\cap\mathcal{T}_\delta,
\end{align*}
where we denote
\begin{align*}
    \mathcal{T}_\delta=\{x\in\R:\operatorname{dist}(x,\operatorname{Ran}(V_{22}))\ge\delta>0\}.
\end{align*}
\end{proposition}	

\begin{proof}
  For each $\omega\in\Omega$, let $\mathcal{G}_\omega$ denote the set of generalized eigenvalues of $\mathcal{S}_{\omega}$. Since $V_{12,\omega}(n)\neq0$, the component $a$ of any nonzero solution $\boldsymbol{u}$ to \eqref{sjo} does not vanish. Moreover, if $E\in\mathcal{UH}$, then the sequence $\{a(j)\}_{j=-\infty}^\infty$ for any nonzero solution of the difference equation \eqref{sjo} must grow exponentially fast on at least one half-line. By definition, this implies $\mathcal{G}_\omega\subseteq\R\backslash\mathcal{UH}$. Since $\mathcal{UH}$ is an open subset of $\R$, it follows that $\sigma(\mathcal{S}_{\omega})\subseteq\R\backslash\mathcal{UH}$ as $\sigma(\mathcal{S}_{\omega})=\overline{\mathcal{G}_\omega}$.

  Suppose $E\in(\R\backslash\mathcal{UH})\cap\mathcal{T}_\delta$. Then there exist $\omega\in\Omega$ and $v\in\mathbb{S}^1$ with $\|A^E_n(\omega)v\|\le1$ for all $n$. Define $a:\Z\to\C$ and $b:\Z\to\C$ by $(a_1,a_0)^\mathsf{T}=v$, $\mathcal{S}_{\omega}\boldsymbol{u}=E\boldsymbol{u}$ and $b_n=\frac{V_{21,\omega}(n)}{E-V_{22,\omega}(n)}$. The vector $\boldsymbol{u}$ thus defined is an generalized eigenvector, implying $E\in\sigma(S_\omega)$.
\end{proof}
\subsection{Integrated density of states}
Denote the restriction of $\mathcal{S}_{\omega}$ to $[1,N]$ with Dirichlet boundary conditions by $\mathcal{S}_{\omega}|_{[1,N]}$, i.e.,
\small
\begin{align*}
    \mathcal{S}_{\omega}|_{[1,N]}=\begin{pmatrix}
			V_\omega(1)&J&&&\\J&V_{\omega}(2)&J&&\\&\ddots&\ddots&\ddots&\\&&J&V_\omega(N-1)&J\\&&&J&V_\omega(N)
		\end{pmatrix}_{2N\times 2N}.
\end{align*}
\normalsize
For $\omega\in\Omega$ and $N\ge1$, define measure $d\mathrm{n}_{\omega,N}$ by placing uniformly distributed point masses at the eigenvalues $E^{(1)}_{\omega,N}\le\cdots\le E^{(2N)}_{\omega,N}$ of $\mathcal{S}_\omega |_{[1,N]}$. That is
\begin{align*}
    \int gd\mathrm{n}_{\omega,N}:=\frac{1}{2N}\sum_{j=1}^{2N}g(E_{\omega,N}^{(j)})=\frac{1}{2N}\text{Tr}(g(\mathcal{S}_{\omega}|_{[1,N]})),
\end{align*}
for bounded measurable $g$.

\begin{proposition}\cite{damanik2022}\label{8.5-lemma3.7.}
For a.e. $\omega\in\Omega$, the measure $d\mathrm{n}_{\omega,N}$ has a weak limit, $$d\mathrm{n}:=\int d\omega\int \frac{1}{2}(d\mu_{\omega,\boldsymbol{\delta}_0}+d\mu_{\omega,\boldsymbol{\gamma}_0}),$$ as $N\to\infty$. 
\end{proposition}
\begin{proof}
One only needs to note the following relation  
 \begin{align*}
     \frac{1}{2N}\mathrm{Tr}((\mathcal{S}_\omega |_{[1,N]})^p)=\frac{1}{2N}\sum_{j=1}^N\langle \boldsymbol{\delta}_j,(\mathcal{S}_\omega |_{[1,N]})^p\boldsymbol{\delta}_j\rangle +\langle \boldsymbol{\gamma}_j,(\mathcal{S}_\omega |_{[1,N]})^p\boldsymbol{\gamma}_j\rangle,
 \end{align*}
 while the remaining details can be found in Theorem 4.3.8 of \cite{damanik2022}.
\end{proof}

Define 
\begin{align*}
    \mathrm{n}_{\omega,N}(E):=\frac{1}{2N}\#\left\{\text{eigenvalues }\le E\text{ of }\mathcal{S}_\omega |_{[1,N]}\right\}.
\end{align*}
Then, for a.e. $\omega\in\Omega$, $\mathrm{n}_{\omega,N}$ converges to a continuous, non-decreasing, function $E\mapsto \mathrm{n}(E)$, which is independent of $\omega$, and is called the integrated density of states (IDS) of the operator $\mathcal{S}_{\omega}$. In fact, the IDS is the distribution function of $d\mathrm{n}$. Moreover, $d\mathrm{n}$ is also continuous and $\text{supp}(d\mathrm{n})=\sigma^{\mathcal{S}}(J,V)$ \cite{AS}.

\subsection{Thouless formula}\label{thou}
Let $\mathcal{S}_{\omega}$ be the operator defined in \eqref{sjo}, and suppose that $\ln|z-f_{22}(\cdot)|\in L^1(\T)$ for any $z\in\C$. Then,
the following Thouless formula holds for singular Jacobi operators on the strip:
\begin{lemma}[Thouless formula]\label{11-27-lemma3.24}
For every $z\in\C$, we have
\begin{align*}
\frac{1}{2}L(T,A^z)+\frac{1}{2}\E(\ln|z-f_{22}(\cdot)|)=\int\ln|E-z|d\mathrm{n}(E).
\end{align*}
\end{lemma}	
	
\begin{proof}
Denote $B^z(\omega):=(z-f_{22}(\omega))A^z(\omega)$, then $L(T,B^z)=L(T,A^z)+\E(\ln|z-f_{22}(\cdot)|)$. 

Let us first consider $z\in\C\backslash\R$. According to \eqref{8.5-eq10}, we have
\begin{align*}
    B^z_n(\omega)=\begin{pmatrix}
        P_{n,\omega}(z)&-(z-f_{22}(T\omega))Q_{n-1,\omega}(z)\\ (z-f_{22}(T^n\omega))P_{n-1,\omega}(z)&-(z-f_{22}(T\omega))(z-f_{22}(T^{n}\omega))Q_{n-2,\omega}(z)
    \end{pmatrix}
\end{align*}
for $n\in\N$, where $P_{n,\omega}(z)=\det(z-\mathcal{S}_{\omega}|_{[1,n]})$ and $Q_{n,\omega}(z)=\det(z-\mathcal{S}_{T\omega}|_{[1,n]})$. Moreover, for $n\in\N$, we have
\begin{align*}
P_{n,\omega}(z)=\prod_{j=1}^{2n}\left(z-E_{\omega,n}^{(j)}\right),\quad Q_{n,\omega}(z)=\prod_{j=1}^{2n}\left(z-E_{T\omega,n}^{(j)}\right),
\end{align*}
where $E_{\omega,n}^{(1)},\dots,E_{\omega,n}^{(2n)}$ denote the eigenvalues of $\mathcal{S}_\omega |_{[1,n]}$. Therefore,
\begin{align*}
\frac{1}{2n}\ln\left|P_{n,\omega}(z)\right|=\int\ln|E-z|d\mathrm{n}_{\omega,n}(E),
\end{align*}
and
\begin{align*}
\frac{1}{2n}\ln\left|Q_{n,\omega}(z)\right|=\int\ln|E-z|d\mathrm{n}_{T\omega,n}(E).
\end{align*}

Since $z\in\C\backslash\R$, the function $E\to\ln|E-z|$ is bounded and continuous on $\sigma^{\mathcal{S}}(J,V)=\text{supp}(d\mathrm{n})$, so Proposition \ref{8.5-lemma3.7.} yields
\begin{align*}
\lim_{n\to\infty}\frac{1}{2n}\ln|P_{n,\omega}(z)|=\int\ln|E-z|d\mathrm{n}(E)
\end{align*}
for almost every $\omega$. Since all norms on $2\times2$ matrices are equivalent, it holds that
\begin{align}\label{12-6-eq35}
    \frac{1}{2}L(T,B^z)=\int\ln|E-z|d\mathrm{n}(E)
\end{align}
for $z\in\C\backslash\R$.

Let us now consider $E\in\R$. Denoting the right-hand side of \eqref{12-6-eq35} by $\gamma(z)$, we have
\begin{align*}
    \frac{1}{\pi r^2}\int_{|z-E|\le r}\frac{1}{2}L(T,B^z)dz=\frac{1}{\pi r^2}\int_{|z-E|\le r}\gamma(z)dz
\end{align*}
for every $r>0$, since the integrands agree on $B_r(E)\backslash\R$. Since both sides are subharmonic, we have
\begin{align*}
\frac{1}{2}L(T,B^E)=\lim_{r\to0}\frac{1}{\pi r^2}\int_{|z-E|\le r}\frac{1}{2}L(T,B^z)dz=\lim_{r\to0}\frac{1}{\pi r^2}\int_{|z-E|\le r}\gamma(z)dz=\gamma(E). 
\end{align*}
Thus, the equation \eqref{12-6-eq35} holds for every $z\in\C$.
\end{proof}

\subsection{Estimate of Lyapunov exponent}

\subsubsection{Harmonic analysis argument}\label{har-lya}
Deift and Simon \cite{DPSB} provide a harmonic function-based proof to estimate the lower bound of the Lyapunov exponent for Schr\"odinger cocycle. While their argument is feasible for single-chain Jacobi cocycles, it is only partially applicable to singular Jacobi cocycles on the strip. The key lies in constructing an analytic function derived from Thouless formula, defined as
\begin{align*}
F(z)=-2\cosh\Big(-2\int\ln(E-z)d\mathrm{n}(E)+\int\ln(V_{22}(\omega)-z)d\omega\Big),   
\end{align*}
where the branch of $\ln$ is taken with $\ln(-1)=-\pi i$, and $\ln(z)$ remains continuous in the region $\Im z\le0$. Denoting the argument by
\begin{align*}
    \beta(z)=2\int\operatorname{arg}(z-E)d\mathrm{n}(E)-\int\operatorname{arg}(z-V_{22}(\omega))d\omega,
\end{align*}
the Thouless formula implies that
\begin{align*}
    F(z)=-2\cosh(-L(T,A^z)+i\beta(z))=z-\E(V_{11}(\cdot))+O(|z|^{-1}),
\end{align*}
uniformly as $|z|\to\infty$ in the upper half-plane. Notably, $\Im F(z)=2\sin\beta(z)\sinh L(T,A^z)$ and $\Im z$ are harmonic functions. To establish the inequality $\Im F(z)\ge\Im z$. It suffices to verify $\Im F(z)\ge\Im z-\epsilon$ on the boundary of $D_{R,\epsilon}=\{z:|z|<R,\Im z>\epsilon\}$ for large $R$ and $\epsilon\to0$. However, on the boundary segment where $\Im z=\epsilon$, the inequality requires $\beta(z)\in[0,\pi]$, which is not satisfied in this case. Indeed, if $\Re z<\inf V_{22}$ and $\mathrm{n}(\Re z)>\frac{1}{2}$, we find that
\begin{align*}
    \lim_{\Im z\to0^+}\beta(z)=2\pi(1-\mathrm{n}(\Re z))-\pi=2\pi(\frac{1}{2}-\mathrm{n}(\Re z))<0.
\end{align*}
So we introduce a hyperbolic geometry argument:
\subsubsection{Hyperbolic geometry argument}
In the following, we consider the Schr\"odinger type cocycle
$$A(x) =S^{\tilde v}(x):=\begin{pmatrix}
           \tilde v(x)  &-1\\1&0
        \end{pmatrix}.$$
It is well-known result that if $\Im \tilde v>0$, then $(\alpha, S^{\tilde v})$ is uniformly hyperbolic \cite{aviladensity}, our aim to estimate its Lyapunov exponent.

For matrices in $SL(2,\C)$, an action on the Riemann Sphere $\overline{\C}$ through M\"{o}bius transformations is defined as:
\begin{align*}
    \begin{pmatrix}
        a&b\\c&d
    \end{pmatrix}\cdot z=\frac{az+b}{cz+d}.
\end{align*}
For $A\in SL(2,\C)$, let $\mathring{A}=QAQ^{-1}$ where
\begin{align*}
   Q=\frac{-1}{1+i}\begin{pmatrix}
        1&-i\\1&i
    \end{pmatrix}.
\end{align*}
The map $A\to\mathring{A}$ maps bijectively $SL(2,\R)$ to $SU(1,1)$, the real Lie group of matrices $\begin{pmatrix}
    u&\bar{v}\\ v&\bar{u}
\end{pmatrix},u,v\in\C$ such that $|u|^2-|v|^2=1$. Denote $\mathbb{D}$ the unit disk of  $\C$, then we have the following:

\begin{lemma}\label{hyb}\cite{avilacocycle}
 Suppose $(\alpha,A)\in\R\backslash\Q\times C^0(\T,SL(2,\C))$. If $\mathring{A}(x)\cdot\D\subset\D_{e^{-\epsilon}}$ for every $x\in \T$,  then $L(\alpha,A)>\frac{\epsilon}{2}$.
\end{lemma}

One can easily check that $|\mathring{S}^{\tilde v}\cdot(-1)|=1$, which implies that if $\Im \tilde v>0$, $S^{\tilde v}$ can't uniformly contract the upper half-plane. To see the hyberbolicity, one way is to consider  the open hemisphere of $\mathbb{PC}^2$ centered on the line through $\begin{pmatrix}
    i \\ 1
\end{pmatrix}$ \cite{aviladensity}, another way is to consider second iterate of the 
Schr\"odinger type cocycle:

\begin{lemma}\label{8.8-theorem4.5.}
Let  $(\alpha,S^{\tilde{v}})\in\R\backslash\Q\times C^0(\T,SL(2,\C))$. 
Assume that  $|\Re \tilde{v}|\leq l  $, $l\ge\Im \tilde{v}\geq \epsilon>0$, then $
    L(\alpha,S^{\tilde{v}})\ge C(l) \epsilon,$
where $C=C(l)>0$ is a universal constant.
\end{lemma}

\begin{proof}
For simplicity, denote $\tilde{v}(x)=\xi= p+iq$,  $\tilde{v}(x+\alpha)=\eta= s+it$, then the second iterate of the Schr\"odinger cocycle takes the form 
\begin{align*}
        B(x)=(S^{\tilde{v}})^2(x)= S^{\tilde{v}}(x+\alpha)S^{\tilde{v}}(x)=\begin{pmatrix}
            \eta\xi-1&-\eta\\\xi&-1
        \end{pmatrix},
    \end{align*}
by assumptions, $p,s\in[-l,l]$, $q,t\in[\epsilon,l]$, and let $l\ge1$.

We only need to calculate
\begin{align*}L=\sup_{\zeta\in\D}\|\mathring{B}(x)\cdot\zeta\|_0=\sup_{r\in\R}\|QB(x)\cdot r\|_0.
\end{align*}
Direct computation shows that 
\footnotesize
\begin{align*}
    L^2&=\frac{|(\eta\xi-1-i\xi)r-\eta+i|^2}{|(\eta\xi-1+i\xi)r-\eta-i|^2}\\&=1-\frac{4q(tq+1)r^2+4t(pr-1)^2}{[(sq+pt+p)^2+(-sp+tq+q+1)^2]r^2-2(ps^2+pt^2+2pt+p-s)r+[(t+1)^2+s^2]}\\&\le1-4\epsilon\frac{[\epsilon^2+1]r^2+(pr-1)^2}{[(3l^2)^2+(4l^2)^2]r^2+(6l^3)^2+r^2+5l^2}\\&\le1-4\epsilon\frac{r^2+(pr-1)^2}{l^6(26r^2+41)}\\&\le 1-\frac{4\epsilon}{41l^6}\frac{r^2+(pr-1)^2}{r^2+1}\\&=:1- \frac{4\epsilon}{41l^6}f(r,p).
\end{align*}
\normalsize
The first inequality is due to $(sq+pt+p)^2+(-sp+tq+q+1)^2>0$ and
\begin{align*}
    &4(ps^2+pt^2+2pt+p-s)^2-4[(sq+pt+p)^2+(-sp+tq+q+1)^2][(t+1)^2+s^2]\\&=-4(q+t+2qt+qs^2+qt^2+1)^2<0.
\end{align*}
Since $f(r,p)$ uniformly converges to $p^2+1$ on $\R\times[-l,l]$ as $r\to\pm\infty$, there exists $R>0$ such that $f(r,p)>\frac{1}{2}$ for every $|r|>R$ and every $p\in[-l,l]$. Moreover, since $f(r,p)$ is positive and continuous, it attains a minimum value on the compact set $[-R,R]\times[-l,l]$, and this minimum value is strictly positive, i.e.,
\begin{align*}
    \ln L\le-\frac{2\epsilon}{41l^6}\min_{(r,p)\in[-R,R]\times[-l,l]}\{\frac{1}{2},f(r,p)\})=:-c(l)\epsilon.
\end{align*}
This implies that $\mathring{B}(x)\cdot \D\subset\D_{e^{-c(l)\epsilon}}$,  by Lemma \ref{hyb}, $L(2\alpha, (S^{\tilde{v}})^2)\geq \frac{1}{2}c(l) \epsilon$, then the result follows. 
\end{proof}

\subsection{IDS and fibered rotation number}\label{ids-rota}
If $z\notin\{V_{22}(j):j\in\Z\}$ is an eigenvalue of $\mathcal{S}|_{[1,n]}$, one can easily check that it is simple. But for $z=V_{22}(m)$, the situation became complicated. First  
by \eqref{11-25-eq7} and \eqref{11-25-eq8}, for $z=V_{22,\omega}(m)$ and $\omega\in\Omega_0$, we can construct two solutions of \eqref{sjo} satisfying

\begin{align}
\begin{pmatrix}a_1(n+1,z)\\ a_1(n,z)\end{pmatrix}=&\left\{\begin{array}{ll}A^z_{n-m}(T^m\omega)\begin{pmatrix}1\\ 0\end{pmatrix}&\forall n\ge m,\\&\\\begin{pmatrix}0\\0\end{pmatrix}&\forall n<m,\end{array}\right.\label{11-25-eq9}\\ \begin{pmatrix}a_2(n+1,z)\\ a_2(n,z)\end{pmatrix}=&\left\{\begin{array}{ll}(A^z_{m-n-1}(T^m\omega))^{-1}\begin{pmatrix}0\\1\end{pmatrix}&\forall n< m,\\&\\\begin{pmatrix}0\\0\end{pmatrix}&\forall n\ge m,\end{array}\right.\label{11-25-eq10}
\end{align}
and
\begin{align*}
b_1(n,z)=\left\{\begin{array}{ll}\frac{V_{21,\omega}(n)}{z-V_{22,\omega}(n)}a_1(n)&\forall n>m,\\&\\-\frac{1}{V_{12,\omega}(m)}&n=m,\\&\\0&\forall n<m,\end{array}\right.\quad b_2(n,z)=\left\{\begin{array}{ll}0&\forall n>m,\\&\\-\frac{1}{V_{12,\omega}(m)}&n=m,\\&\\\frac{V_{21,\omega}(n)}{z-V_{22,\omega}(n)}a_2(n)&\forall n<m,\end{array}\right.
\end{align*}
and denote by $\boldsymbol{u}_1(\cdot,z),\boldsymbol{u}_2(\cdot,z)$, respectively. Then  we have the following:

\begin{lemma}\label{7-17-3.4}
Let $\boldsymbol{u}$ be an eigenvector of $\mathcal{S}_\omega|_{[1,N]}$ with $\omega\in\Omega_0$ and $N\ge2$ corresponding to the eigenvalue $z$. 
\begin{itemize}
\item[(i)] If $z\neq V_{22,\omega}(j),\forall j\in[1,N]$, then $a(1)a(N)\neq0$; if $z=V_{22,\omega}(1)$, then $a(N)\neq0$; if $z=V_{22,\omega}(N)$, then $a(1)\neq0$. Moreover, if $z\neq V_{22,\omega}(j),\forall j\in[2,N-1]$, then $z$ is simple.

\item[(ii)] If $z=V_{22,\omega}(m)$ for some $m\in[2,N-1]$, then the multiplicity does not exceed $2$. Moreover, if the multiplicity is $2$, then there exists an eigenvector whose $a(N)$ component is non-zero, and  there exists an eigenvector whose $a(N)$ component is zero.
\end{itemize}
\end{lemma}

\begin{proof}
Since $\omega\in\Omega_0$, we have $V_{12,\omega}(i)\neq0$ and $V_{22,\omega}(i)\neq V_{22,\omega}(j)$ for all $i\neq j$.

(i) Suppose $z\neq V_{22,\omega}(j),\forall j\in[1,N]$. Note that
\begin{align*}
(V_{11,\omega}(1)-z)a(1)+V_{12,\omega}(1)b(1)+a(2)=&0\\
V_{21,\omega}(1)a(1)+(V_{22,\omega}(1)-z)b(1)=&0\\
a(1)+(V_{11,\omega}(2)-z)a(2)+V_{12,\omega}(2)b(2)+a(3)=&0\\
V_{21,\omega}(2)a(2)+(V_{22,\omega}(2)-z)b(2)=&0\\
\dots&
\end{align*}
If $a(1)=0$, then $b(1)=a(2)=b(2)=0$, and generates the zero vector by induction; contradiction. Similarly, $a(N)$ is also non-zero.

Suppose $z=V_{22,\omega}(1)$, then $a(1)=0$ and $a(2)=-V_{12,\omega}(1)b(1)$. By using the eigenvalue equation and induction, we have 
\begin{align*}
\boldsymbol{u}(j)=a(2)\boldsymbol{u}_1(j,z)\text{ for each }1\le j\le N.
\end{align*}
and $
a_1(N+1,z)=0.$
This implies that the eigenvalue $z=V_{22,\omega}(1)$ is simple and $a_1(N)\neq0$. Since $a(2)\neq0$, we have $a(N)=a(2)a_1(N)\neq0$.The argument is similar for $z=V_{22,\omega}(N)$.

Suppose $z\neq V_{22,\omega}(j),\forall j\in[2,N-1]$ is degenerate, choose two linearly independent eigenvectors. The above indicates that a non-trivial linear combination starts off with a zero generates the zero vector; contradiction.

(ii) Suppose $z=V_{22,\omega}(m)$ for some $m\in[2,N-1]$. By using the eigenvalue equation and induction, we have
\begin{align*}
&\boldsymbol{u}(j)=a(m-1)\boldsymbol{u}_2(j,z)\text{ for each }1\le j\le m-1,\\
&\boldsymbol{u}(j)=a(m+1)\boldsymbol{u}_1(j,z)\text{ for each }m+1\le j\le N,\\
&\boldsymbol{u}(m)=(0,-\frac{a(m-1)+a(m+1)}{V_{12,\omega}(m)})^{\mathsf{T}},
\end{align*}
and $
a(m-1)a_2(0,z)=a(m+1)a_1(N+1,z)=0.$
Thus, the multiplicity does not exceed $2$. If the multiplicity is $2$, let $(a(m-1),a(m+1))=(1,0)\text{ or }(0,1)$, resulting in two linearly independent eigenvectors and $a_2(0,z)=a_1(N+1,z)=0$. Therefore, neither $a_2(1,z)$ nor $a_1(N,z)$ vanishes; otherwise, $a_2(m-1,z)$ or $a_1(m+1,z)$ would be zero, leading to a contradiction with \eqref{11-25-eq9} and \eqref{11-25-eq10}. Furthermore, we deduce that $a(1)$ and $a(N)$ cannot vanish at the same time.
\end{proof}

\begin{corollary}\label{11-23-corollary3.14}
  For every $N\ge2$, $\{E_{N}^{(j)}\}_{j=1}^{2N}$ denote the $2N$ eigenvalues of $\mathcal{S}_\omega|_{[1,N]}$, with $\omega\in\Omega_0$, sorted in ascending order by $j$. Let $\boldsymbol{v}_1,\dots,\boldsymbol{v}_{2N}$ be the orthonormal eigenvectors of $\mathcal{S}_\omega|_{[1,N]}$, chosen so that $\mathcal{S}_\omega|_{[1,N]}\boldsymbol{v}_j=E^{(j)}_N\boldsymbol{v}_j$. Then the pair $(E_N^{(j)},\langle \boldsymbol{v}_j,\boldsymbol{\delta}_N\rangle )$ can only be one of the following cases:
  \begin{enumerate}
      \item[(i)] If $E_N^{(j)}=V_{22,\omega}(N)$, then $\langle \boldsymbol{v}_j,\boldsymbol{\delta}_N\rangle =0$ and $V_{22,\omega}(N)\notin\sigma(\mathcal{S}_\omega|_{[1,N-1]})$;

      \item[(ii)] If $E_N^{(j)}\neq V_{22,\omega}(N)$ is a simple eigenvalue of $\mathcal{S}_\omega|_{[1,N]}$ and $\langle \boldsymbol{v}_j,\boldsymbol{\delta}_N\rangle =0$, then it is also an eigenvalue of $\mathcal{S}_\omega|_{[1,N-1]}$.

      \item[(iii)] If $E_N^{(j)}=E_N^{(j+1)}$ are two equal eigenvalues of $\mathcal{S}_\omega|_{[1,N]}$, then the multiplicity is $2$, and exactly one of $\langle \boldsymbol{v}_j,\boldsymbol{\delta}_N\rangle , \langle \boldsymbol{v}_{j+1},\boldsymbol{\delta}_N\rangle $ is zero and the other must be non-zero. Moreover, $E_N^{(j)}$ is a simple eigenvalue of $\mathcal{S}_\omega|_{[1,N-1]}$.

      \item[(iv)]  Otherwise, $E_N^{(j)}$ is a simple eigenvalue of $\mathcal{S}_\omega|_{[1,N]}$ and $\langle \boldsymbol{v}_j,\boldsymbol{\delta}_N\rangle \neq0$, which implies $E_N^{(j)}\notin\sigma(\mathcal{S}_\omega|_{[1,N-1]})$.
  \end{enumerate}
\end{corollary}

\begin{proof}
Define a meromorphic function $h$ by
\begin{align*}
h(z)=\frac{\det(z-\mathcal{S}_\omega|_{[1,N-1]})}{\det(z-\mathcal{S}_\omega|_{[1,N]})}.
\end{align*}
By Cramer's rule, 
\begin{align*}h(z)=&\frac{1}{z-V_{22,\omega}(N)}\langle \boldsymbol{\delta}_N,(z-\mathcal{S}_\omega|_{[1,N]})^{-1}\boldsymbol{\delta}_N\rangle 
\end{align*}
 for all $z\notin\sigma(\mathcal{S}_\omega|_{[1,N]})$. If we expand $\boldsymbol{\delta}_N$  in the basis $\{\boldsymbol{v}_1,\dots,\boldsymbol{v}_{2N}\}$, we see that
\begin{align}
h(z)=&\frac{1}{z-V_{22,\omega}(N)}\langle \sum_{j=1}^{2N}\langle \boldsymbol{v}_j,\boldsymbol{\delta}_N\rangle \boldsymbol{v}_j,\sum_{j=1}^{2N}\frac{1}{z-E^{(j)}_N}\langle \boldsymbol{v}_j,\boldsymbol{\delta}_N\rangle \boldsymbol{v}_j\rangle \nonumber\\=&\frac{1}{z-V_{22,\omega}(N)}\sum^{2N}_{j=1}\frac{|\langle \boldsymbol{v}_j,\boldsymbol{\delta}_N\rangle |^2}{z-E_N^{(j)}}\label{11-23-eq11}
\end{align}
for all $z\notin\sigma(\mathcal{S}_\omega|_{[1,N]})\}$. 

(i) If $E_N^{(j)}=V_{22,\omega}(N)$, then $\langle \boldsymbol{v}_j,\boldsymbol{\delta}_N\rangle =0$ by $\mathcal{S}_\omega|_{[1,N]}\boldsymbol{v}_j=V_{22,\omega}(N)\boldsymbol{v}_j$. Since $V_{22,\omega}(N)$ is a pole of $h$ with order 1, $V_{22,\omega}(N)$ is not an eigenvalue of $\mathcal{S}_\omega|_{[1,N-1]}$.

(ii) If $E_N^{(j)}\neq V_{22,\omega}(N)$ is a simple eigenvalue of $\mathcal{S}_\omega|_{[1,N]}$ and $\langle \boldsymbol{v}_j,\boldsymbol{\delta}_N\rangle =0$, then it is not a pole of $h$ and is an eigenvalue of $\sigma(\mathcal{S}_\omega|_{[1,N-1]})$.

(iii) If $E_N^{(j)}=E_N^{(j+1)}$, then the first part it is the direct result  by Lemma \ref{7-17-3.4}. Obviously, $E_N^{(j)}\neq V_{22,\omega}(N)$. By \eqref{11-23-eq11}, such eigenvalues is a pole of $h$ with order $1$, which implies that it is an eigenvalue of $\mathcal{S}_\omega|_{[1,N-1]}$.

(iv) If $E_N^{(j)}$ is a simple eigenvalue of $\mathcal{S}_\omega|_{[1,N]}$ and $\langle \boldsymbol{v}_j,\boldsymbol{\delta}_N\rangle \neq0$, then it is not an eigenvalue of $\mathcal{S}_\omega|_{[1,N-1]}$ by \eqref{11-23-eq11}.
\end{proof}

Now, we can obtain the following Interlacing Lemma:
\begin{lemma}[Interlacing Lemma]\label{11-23-lemma3.15}
Let $\{E_N^{(n_j)}\}_{j=1}^K$ be all the eigenvalues of $\mathcal{S}_\omega|_{[1,N]}$ with $\langle \boldsymbol{v}_{n_j},\boldsymbol{\delta}_N\rangle \neq0$ and $\omega\in\Omega_0$, sorted in ascending order by $j$, then $K\ge3$ and all the other $2N-K$ eigenvalues of $\mathcal{S}_\omega|_{[1,N]}$ except $V_{22,\omega}(N)$ are belong to $\sigma(\mathcal{S}_\omega|_{[1,N-1]})$. We also have the following strictly interlaced properties:
    
\textbf{Case 1:} If $V_{22,\omega}(N)\notin\sigma(\mathcal{S}_\omega|_{[1,N]})$,
then there exists $m\in[1,K-1]$ such that $$E_N^{(n_m)}<V_{22,\omega}(N)<E_N^{(n_{m+1})}.$$
Moreover, there exist $K-2$ eigenvalues $\{E_{N-1}^{(\tilde n_j)}\}_{j=1}^{K-2}$ of $\mathcal{S}_\omega|_{[1,N-1]}$ satisfying
\begin{align*}
    E_N^{(n_j)}<E_{N-1}^{(\tilde n_j)}<E_N^{(n_{j+1})}
\end{align*}
for $1\le j\le m-1$, and
\begin{align*}
    E_N^{(n_j)}<E_{N-1}^{(\tilde n_{j-1})}<E_N^{(n_{j+1})}
\end{align*}
for $m+1\le j\le K-1$. 

\textbf{Case 2:} If $V_{22,\omega}(N)\in\sigma(\mathcal{S}_\omega|_{[1,N]})$, then there exist $K-1$ eigenvalues $\{E_{N-1}^{(\tilde n_j)}\}_{j=1}^{K-1}$ of $\mathcal{S}_\omega|_{[1,N-1]}$ satisfying
\begin{align*}
    E_N^{(n_j)}<E_{N-1}^{(\tilde n_j)}<E_N^{(n_{j+1})}
\end{align*}
for $1\le j\le K-1$.
\end{lemma}

\begin{proof}
We inherit the notations in Corollary \ref{11-23-corollary3.14}. By Lemma \ref{7-17-3.4}, there are at most $2N-3$ eigenvalues such that the corresponding $\langle \boldsymbol{v}_j,\boldsymbol{\delta}_N\rangle $ vanishes, thus $K\ge3$. By Corollary \ref{11-23-corollary3.14} (i)-(iii), all the other $2N-K$ eigenvalues $E_N^{(j)}$ (i.e., $\langle \boldsymbol{v}_j,\boldsymbol{\delta}_N\rangle \neq0$) of $\mathcal{S}_\omega|_{[1,N]}$ except $V_{22,\omega}(N)$ are belong to $\sigma(\mathcal{S}_\omega|_{[1,N-1]})$.

\textbf{Case 1}: If $V_{22,\omega}(N)\notin\sigma(\mathcal{S}_\omega|_{[1,N]})$, then $\lim_{z\to V_{22,\omega}(N)}h(z)$ exists. By Corollary \ref{11-23-corollary3.14} (ii), we have $$\lim_{z\to V_{22,\omega}(N)}\sum^{2N}_{j=1}\frac{|\langle \boldsymbol{v}_j,\boldsymbol{\delta}_N\rangle |^2}{z-E_N^{(j)}}=\sum_{\langle \boldsymbol{v}_j,\boldsymbol{\delta}_N\rangle \neq0}\frac{|\langle \boldsymbol{v}_j,\boldsymbol{\delta}_N\rangle |^2}{V_{22,\omega}(N)-E_N^{(j)}}=0.$$
Therefore, there exists $m\in[1,K-1]$ such that $E_N^{(n_m)}<V_{22,\omega}(N)<E_N^{(n_{m+1})}.$ 

Applying Corollary \ref{11-23-corollary3.14}, one has
\begin{align*}
\det(z-\mathcal{S}_\omega|_{[1,N-1]})=p(z)\prod_{\tiny\begin{matrix}\langle \boldsymbol{v}_j,\boldsymbol{\delta}_N\rangle \neq0,\\ E_N^{(j)}\text{ is degenerate}\\\text{for }\mathcal{S}_\omega|_{[1,N]}\end{matrix}}(z-E_N^{(j)})\prod_{\tiny\begin{matrix}\langle \boldsymbol{v}_j,\boldsymbol{\delta}_N\rangle =0,\\ E_N^{(j)}\text{ is simple}\\\text{for }\mathcal{S}_\omega|_{[1,N]}\end{matrix}}(z-E_N^{(j)})
\end{align*}
\normalsize
and
\begin{align*}
&\det(z-\mathcal{S}_\omega|_{[1,N]})\\=&\prod_{\tiny\begin{matrix}\langle \boldsymbol{v}_j,\boldsymbol{\delta}_N\rangle \neq0\\ E_N^{(j)}\text{ is simple}\\\text{for }S_\omega|_{[1,N]}\end{matrix}}(z-E_N^{(j)})\prod_{\tiny\begin{matrix}\langle \boldsymbol{v}_j,\boldsymbol{\delta}_N\rangle \neq0,\\ E_N^{(j)}\text{ is degenerate}\\\text{for }\mathcal{S}_\omega|_{[1,N]}\end{matrix}}(z-E_N^{(j)})^2\prod_{\tiny\begin{matrix}\langle \boldsymbol{v}_j,\boldsymbol{\delta}_N\rangle =0,\\ E_N^{(j)}\text{ is simple}\\\text{for }\mathcal{S}_\omega|_{[1,N]}\end{matrix}}(z-E_N^{(j)})
\end{align*}
\normalsize
where $p(z)$ is a polynomial, $p(E_N^{(j)})\neq0$ if $E_N^{(j)}$ is degenerate for $\mathcal{S}_\omega|_{[1,N]}$. Therefore, the meromorphic function $h$ can be simplified to
\begin{align*}
h(z)=\frac{p(z)}{\prod_{\tiny\begin{matrix}\langle \boldsymbol{v}_j,\boldsymbol{\delta}_N\rangle \neq0\end{matrix}}(z-E_N^{(j)})}=\frac{1}{z-V_{22}(N)}\sum_{\tiny\begin{matrix}\langle \boldsymbol{v}_j,\boldsymbol{\delta}_N\rangle \neq0\end{matrix}}\frac{|\langle \boldsymbol{v}_{j},\boldsymbol{\delta}_N\rangle |^2}{z-E_N^{(j)}},
\end{align*}
which implies that the degree of $p(z)$ is $K-2$.

For each $1\le j\le m-1$, we have
\begin{align*}
\lim_{z\downarrow E_{N}^{(n_j)}}h(z)=-\infty\text{ and }\lim_{z\uparrow E_{N}^{(n_{j+1})}}h(z)=+\infty,
\end{align*}
so $p(z)$ vanishes somewhere in the open interval $\left(E_N^{(n_j)},E_N^{(n_{j+1})}\right)$. It follows that $\mathcal{S}_\omega|_{[1,N-1]}$ has an eigenvalue strictly between $E_N^{(n_j)}$ and $E_N^{(n_{j+1})}$ for 
each $1\le j\le m-1$.

Similarly, for each $m+1\le j\le K-1$, we have
\begin{align*}
\lim_{z\downarrow E_{N}^{(n_j)}}h(z)=+\infty\text{ and }\lim_{z\uparrow E_{N}^{(n_{j+1})}}h(z)=-\infty,
\end{align*}
so $p(z)$ vanishes somewhere in the open interval $\left(E_N^{(n_j)},E_N^{(n_{j+1})}\right)$. It follows that $\mathcal{S}_\omega|_{[1,N-1]}$ has an eigenvalue strictly between $E_N^{(n_j)}$ and $E_N^{(n_{j+1})}$ for 
each $m+1\le j\le K-1$. So we have found all the $K-2$ roots of $p(z)$.

\textbf{Case 2}: If $V_{22,\omega}(N)\in\sigma(\mathcal{S}_\omega|_{[1,N]})$, then $V_{22,\omega}(N)\notin\sigma(\mathcal{S}_\omega|_{[1,N-1]})$. Applying Corollary \ref{11-23-corollary3.14}, one has
\begin{align*}
\det(z-\mathcal{S}_\omega|_{[1,N-1]})=q(z)\prod_{\tiny\begin{matrix}\langle \boldsymbol{v}_j,\boldsymbol{\delta}_N\rangle \neq0,\\ E_N^{(j)}\text{ is degenerate}\\\text{for }\mathcal{S}_\omega|_{[1,N]}\end{matrix}}(z-E_N^{(j)})\prod_{\tiny\begin{matrix}\langle \boldsymbol{v}_j,\boldsymbol{\delta}_N\rangle =0,\\ E_N^{(j)}\neq V_{22,\omega}(N)\text{ is}\\\text{simple for }\mathcal{S}_\omega|_{[1,N]}\end{matrix}}(z-E_N^{(j)})
\end{align*}
\normalsize
and
\small
\begin{align*}
&\det(z-\mathcal{S}_\omega|_{[1,N]})\\&=(z-V_{22,\omega}(N))\prod_{\tiny\begin{matrix}\langle \boldsymbol{v}_j,\boldsymbol{\delta}_N\rangle \neq0\\ E_N^{(j)}\text{ is simple}\\\text{for }\mathcal{S}_\omega|_{[1,N]}\end{matrix}}(z-E_N^{(j)})\prod_{\tiny\begin{matrix}\langle \boldsymbol{v}_j,\boldsymbol{\delta}_N\rangle \neq0,\\ E_N^{(j)}\text{ is degenerate}\\\text{for }\mathcal{S}_\omega|_{[1,N]}\end{matrix}}(z-E_N^{(j)})^2\prod_{\tiny\begin{matrix}\langle \boldsymbol{v}_j,\boldsymbol{\delta}_N\rangle =0,\\ E_N^{(j)}\neq V_{22,\omega}(N)\text{ is}\\\text{simple for }\mathcal{S}_\omega|_{[1,N]}\end{matrix}}(z-E_N^{(j)})
\end{align*}
\normalsize
where $q(z)$ is a polynomial, $q(E_N^{(j)})\neq0$ if $E_N^{(j)}$ is degenerate for $\mathcal{S}_\omega|_{[1,N]}$. Therefore, the meromorphic function $h$ can be simplified to
\begin{align*}
h(z)=\frac{q(z)}{(z-V_{22,\omega}(N))\prod_{\tiny\begin{matrix}\langle \boldsymbol{v}_j,\boldsymbol{\delta}_N\rangle \neq0\end{matrix}}(z-E_N^{(j)})}=\frac{1}{z-V_{22,\omega}(N)}\sum_{\tiny\begin{matrix}\langle \boldsymbol{v}_j,\boldsymbol{\delta}_N\rangle \neq0\end{matrix}}\frac{|\langle \boldsymbol{v}_{j},\boldsymbol{\delta}_N\rangle |^2}{z-E_N^{(j)}},
\end{align*}
which implies that the degree of $q(z)$ is $K-1$.

(i) If there exists $m\in[1,K-1]$ such that $E_N^{(n_m)}<V_{22,\omega}(N)<E_N^{(n_{m+1})}$, repeating the same argument in Case 1, we can find a root of $q(z)$ in the interval $(E_N^{(n_j)},E_N^{(n_{j+1})})$ for each $j=1,\dots,m-1,m+1,\dots,K-1$.

Note that
\begin{align*}
\prod_{\tiny\begin{matrix}\langle \boldsymbol{v}_j,\boldsymbol{\delta}_N\rangle \neq0\end{matrix}}(V_{22,\omega}(N)-E_N^{(j)})\sum_{\tiny\begin{matrix}\langle \boldsymbol{v}_j,\boldsymbol{\delta}_N\rangle \neq0\end{matrix}}\frac{|\langle \boldsymbol{v}_{j},\boldsymbol{\delta}_N\rangle |^2}{V_{22,\omega}(N)-E_N^{(j)}}=q(V_{22,\omega}(N))\neq 0
\end{align*}
since $V_{22,\omega}(N)\notin\sigma(\mathcal{S}_\omega|_{[1,N-1]})$ by Corollary \ref{11-23-corollary3.14} (i).

Then it holds that
\begin{align*}
\lim_{z\downarrow E_{N}^{(n_m)}}h(z)=\lim_{z\uparrow E_{N}^{(n_{m+1})}}h(z)=-\infty,
\end{align*}
and
\begin{align*}
\lim_{z\downarrow V_{22}(N)}h(z)=\pm \infty\text{ and }\lim_{z\uparrow V_{22}(N)}h(z)=\mp\infty.
\end{align*}
So there exists a root in $(E_N^{(n_m)},E_N^{(n_{m+1})})$, and the distribution of all the $K-1$ roots of $q(z)$ is clear.

(ii) If $V_{22}(N)<E_N^{(n_1)}$ or $E_N^{(n_K)}<V_{22}(N)$,  there exists a root of $q(z)$ in the interval $(E_N^{n_j},E_N^{(n_{j+1})})$ for each $j\in[1,K-1]$.
\end{proof}

Recall that the singular Jacobi cocycle on the strip takes the form 
\begin{align*}
    A^E(\omega)=\begin{pmatrix}
        E-V_{11}(\omega)-\frac{|V_{12}(\omega)|^2}{E-V_{22}(\omega)}&-1\\1&0
    \end{pmatrix}.
\end{align*}
thus if  $E\in\mathcal{T}_\delta$, then actually $A^E \in C^0(\Omega, SL(2,\R))$, thus the fibered rotation number of the cocycle $(\alpha, A^E)$ is well-defined, here, with the help of Interlacing Lemma (Lemma \ref{11-23-lemma3.15}),  we will try to establish the relationship between the integrated density of states and the fibered rotation number as follows:

\begin{theorem}\label{11-29-theorem5.21}Suppose $(\Omega,\nu,T)$ is uniquely ergodic. Denote $\rho(E)$ the fibered rotation number of $(T,A^E)$. Then, the following hold:
\begin{enumerate}
     \item  For $\sup V_{22}<E$, we have
$\mathrm{n}(E)=1-\rho(E).$

\item For $E<\inf V_{22}$, we have
$\mathrm{n}(E)=\frac{1}{2}-\rho(E).$
    \end{enumerate}
\end{theorem}

\begin{proof}
Denote $\Delta_{n,\omega}(E)=\frac{\det(E-\mathcal{S}_{\omega}|_{[1,n]})}{\prod_{j=1}^n(E-V_{22,\omega}(j))}$, and
\begin{align*}
    Z_{n,\omega}(E)=\begin{pmatrix}
        \Delta_n(E)\\\Delta_{n-1}(E)
    \end{pmatrix}, \qquad      Z_{n,\omega}(E)=A^E(T^{n-1}\omega)Z_{n-1,\omega}(E),
\end{align*}
It is clear that 
\begin{align*}
    Z_{n,\omega}(E)=A^E_n(\omega)\begin{pmatrix}
        1\\0
    \end{pmatrix}, \qquad  Z_{n-1,\omega+\alpha}(E)=A^E_{n-1}(\omega+\alpha)\begin{pmatrix}
        0\\-1
    \end{pmatrix}.
\end{align*}
Thus we have
\begin{align*}
    (Z_{n,\omega}(E),-Z_{n-1,\omega+\alpha}(E))=A_n^E(\omega)=\begin{pmatrix}\frac{\det(E-\mathcal{S}_{\omega}|_{[1,n]})}{\prod_{j=1}^n(E-V_{22,\omega}(j))}&-\frac{\det(E-\mathcal{S}_{\omega}|_{[2,n]})}{\prod_{j=2}^{n}(E-V_{22,\omega}(j))}\\ & \\\frac{\det(E-\mathcal{S}_{\omega}|_{[1,n-1]})}{\prod_{j=1}^{n-1}(E-V_{22,\omega}(j))}&-\frac{\det(E-\mathcal{S}_{\omega}|_{[2,n-1]})}{\prod_{j=2}^{n-1}(E-V_{22,\omega}(j))}\end{pmatrix}.
\end{align*}

Consider a projective cocycle $F_{A^z}$ on $\Omega\times\mathbb{S}^1$:
\begin{align*}(\omega,\psi)\mapsto\left(\omega+\alpha,\frac{A^z(\omega)\psi}{\|A^z(\omega)\psi\|}\right).
\end{align*}
It is possible to choose a continuous lift of $F_{A^z}$ to $\R$. Suppose $\xi_{n,\omega}(E)$ is a lift of $Z_{n,\omega}(E)$, then we have
\begin{itemize}
    \item $\lim_{n\to\infty}\frac{1}{n}\xi_{n,\omega}(E)=\rho(E).$
    \item If we denote by $\gamma_{E_1,E_2}:[E_1,E_2]\to\C$ the path $t\mapsto \zeta_n(E)$ (where $\zeta_n(E)$ represents the complex-plane point corresponding to $Z_{n,\omega}(E)$), then we have
    \begin{align*}
        \xi_{n,\omega}(E_2)-\xi_{n,\omega}(E_1)=\frac{1}{2\pi i}\int_{\gamma_{E_1,E_2}}\frac{dz}{z}=\frac{1}{2\pi i}\int_{E_1}^{E_2}\frac{\zeta_n'(E)}{\zeta_n(E)}dE.
    \end{align*}
\end{itemize}

Since IDS and the fibered rotation number are independent of the phase, it suffices to verify the following inequalities on the full-measure subset $\Omega_0$.
\begin{lemma}\label{11-29-lemma5.20}
  (1)  For $\sup V_{22}<E_1<E_2$ and $\omega\in\Omega_0$, we have
    \begin{align*}
        \left|2\xi_{n,\omega}(E_2)-2\xi_{n,\omega}(E_1)+\#\left\{\sigma(\mathcal{S}_{\omega}|_{[1,n]})\cap[E_1,E_2]\right\}\right|\le2,
    \end{align*}

(2) For $E_1<E_2<\inf V_{22}$ and $\omega\in\Omega_0$, we also have
    \begin{align*}
        \left|2\xi_{n,\omega}(E_2)-2\xi_{n,\omega}(E_1)+\#\left\{\sigma(\mathcal{S}_{\omega}|_{[1,n]})\cap[E_1,E_2]\right\}\right|\le2,
    \end{align*}
\end{lemma}
\begin{proof}
(1) According to Lemma \ref{11-23-lemma3.15}, it suffices to observe that if (for example) $$\max\{\sup V_{22},E_{n}^{(k)}\}<E_1<E_{n-1}^{(k-1)},\quad E_{n}^{(l)}<E_2<E_{n-1}^{(l-1)},$$ the integral $\xi_{n}(E_2)-\xi_n(E_1)=\int_{\gamma_{E_1,E_2}}dz/z$ equals $1/(2\pi i)$ multiplied by
\begin{align*}
    \int_{\gamma_{E_1,E_{n-1}^{(k-1)}}}(dz/z)+\int_{\gamma_{E_{n-1}^{(k-1)},E_{n}^{(k+1)}}}(dz/z)+\int_{\gamma_{E_{n-1}^{(k-1)},E_{n}^{(k+1)}}}(dz/z)+\cdots+\int_{\gamma_{E_{n}^{(l)},E_{2}}}(dz/z).
\end{align*}
Since each term of this sum, except for the first and the last, equals $-\pi/2$ (note that if $\Delta_n(E)>0$ and $\Delta_{n-1}(E)=0$, then according to Lemma \ref{11-23-lemma3.15}, $\Delta_n(E+)>0$ and $\Delta_{n-1}(E+)<0$) and the extreme terms are smaller than $\pi/2$, we observe that this integral equals $-(l-k)\pi$ up to an approximation of $\pi$. It is easy to see that this gives the conclusion of the first part. 

(2) Similarly, it suffices to observe that if (for example) $$E_{n}^{(k)}<E_1<E_{n-1}^{(k)},\quad E_{n}^{(l)}<E_2<\min\{\inf V_{22},E_{n-1}^{(l)}\},$$ the integral $\xi_{n}(E_2)-\xi_n(E_1)=\int_{\gamma_{E_1,E_2}}dz/z$ equals $1/(2\pi i)$ multiplied by
\begin{align*}
    \int_{\gamma_{E_1,E_{n-1}^{(k)}}}(dz/z)+\int_{\gamma_{E_{n-1}^{(k)},E_{n}^{(k+1)}}}(dz/z)+\int_{\gamma_{E_{n}^{(k+1)},E_{n-1}^{(k+1)}}}(dz/z)+\cdots+\int_{\gamma_{E_{n}^{(l)},E_{2}}}(dz/z).
\end{align*}
Since each term of this sum, except for the first and the last, equals $-\pi/2$ (note that if $\Delta_n(E)>0$ and $\Delta_{n-1}(E)=0$, then according to Lemma \ref{11-23-lemma3.15}, we also have $\Delta_n(E+)>0$ and $\Delta_{n-1}(E+)<0$) and the extreme terms are smaller than $\pi/2$, we observe that this integral equals $-(l-k)\pi$ up to an approximation of $\pi$. It is easy to see that this gives the conclusion of the second part.
\end{proof}

Applying Lemma \ref{11-29-lemma5.20}, we have
\begin{align*}
    \lim_{n\to\infty}\left(\frac{1}{n}\xi_{n,\omega}(E_2)-\frac{1}{n}\xi_{n,\omega}(E_1)\right)+\int_{[E_1,E_2]}d\mathrm{n}_{\omega,n}=0
\end{align*}
for $E_1<E_2<\inf V_{22}$ or $\sup V_{22}<E_1<E_2$.
Using the definition of $\rho(E)$ and $\mathrm{n}$, one has
\begin{align*}
  \rho(E_2)-\rho(E_1)+\mathrm{n}(E_2)-\mathrm{n}(E_1)=0
\end{align*}
for $E_1<E_2<\inf V_{22}$ or $\sup V_{22}<E_1<E_2$.

(1) If $E_1<E_2<\inf V_{22}$, letting $E_1=-\infty$ and using $\rho(-\infty)=\frac{1}{2}$, one has $\mathrm{n}(E_2)=\frac{1}{2}-\rho(E_2).$

(2) If $\sup V_{22}<E_1<E_2$, letting $E_2=\infty$ and using $\rho(\infty)=0$, one has $\mathrm{n}(E_1)=1-\rho(E_1).$
\end{proof}

\section{Subordinacy Theory}	
\subsection{Cyclic vectors and Green's function}
For $2\times2$ non-singular matrix-valued Jacobi operators, $\ell^2(\Z,\C^2)$ can be divided into the direct sum of four cyclic subspaces. However, for singular matrix-valued Jacobi operators, such cyclic subspaces are unclear and may be even infinite. In this section, we study the spectrum outside $\operatorname{Ran}(V_{22})$. More precisely, if we  denote $\boldsymbol{\delta}_n=\begin{pmatrix}1\\0\end{pmatrix}\delta_n$ and $\boldsymbol{\gamma}_n=\begin{pmatrix}0\\1\end{pmatrix}\delta_n$, then
we have the following proposition:
\begin{proposition}\label{pro1}
Denote 
\begin{align*}
    \mathcal{T}_\delta=\{x\in\R:\operatorname{dist}(x,\operatorname{Ran}(V_{22}))\ge\delta\}
\end{align*}
with $\delta>0$. Then, for any compactly supported $\boldsymbol{\phi}\in \ell^2(\Z,\C^2)$, we have
\begin{align*}
\chi_{\mathcal{T}_\delta}(\mathcal{S})\boldsymbol{\phi}=\Big(\prod_{j=k}^{l}f_n(\mathcal{S})\Big)\Big[p(\mathcal{S})\chi_{\mathcal{T}_\delta}(\mathcal{S})\boldsymbol{\delta}_0+q(\mathcal{S})\chi_{\mathcal{T}_\delta}(\mathcal{S})\boldsymbol{\delta}_1\Big]
\end{align*}
 with suitable real polynomials $p,q$ and $k,l\in\N$. Moreover, $\|f_n(\mathcal{S})\|\le\frac{1}{\delta},\forall n\in\Z$.
\end{proposition}

\begin{proof}
 Notice first that 
$$[\mathcal{S}-V_{22}(n)]\boldsymbol{\gamma}_n=V_{12}(n)\boldsymbol{\delta}_n,$$
and
$$\mathcal{S}\boldsymbol{\delta}_n=\boldsymbol{\delta}_{n-1}+\boldsymbol{\delta}_{n+1}+V_{11}(n)\boldsymbol{\delta}_n+V_{21}(n)\boldsymbol{\gamma}_n.$$

Since $\operatorname{dist}(\mathcal{T}_\delta,\operatorname{Ran}(V_{22}))\ge\delta$, we have
\begin{align*}
    \chi_{\mathcal{T}_\delta}(x)=\chi_{\mathcal{T}_\delta}(x)-\chi_{\mathcal{T}_\delta}(V_{22}(n))=(x-V_{22}(n))f_n(x),
\end{align*}
where
\begin{align*}
    f_n(x)=\left\{\begin{array}{ll}
        \frac{1}{x-V_{22}(n)} & \forall x\in\mathcal{T}_\delta, \\
        0 & \text{else}.
    \end{array}\right.
\end{align*}
Obviously, $|f_n(x)|\le\frac{1}{\delta}$ and $f_n(x)\chi_{\mathcal{T}_\delta}(x)=f_n(x)$. Since the function $f_n(x)$ is bounded and can be approximated by polynomials, $f_n(\mathcal{S})$ is well-defined, which implies
\begin{align*}
    \chi_{\mathcal{T}_\delta}(\mathcal{S})\boldsymbol{\gamma}_n=V_{12}(n)f_n(\mathcal{S})\chi_{\mathcal{T}_\delta}(\mathcal{S})\boldsymbol{\delta}_n,
\end{align*}
and
\small
\begin{align*}
    &\chi_{\mathcal{T}_\delta}(\mathcal{S})\boldsymbol{\delta}_{n+1}\\&=(\mathcal{S}-V_{11}(n)-|V_{12}(n)|^2f_n(\mathcal{S}))\chi_{\mathcal{T}_\delta}(\mathcal{S})\boldsymbol{\delta}_n-\chi_{\mathcal{T}_\delta}(\mathcal{S})\boldsymbol{\delta}_{n-1}\\&=f_n(\mathcal{S})\left\{\left[(\mathcal{S}-V_{11}(n))(\mathcal{S}-V_{22}(n))-|V_{12}(n)|^2\right]\chi_{\mathcal{T}_\delta}(\mathcal{S})\boldsymbol{\delta}_n-(\mathcal{S}-V_{22}(n))\chi_{\mathcal{T}_\delta}(\mathcal{S})\boldsymbol{\delta}_{n-1}\right\}.
\end{align*}
\normalsize
By induction, there are polynomials $p$ and $q$ with no more than $2n-2$ degrees so that $\boldsymbol{\delta}_n=\left(\prod_{j=1}^{n-1}f_n(\mathcal{S})\right)[p(\mathcal{S})\boldsymbol{\delta}_0+q(\mathcal{S})\boldsymbol{\delta}_1]$ for all $n>1$. It is similar for $n<0$. Taking finite linear combinations, the proposition follows.
\end{proof}

Hence, the sum of the spectral measures associated with $\boldsymbol{\delta}_0$ and $\boldsymbol{\delta}_1$ can serve as a canonical spectral measure on $\mathcal{T}_\delta$ for $\mathcal{S}$, i.e., $\mu=\mu_{\boldsymbol{\delta}_0}+\mu_{\boldsymbol{\delta}_1}$. Then, $\mu_{\boldsymbol{\phi}}|_{\mathcal{T}_\delta}$ is absolutely continuous with respect to $\mu|_{\mathcal{T}_\delta}$ for any $\boldsymbol{\phi}\in\ell^2(\Z,\C^2)$, where $\mu|_{\mathcal{T}_\delta}=\mu(\cdot\cap\mathcal{T}_\delta)$. 
Meanwhile, we have the following elementary observation:

\begin{lemma}
Suppose $z\notin\{V_{22}(j):j\in\Z\}$, and $\boldsymbol{u},\boldsymbol{v}$ solve $\mathcal{S}\boldsymbol{x}=z\boldsymbol{x}$. For any pair of solutions $\boldsymbol{u},\boldsymbol{v}$, the Wronskian
$$W_{[\boldsymbol{u},\boldsymbol{v}]}(n):=\langle \bar{\boldsymbol{v}}(n-1),J\boldsymbol{u}(n)\rangle _{\C^2}-\langle \bar{\boldsymbol{u}}(n-1),J\boldsymbol{v}(n)\rangle _{\C^2},$$
is independent of $n,z$. Moreover, $W_{[\boldsymbol{u},\boldsymbol{v}]}$ vanishes iff $\boldsymbol{u},\boldsymbol{v}$ are linearly dependent.
\end{lemma}

\begin{proof}
It is a direct result by \eqref{eq6} and \eqref{eq5}. 
\end{proof}	

Once we have this, we can introduce Green's function. 
For $\Im z>0$, $R(\mathcal{S},z)=(\mathcal{S}-zI)^{-1}$ exists and is a bounded operator, then it is easy to check that there exists unique $\boldsymbol{u}^+_{z}$ obeying 
\begin{align*}
J\boldsymbol{u}_{z}^+(n+1)+J\boldsymbol{u}_{z}^+(n-1)+V(n)\boldsymbol{u}_{z}^+(n)&=z\boldsymbol{u}_{z}^+(n),\\ \boldsymbol{u}_{z}^+(0)= \left(a_{z}^{+}(0), b_{z}^{+}(0)\right)^{\mathsf{T}}&=
   \left(1,\frac{V_{21}(0)}{z-V_{22}(0)}\right)^{\mathsf{T}},\\\sum_{n=0}^\infty\|\boldsymbol{u}_{z}^+(n)\|^2&<\infty.
\end{align*}
Similarly, there exists unique $\boldsymbol{u}_{z}^-(n)$ obeying
\begin{align*}
J\boldsymbol{u}_{z}^-(n+1)+J\boldsymbol{u}_{z}^-(n-1)+V(n)\boldsymbol{u}_{z}^-(n)&=z\boldsymbol{u}_{z}^-(n),\\ \boldsymbol{u}_{z}^-(0)=\left(a_{z}^{-}(0), b_{z}^{-}(0)\right)^{\mathsf{T}}&=\left(1,\frac{V_{21}(0)}{z-V_{22}(0)}\right)^{\mathsf{T}},\\\sum_{n=-\infty}^0\|\boldsymbol{u}_{z}^-(n)\|^2&<\infty.
\end{align*}

The Green's function of $\mathcal{S}$ is denoted by
$$G(n,m;z):=\langle \boldsymbol{\delta}_n,(\mathcal{S}-z)^{-1}\boldsymbol{\delta}_m\rangle ,$$
where $n,m\in\mathbb{Z},z\notin\sigma(\mathcal{S})$. Note that, in general, when considering the Green's function, one should also consider the matrix elements 
$\langle \boldsymbol{\delta}_n,(\mathcal{S}-z)^{-1}\boldsymbol{\gamma}_m\rangle $, $\langle \boldsymbol{\gamma}_n,(\mathcal{S}-z)^{-1}\boldsymbol{\delta}_m\rangle $ and $\langle \boldsymbol{\gamma}_n,(\mathcal{S}-z)^{-1}\boldsymbol{\gamma}_m\rangle $. However, by Proposition \ref{pro1}, $\chi_{\mathcal{T}_\delta}(\mathcal{S})\boldsymbol{\delta}_0,\chi_{\mathcal{T}_\delta}(\mathcal{S})\boldsymbol{\delta}_1$ form cyclic vectors in the subspace $\chi_{\mathcal{T}_\delta}(\mathcal{S})\ell^2(\Z,\C^2)$. Consequently, to study the spectrum of $\mathcal{S}$ in $\mathcal{T}_\delta$, it suffices to analyze the Green's function elements $\langle \boldsymbol{\delta}_n,(\mathcal{S}-z)^{-1}\boldsymbol{\delta}_m\rangle $. Indeed, direct computations show the following:

\begin{lemma}\label{pro3.4}
Let $\Im z>0$, and $\boldsymbol{u}_{z}^{\pm}$ be as above. Then we have
\begin{align*}
G(n,m;z)=\frac{a_{z}^-(n\land m)a_{z}^+(n\lor m)}{W_{[\boldsymbol{u}_{z}^+,\boldsymbol{u}_{z}^-]}},
\end{align*}
where $n\land m=\min(n,m),n\lor m=\max(n,m)$, and $W_{[\boldsymbol{u}_{z}^+,\boldsymbol{u}_{z}^-]}$ denotes the Wronskian.
\end{lemma}

\begin{proof}
Due to the assumption of  $z$, $\boldsymbol{u}_{z}^-$ and $\boldsymbol{u}_{z}^+$ must be linearly independent, so we may assume without loss of generality that 
$$W_{[\boldsymbol{u}_{z}^+,\boldsymbol{u}_{z}^-]}(1)=a_{z}^+(1)a_{z}^-(0)-a_{z}^-(1)a_{z}^+(0)=1.$$
Define 
$$\boldsymbol{\phi}_m(n):=a_{z}^-(n\land m)a_{z}^+(n\lor m)\begin{pmatrix}1\\\frac{V_{21}(n)}{z-V_{22}(n)}\end{pmatrix}.$$ 
Since $\boldsymbol{u}_{z}^\pm$ are square-summable at $\pm\infty$, it follows that $\boldsymbol{\phi}_m\in \ell^2(\Z,\C^2)$. If $n>m$, we have
$[\mathcal{S}\boldsymbol{\phi}_m](n)=z\boldsymbol{\phi}_m(n).$
Similarly, we also have $[\mathcal{S}\boldsymbol{\phi}_m](n)=z\boldsymbol{\phi}_m(n)$ whenever $n<m$. Finally, we observe that
$$[S\boldsymbol{\phi}_m](m)=z\boldsymbol{\phi}_m(m)+\begin{pmatrix}1\\0\end{pmatrix}.$$
Thus, we have that $(\mathcal{S}-z)\boldsymbol{\phi}_m=\boldsymbol{\delta}_m$, so $\boldsymbol{\phi}_m=(\mathcal{S}-z)^{-1}\boldsymbol{\delta}_m$.
Therefore,
\begin{align*}
[(\mathcal{S}-z)^{-1}\boldsymbol{\delta}_m](n)=\frac{a_z^-(n\land m)a_z^+(n\lor m)}{W_{[\boldsymbol{u}_{z}^+,\boldsymbol{u}_{z}^-]}}\begin{pmatrix}1\\ \frac{V_{21}(n)}{z-V_{22}(n)}\end{pmatrix}.
\end{align*}
By direct calculation, the result follows.
\end{proof}

\begin{corollary}\label{mp}
Under the same assumptions above, we have that
\begin{equation*}\begin{split}&\langle \boldsymbol{\delta}_0,(\mathcal{S}-z)^{-1}\boldsymbol{\delta}_0\rangle =\frac{m^-(z)}{1-m^+(z)m^-(z)},\\&\langle \boldsymbol{\delta}_1,(\mathcal{S}-z)^{-1}\boldsymbol{\delta}_1\rangle =\frac{m^+(z)}{1-m^+(z)m^-(z)},\end{split}\end{equation*}
where 
$$m^+(z):=-\frac{a_{z}^+(1)}{a_{z}^+(0)}\text{ and }m^-(z):=-\frac{a_{z}^-(0)}{a_{z}^-(1)}.$$
Moreover, the Borel transform of $\mu$ takes the form
\begin{equation*}
\begin{split}
M(z)&:=F_{\mu}(z)=\int\frac{d\mu(E)}{E-z}=\langle \boldsymbol{\delta}_0,(\mathcal{S}-z)^{-1}\boldsymbol{\delta}_0\rangle +\langle \boldsymbol{\delta}_1,(\mathcal{S}-z)^{-1}\boldsymbol{\delta}_1\rangle \\&=\frac{m^+(z)+m^-(z)}{1-m^+(z)m^-(z)}.\end{split}
\end{equation*}
\end{corollary}
\begin{proof}
By the equation \eqref{3.4-eq7}, $a_{z}^+(n)$ can not vanish and it is the same for $a_{z}^-(n)$. Applying Lemma \ref{pro3.4}, we obtain the results.
\end{proof}

In the following, we aim to establish the subordinacy theory for a class of singular matrix-valued Jacobi operator under the assumption in Proposition \ref{pro1}. 
Recall the following definition:

\begin{definition}
A solution $\boldsymbol{u}$ of $\mathcal{S}\boldsymbol{u}=z\boldsymbol{u}$ is called subordinate at $+\infty$ if
\begin{align*}
\lim_{L\to\infty}\frac{\|\boldsymbol{u}\|_L}{\|\boldsymbol{v}\|_L}=0
\end{align*}
for any linearly independent solution $\boldsymbol{v}$, where $\|\cdot\|_L$ defined as follows.
\begin{align*}
\|\boldsymbol{u}\|_L:=\left(\sum_{n=1}^{\lfloor L\rfloor}|\boldsymbol{u}(n)|^2+(L-\lfloor L\rfloor)|\boldsymbol{u}(\lfloor L\rfloor+1)|^2\right)^{1/2}.
\end{align*}
Obviously, $\frac{1}{{\sqrt{2}}}(\|a\|_L+\|b\|_L)\le\|\boldsymbol{u}\|_L\le\|a\|_L+\|b\|_L$.
\end{definition}

\subsection{Jitomirskaya-Last inequality}\label{jito-last}

To establish the subordinacy theory, the key understanding is Jitomirskaya-Last inequality \cite{jitomirskaya1999}, which relates $m-$function with 
subordinate solution. By Corollary \ref{mp}, this will enable us to relate the spectral measure with  the growth of the cocycle. In this subsection, we first establish  Jitomirskaya-Last inequality. 

To begin with, as pointed out previously, for $\Im z>0$, $\mathcal{S}\boldsymbol{u}=z\boldsymbol{u}$ has a (up to normlization) unique solution $\boldsymbol{u}_z^+$ that is summable at $+\infty$. We normalize $\boldsymbol{u}_z^+$ by letting $a_z^+(0)=1$, and hence $a_z^+(1)=-m^+(z)$.

Suppose $\boldsymbol{u}\in \ell^2(\Z,\C^2)$ solves the difference equation \eqref{sjo}. Multiplying both sides by $\overline{\boldsymbol{u}(n)^{\mathsf{T}}}$, taking imaginary parts, and summing both sides from 1 to $\infty$, we obtain the equality
\begin{equation}\label{3.4-eq7}
\Im(\overline{a(1)}a(0))=\Im z\sum_{n=1}^\infty|a(n)|^2+|b(n)|^2,
\end{equation}
which for $\boldsymbol{u}=\boldsymbol{u}_z^+$ implies
\begin{equation}\label{eq32}
\frac{\Im m^+(z)}{\Im z}=\sum_{n=1}^\infty|u_z^+(n)|^2.
\end{equation}

Denote $z=E+i\epsilon$ and suppose $E\in\mathcal{T}_\delta$. Let $\boldsymbol{u}_1$ and $\boldsymbol{u}_2$ solve $\mathcal{S}\boldsymbol{u}=E\boldsymbol{u}$ and obey the initial conditions
\begin{equation*}
\begin{pmatrix}a_1(1)&a_2(1)\\a_1(0)&a_2(0)\end{pmatrix}=\begin{pmatrix}1&0\\0&1\end{pmatrix}
\end{equation*}
and
\begin{align*}
\begin{pmatrix}b_1(1)&b_2(1)\\b_1(0)&b_2(0)\end{pmatrix}=\begin{pmatrix}\frac{V_{21}(1)}{E-V_{22}(1)}&0\\0&\frac{V_{21}(0)}{E-V_{22}(0)}\end{pmatrix}.
\end{align*}

\begin{lemma}
Suppose  that $E\in\mathcal{T}_\delta$. For any $n\in\N$, $\boldsymbol{u}_z^+$ obeys the equality
\begin{equation}\label{3.8-eq8}
\begin{split}
a_z^+(n)=a_2(n)-m^+(z)a_1(n)-&i\epsilon a_2(n) \left(\sum_{k=1}^na_1(k)a_z^+(k)+\frac{V_{12}(k)}{E-V_{22}(k)}a_1(k)b_z^+(k)\right)\\+&i\epsilon a_1(n)\left(\sum_{k=1}^na_2(k)a_z^+(k)+\frac{V_{12}(k)}{E-V_{22}(k)}a_2(k)b_z^+(k)\right).
\end{split}
\end{equation}
Adopting the convention in which the empty sum is zero, \eqref{3.8-eq8} also holds true for $n=0$. 
Moreover, it holds that
\begin{align}\label{3.8-eq9}
\nonumber2\|a_z^+\|_L+2\|b_z^+\|_L\ge&\|a_2-m^+(z)a_1\|_L+\|b_2-m^+(z)b_1\|_L\\&-2\epsilon(\|a_1\|_L+\|b_1\|_L)(\|a_2\|_L+\|b_2\|_L)(\|a_z^+\|_L+\|b_z^+\|_L)
\end{align}
for any $|\Im z|<\delta$ and $L>1$.
\end{lemma}

\begin{proof}
Let $\widetilde v(n)$ be the right-hand side of $\eqref{3.8-eq8}$ for $n\in\N$ and let $\widetilde v(0)=1$. On one hand,
$$a_1(n+1)a_2(n)-a_2(n+1)a_1(n)=1$$
for every $n\in\Z$. By conservation of the Wronskian, so one can verify that $(\widetilde v(n))_{n=0}^\infty$ obeys
\begin{align*}
\widetilde v(n+1)=-\widetilde v(n-1)&+\left(E-V_{11}(n)-\frac{|V_{12}(n)|^2}{E-V_{22}(n)}\right)\widetilde v(n)\\&+i\epsilon \left(a_z^+(n)+\frac{V_{12}(n)}{E-V_{22}(n)}b_z^+(n)\right)
\end{align*}
for every $n>0$.
 On the other hand,
 \begin{equation*}
 \begin{split}
a_z^+(n+1)=-a_z^+(n-1)&+\left(E-V_{11}(n)-\frac{|V_{12}(n)|^2}{E-V_{22}(n)}\right)a_z^+(n)\\&+i\epsilon\left(a_z^+(n)+\frac{V_{12}(n)}{E-V_{22}(n)}b_z^+(n)\right)
\end{split}
\end{equation*}
for all $n>0$. Since $\widetilde v(0)=1=a_z^+(0),\widetilde v(1)=-m^+(z)=a_z^+(1)$, we see that $\widetilde v(n)=a_z^+(n)$ for all $n\ge0$ by induction.

Since $|V_{12}(k)|=|V_{21}(k)|$, we have
\begin{align}\label{3.8-eq10}
\nonumber\|a_z^+\|_L\ge\|a_2-m^+(z)a_1\|_L-\epsilon(&2\|a_1\|_L\|a_2\|_L\|a_z^+\|_L+\|b_1\|_L\|a_2\|_L\|b_z^+\|_L\\&+\|a_1\|_L\|b_2\|_L\|b_z^+\|_L).
\end{align}

By equality \eqref{3.8-eq8}, we have
\small
\begin{align*}
    \frac{z-V_{22}(n)}{E-V_{22}(n)}b_z^+(n)=b_2(n)-m^+(z)b_1(n)-&i\epsilon b_2(n) \left(\sum_{k=1}^na_1(k)a_z^+(k)+\frac{V_{12}(k)}{E-V_{22}(k)}a_1(k)b_z^+(k)\right)\\+&i\epsilon b_1(n)\left(\sum_{k=1}^na_2(k)a_z^+(k)+\frac{V_{12}(k)}{E-V_{22}(k)}a_2(k)b_z^+(k)\right).
\end{align*}
\normalsize
which implies that
\begin{align}\label{3.8-eq11}
\nonumber2\|b_z^+\|_L\ge\|b_2-m^+(z)b_1\|_L-\epsilon(&2\|b_1\|_L\|b_2\|_L\|b_z^+\|_L+\|a_1\|_L\|b_2\|_L\|a_z^+\|_L\\&+\|b_1\|_L\|a_2\|_L\|a_z^+\|_L)
\end{align}
for $|\Im z|<\delta$.
Combining \eqref{3.8-eq10} and \eqref{3.8-eq11}, we obtain \eqref{3.8-eq9}.
\end{proof}	
	
Now that we have related $\boldsymbol{u}_1$ and $\boldsymbol{u}_2$ to $\boldsymbol{u}_z^+$, the next step is to choose the right length scale $L$. Indeed, it is easy to check that $P(L)=\|\boldsymbol{u}_1\|_L\|\boldsymbol{u}_2\|_L$
is a strictly increasing continuous function for which $P(1)=0$ and $P(L)\to\infty$ as $L\to\infty$.
Thus for any $\epsilon>0$, define $L_\epsilon\in(1,\infty)$ by
\begin{equation}\label{u12}
    \|\boldsymbol{u}_1\|_L\|\boldsymbol{u}_2\|_L=\frac{1}{4\epsilon}
\end{equation}
then $L_\epsilon$ is well-defined. Then we state the Jitomirskaya-Last's inequality for singular Jacobi operator on the strip:

\begin{theorem}\label{j-l-equ}
Suppose $E\in\mathcal{T}_\delta$. The following inequality holds for all $0<\epsilon<\delta$:
\begin{equation*}
\frac{C_1}{|m^+(E+i\epsilon)|}<\frac{\|\boldsymbol{u}_1\|_{L_\epsilon}}{\|\boldsymbol{u}_2\|_{L_\epsilon}}<\frac{C_2}{|m^+(E+i\epsilon)|},
\end{equation*}
where $0<C_1<C_2$ are constants.
\end{theorem}

\begin{proof}
Write $z=E+i\epsilon$. From \eqref{3.8-eq9} we see that for any $L>1$,
\begin{align}\label{3.10-eq12}
\nonumber2\sqrt{2}\|\boldsymbol{u}_z^+\|_L\ge&\|a_2-m^+(z)a_1\|_{L}+\|b_2-m^+(z)b_1\|_{L}\\&-4\sqrt{2}\epsilon\|\boldsymbol{u}_1\|_L\|\boldsymbol{u}_2\|_L\|\boldsymbol{u}_z^+\|_L.
\end{align}
By considering $L=L_\epsilon$, i.e. \eqref{u12} holds, then 
substitute it into \eqref{3.10-eq12}, we obtain
\begin{equation*}
4\sqrt{2}\|\boldsymbol{u}\|_{L_\epsilon}\ge\|a_2-m^+(z)a_1\|_{L_\epsilon}+\|b_2-m^+(z)b_1\|_{L_\epsilon}.
\end{equation*}
Squaring the two sides and noting that by \eqref{eq32}
$$\|\boldsymbol{u}\|_{L_\epsilon}^2=\|a_z\|_{L_\epsilon}^2+\|b_z\|_{L_\epsilon}^2\le\sum_{n=1}^\infty|a_z^+(n)|^2+|b_z^+(n)|^2=\frac{\Im m^+(z)}{\epsilon}.$$
Thus, we obtain
\begin{align*}
\frac{\Im m^+(z)}{\epsilon}&\ge\frac{1}{32}(\|a_2-m^+(z)a_1\|_{L_\epsilon}^2+\|b_2-m^+(z)b_1\|_{L_\epsilon}^2)\\&\ge\frac{1}{32}(\|\boldsymbol{u}_2\|^2_{L_\epsilon}+|m^+(z)|^2\|\boldsymbol{u}_1\|_{L_\epsilon}^2-4|m^+(z)|\|\boldsymbol{u}_1\|_{L_\epsilon}\|\boldsymbol{u}_2\|_{L_\epsilon})\\&=\frac{1}{128\epsilon}(\frac{\|\boldsymbol{u}_2\|_{L_\epsilon}}{\|\boldsymbol{u}_1\|_{L_\epsilon}}+|m^+(z)|\frac{\|\boldsymbol{u}_1\|_{L_\epsilon}}{\|\boldsymbol{u}_2\|_{L_\epsilon}}-4|m^+(z)|)
\end{align*}
Consequently, it holds that
\begin{equation}\label{eq35}
132|m^+(z)|>\frac{\|\boldsymbol{u}_2\|_{L_\epsilon}}{\|\boldsymbol{u}_1\|_{L_\epsilon}}+|m^+(z)|^2\frac{\|\boldsymbol{u}_1\|_{L_\epsilon}}{\|\boldsymbol{u}_2\|_{L_\epsilon}}.
\end{equation}
Solving \eqref{eq35} as a quadratic inequality for $|m^+(z)|$, one obtains
$$(66-\sqrt{66^2-1})\frac{\|\boldsymbol{u}_2\|_{L_\epsilon}}{\|\boldsymbol{u}_1\|_{L_\epsilon}}<|m^+(z)|<(66+\sqrt{66^2+1})\frac{\|\boldsymbol{u}_2\|_{L_\epsilon}}{\|\boldsymbol{u}_1\|_{L_\epsilon}}.$$
\end{proof}

As direct consequence, we have the following key result of subordinacy theory:
\begin{corollary}\label{subordinnacy}
    Let $\mathcal{B}_{\mathcal{T}_\delta}$ be the set of $E\in\mathcal{T}_\delta$ such that the cocycle $(T,A^E)$ is bounded. Then $\mu_\omega|_{\mathcal{B}_{\mathcal{T}_\delta}}$ is absolutely continuous for all $\omega\in\Omega$.
\end{corollary}

Moreover, we have the following explicit estimate:

\begin{lemma}\label{growth-cocycle}Suppose  that $E\in\mathcal{T}_\delta$ and $0<\epsilon<\delta$.
    There exists universal constant $C>0$ such that $\mu_\omega(E-\epsilon,E+\epsilon)\le C\epsilon\sup_{0\le s\le\epsilon^{-1}}\|A^E_s\|_0^2$.
\end{lemma}

\begin{proof}
It's immediate consequence of Jitomirskaya-Last inequality (Theorem \ref{j-l-equ}) and Corollary \ref{mp}, one can consult \cite[Lemma 2.5]{A01} for details. 
\end{proof}

\section{New exact mobility edges of critical states and extended states}\label{section8}

In this section, we demonstrate $\mathcal{S}_{J,V,\omega}$ defined as \eqref{ac-sc} has exact Type III ME. 
\subsection{Singular continuous spectrum of $\mathcal{S}$}

\begin{theorem}\label{11-28-theorem8.2}
    For any irrational $\alpha$ and for any non-$2\alpha$-rational $\omega \in\T$ the operator $\mathcal{S}_{J,V,\omega}$ exhibits  purely singular continuous spectrum in $\sigma^{\mathcal{S}}(J,V)\cap(-\lambda,\lambda)$.
\end{theorem}
\begin{proof}
We begin by demonstrating that $\mathcal{S}_{J,V,\omega}$ has  no absolutely continuous spectrum within $\sigma^{\mathcal{S}}(J,V)\cap(-\lambda,\lambda)$ for any $\omega\in\T$. 
Note that $\sigma^{\mathcal{S}}(J,V)$ coincides with $\sigma^{\mathcal{H}}(c,v)$ by Aubry duality. Let $\nu_\omega^{(ac)}$ denote the absolutely continuous component of any spectral measure $\nu_\omega$ associated with $\mathcal{S}_{J,V,\omega}$.   
By Corollary \ref{corollary6.2}, we have $\nu_\omega^{(ac)}(\sigma^{\mathcal{S}}(J,V)\cap[-\lambda,\lambda])=0$ for all $\omega\in\T$, thereby establishing the claim.

In the remaining part, we aim to demonstrate the absence of a point spectrum for all non-$2\alpha$-rational values of $\omega$. Otherwise, it would imply the existence of a vector $\boldsymbol{u} \in \ell^2(\mathbb{Z}, \mathbb{C}^2)$ associated with an eigenvalue $E \in (-\lambda, \lambda)$ that satisfies the equation $\mathcal{S}_{J,V,\omega}\boldsymbol{u} = E\boldsymbol{u}$, where $\boldsymbol{u}(n) = (a(n), b(n))^{\mathsf{T}}$. This equation can be expressed as:
\begin{align}\label{10.15-eq48}
J\boldsymbol{u}(n+1) + J\boldsymbol{u}(n-1) + V(\omega + 2n\alpha)\boldsymbol{u}(n) = E\boldsymbol{u}(n).
\end{align}
First, we note that $a \not\equiv 0$; otherwise, $b \equiv 0$ would follow, since $V_{12,\omega}(n) = i \sin(2\pi(\omega + 2n\alpha)) \neq 0$ for all non-$2\alpha$-rational values of $\omega \in \mathbb{T}$ and $n \in \mathbb{Z}$. Let us define $\Psi(x) := (f(x), g(x))^{\mathsf{T}} = \sum_{n \in \mathbb{Z}} \boldsymbol{u}(n)e^{2\pi inx}$, ensuring that $f \not\equiv 0$.
Proceeding as before, we multiply equation \eqref{10.15-eq48} by $e^{2\pi inx}$ and sum over $n$, yielding
\footnotesize
\begin{align}
   \label{10-18-eq49} \frac{\lambda e^{2\pi i\omega}}{2}(f(x+2\alpha)+g(x+2\alpha))+\frac{\lambda e^{-2\pi i\omega}}{2}(f(x-2\alpha)-g(x-2\alpha))+2\cos2\pi x\,f(x)&=Ef(x)\\
    -\frac{\lambda e^{2\pi i\omega}}{2}(f(x+2\alpha)+g(x+2\alpha))+\frac{\lambda e^{-2\pi i\omega}}{2}(f(x-2\alpha)-g(x-2\alpha))&=Eg(x).\nonumber
\end{align}
\normalsize
From this, we have the following equations
\begin{align*}
    \lambda e^{2\pi i\omega}(f(x+2\alpha)+g(x+2\alpha))+2\cos2\pi x\, f(x)&=E(f(x)-g(x))\\
    \lambda e^{-2\pi i\omega}(f(x-2\alpha)-g(x-2\alpha))+2\cos2\pi x\, f(x)&=E(f(x)+g(x)),
\end{align*}
and then
\begin{align*}
    E(f(x-2\alpha)-g(x-2\alpha))&=\lambda e^{2\pi i\omega}(f(x)+g(x))+2\cos2\pi(x-2\alpha)\,f(x-2\alpha)\\ E(f(x+2\alpha)+g(x+2\alpha))&=\lambda e^{-2\pi i\omega}(f(x)-g(x))+2\cos2\pi(x+2\alpha)\,f(x+2\alpha).
\end{align*}
Substituting these expressions into equation \eqref{10-18-eq49}, we obtain
\footnotesize
\begin{align*}
    \lambda e^{2\pi i\omega}\cos2\pi(x+2\alpha)\,f(x+2\alpha)+\lambda e^{-2\pi i\omega}\cos2\pi(x-2\alpha)\,f(x-2\alpha)+(\lambda^2-E^2+2E\cos2\pi x)f(x)=0.
\end{align*}
\normalsize
Multiplying equation \eqref{10.15-eq48} by $e^{-2\pi inx}$ and repeating the same derivation, we get
\footnotesize
\begin{align*}
    \lambda e^{2\pi i\omega}\cos2\pi(x-2\alpha)\,f(2\alpha-x)+\lambda e^{-2\pi i\omega}\cos2\pi(x+2\alpha)\,f(-x-2\alpha)+(\lambda^2-E^2+2E\cos2\pi x)f(-x)=0.
\end{align*}
\normalsize
Thus, we obtain the following matrix equation
\begin{align*}
    &\begin{pmatrix}
        \frac{E^2-\lambda^2-2E\cos2\pi x}{\lambda\cos2\pi(x+2\alpha)}&-\frac{\cos2\pi(x-2\alpha)}{\cos2\pi(x+2\alpha)}\\1&0
    \end{pmatrix}\begin{pmatrix}
        f(x)&f(-x)\\ e^{-2\pi i\omega}f(x-2\alpha)&e^{2\pi i\omega}f(-x+2\alpha)
    \end{pmatrix}\\&=\begin{pmatrix}
        f(x+2\alpha)&f(-x-2\alpha)\\ e^{-2\pi i\omega}f(x)&e^{2\pi i\omega}f(-x)
    \end{pmatrix}\begin{pmatrix}
        e^{2\pi i\omega}&0\\0&e^{-2\pi i\omega}
    \end{pmatrix}.
\end{align*}
Applying a similar argument as in the proof of Theorem \ref{theorem2.2}, we arrive at a contradiction. 
\end{proof}

\subsection{Purely absolutely continuous spectrum of $\mathcal{S}$}

\begin{theorem}\label{8.4-theorem4.4}
Let $\alpha\in DC$, $V\in C^{\omega}(\T, M(2,\C))$ being hermitian, then for any $\omega\in\T$, the operator $\mathcal{S}_{J,V,\omega}$ defined by 
\begin{align*}
    [\mathcal{S}_{J,V,\omega}\boldsymbol{u}](n):=J\boldsymbol{u}(n+1)+J\boldsymbol{u}(n-1)+V(\omega+n\alpha)\boldsymbol{u}(n),
\end{align*}
has purely absolutely continuous in the set
\begin{align*}
\mathcal{W}=\left\{E\in\sigma^{\mathcal{S}}(J,V):(\alpha,S_E^{V,J})\text{ is subcritical}\right\}.
\end{align*}
\end{theorem}	

\begin{remark}
    For one-frequency analytic Schr\"odinger operator, the corresponding result was proved by Avila \cite{A01,avila-kam}, thus Theorem \ref{8.4-theorem4.4} generalize this well-known result to a  class of singular Jacobi operator on the strip. 
\end{remark}

Let 
\begin{align*}
\mathcal{AR}=\left\{E\in\R:\ (\alpha,S_E^{V,J})\text{ is almost reducible}\right\},
\end{align*}	
which is an open set, i.e.,
\begin{align*}
\mathcal{AR}=\cup_{j=1}(a_j,b_j).
\end{align*}
Note by Almost Reducibility Theorem (Theorem \ref{11-26-themrem2.4}), for any $E\in \mathcal{W}$, we have $E\in \mathcal{AR}$. Thus to prove $\mathcal{S}_{J,V,\omega}$ has purely absolutely continuous spectrum in the bounded interval $(a,b)$, one only need to prove that $\mathcal{S}_{J,V,\omega}$ has purely absolutely continuous spectrum in 
\begin{align*}
\mathcal{W}(\delta_0)=\sigma^\mathcal{S}(J,V)\cap[a+\delta_0,b-\delta_0]
\end{align*}
for any sufficiently small $\delta_0>0$. Before giving the proof, one should note for any $E\in \mathcal{W}$, $(\alpha,S_E^{V,J})
$ is subcritical, automatically it is analytic, which implies that $E\notin \operatorname{Ran}(V_{22})$, thus if  $E\in \mathcal{W}(\delta_0)$, we have 
\begin{equation}\label{dist}
    \operatorname{dist}(E,\operatorname{Ran}(V_{22})) \geq \delta_0.
\end{equation}
	
Applying the Almost Reducibility Theorem (Theorem \ref{11-26-themrem2.4}), we can reduce the cocycle to the perturbative regime for any $E\in\mathcal{W}(\delta_0)$. Specifically, for any $\epsilon_0>0$ and $\alpha\in \R\backslash \Q,$ there exist constants
$\bar{h}>0$ and $\Gamma=\Gamma(\alpha,\epsilon_0)>0$ such that for any $
E\in  \mathcal{W}(\delta_0)$,
there exists $\Phi_{E}\in C^{\omega}_{\bar{h}}(\mathbb{T},PSL(2,\mathbb{R}))$ with $\|\Phi_{E}\|_{\bar{h}}<\Gamma$ satisfying
\begin{equation*}
    \Phi_{E}(\omega+\alpha)^{-1}S^{V,J}_{E}(\omega)\Phi_{E}(\omega)=R_{\Phi_{E}}e^{f_{E}(\omega)}
\end{equation*}
where $\left\|f_{E}\right\|_{\bar{h}}<\epsilon_0$ and $\left|\operatorname{deg} \Phi_{E}\right| \leq C|\ln \Gamma|$ for some constant $C=$
$C(V,\alpha)>0 .$

Then, we can apply the KAM scheme to get precise control of the growth of the cocycles in the resonant sets.  We inductively give the parameters, 
for any $h_0:=\bar{h}>\tilde{h}>0$, $\epsilon_0>0$, $\gamma>0,\sigma>0$, define 
$$
\epsilon_j= \epsilon_0^{2^j}, \quad h_j-h_{j+1}=\frac{\bar{h}-\frac{\bar{h}+\tilde{h}}{2}}{4^{j+1}}, \quad N_j=\frac{2|\ln\epsilon_j|}{h_j-h_{j+1}}.
$$
We can then use the following proposition:
\begin{proposition}\cite{wang}\label{local kam}
Let $\alpha\in DC( \gamma, \sigma)$. If
$$\left\|f_E\right\|_{h_{0}}\leq 
\epsilon_0 \leq D_0(\frac{\gamma}{\kappa^{\sigma}},\sigma) (\frac{\bar{h}-\tilde{h}}{8})^{C_0\sigma},$$
 where $D_{0} = D_0(\gamma,\sigma)$ and $C_{0}$ are constants. 
Then there exists $B_{j} \in C_{h_{j}}^{\omega}\left(\mathbb{T}, PSL(2, \mathbb{R})\right)$ with
$|\deg B_j |  \leq 2 N_{j-1}$,
such that
$$B_{j}^{-1}(\omega+\alpha) R_{\Phi_{E}}e^{f_{E}(\omega)} B_{j}(\omega)=A_{j}(E) e^{f_{j}(\omega)},$$
with  estimates  $
\left\|B_{j}\right\|_{0} \leq |\ln\epsilon_{j-1}|^{4\sigma}$, $\left\|f_{j}\right\|_{h_{j}} \leq \epsilon_{j}.$ 
Moreover,  for any $0<|n| \leq N_{j-1}$, denote 
\begin{equation*}
\begin{split}
\Lambda_{j}(n)=\left\{E\in \mathcal{W}(\delta_0):\|2 \rho(\alpha, A_{j-1}(E))- \langle n, \alpha \rangle\|_{\T}< \epsilon_{j-1}^{\frac{1}{15}}\right\}.
\end{split}
\end{equation*}
If 
 $E\in K_j:= \cup_{|n|=1}^{N_{j-1}}\Lambda_{j}(n)$, then $A_{j}(E)$ can be written as 
$$
A_{j}(E)=M^{-1} \exp \left(\begin{array}{cc}{i t_{j}} & {v_{j}} \\ {\bar{v}_{j}} & {-i t_{j}}\end{array}\right) M,
$$
where
$$M=\frac{1}{1+i}\begin{pmatrix}
1&-i\\1&i
\end{pmatrix},$$
with estimates
$$
|t_j|<\epsilon^{\frac{1}{16}}_{j-1},\ |v_{j}|<\epsilon_{j-1}^{\frac{15}{16}}.
$$
\end{proposition}
In this construction, $K_j$ denotes the set of energies where the cocycle $(\alpha, R_{\Phi_{E}}e^{f_{E}(\omega)})$ is resonant at the $j$-th KAM step. If $E\in K_j$, then we have the following characterization of its IDS and the growth behavior of the cocycles in the resonant sets: 

\begin{proposition}\cite{wang}
 Assume that $\alpha\in DC( \gamma, \sigma)$,  $E \in K_{j},$ then there exists $\tilde{n}_j \in \mathbb{Z}$ with $0<|\tilde{n}_j| <2N_{j-1}$ such that
\begin{equation*}
\|2\rho(\alpha,S_E^{V,J})-\langle \tilde{n}_j, \alpha\rangle\|_{\mathbb{T}} \leqslant 2 \epsilon_{j-1}^{\frac{1}{15}}.
\end{equation*}
Moreover,  we have 
\begin{equation*}
\sup _{0 \leq s \leq \epsilon_{j-1}^{-\frac{1}{8}}}\|(S_{E}^{V,J})_s\|_{0} \leq 4\Gamma^2 |\ln\epsilon_{j-1}|^{8\sigma}.
\end{equation*}
\end{proposition}
Combining this with Theorem \ref{11-29-theorem5.21}, one has the following corollary:
\begin{corollary}\label{lemma4.7}
 Suppose that $\alpha\in DC( \gamma, \sigma)$. If $E \in K_{j},$ then there exists $\tilde{n}_j \in \mathbb{Z}$ with $0<|\tilde{n}_j| <2N_{j-1}$ such that
\begin{equation*}
\|2\mathrm{n}(E)+\langle \tilde{n}_j, \alpha\rangle\|_{\mathbb{T}} \leqslant 2 \epsilon_{j-1}^{\frac{1}{15}}.
\end{equation*}
Moreover,  we have 
\begin{equation*}
\sup _{0 \leq s \leq \epsilon_{j-1}^{-\frac{1}{8}}}\|(S_{E}^{V,J})_s\|_{0} \leq 4\Gamma^2 |\ln\epsilon_{j-1}|^{8\sigma}.
\end{equation*}   
\end{corollary}

To study the regularity of the IDS, we need the following lemma, which is quoted from \cite{AJ1}:
\begin{lemma}\cite{AJ1}\label{ids}
Let $\alpha \in DC$,  if  $(\alpha, A)$ is  analytically almost reducible,  then for any
continuous map $B : \mathbb{T} \rightarrow SL(2, \mathbb{C}),$ we have
$$
|L(\alpha, A)-L(\alpha, B)| \leqslant C_{0}\|B-A\|_{0}^{\frac{1}{2}},
$$
where $C_{0}$ is a constant depending on $\alpha$.
\end{lemma}
Using this lemma, we can derive the following result:
\begin{lemma}\label{idsholder}
 Assume that $\alpha\in DC( \gamma, \sigma)$. Then the integrated density of states $\mathrm{n}$ is $\frac{1}{2}$-H\"older continuous on $\mathcal{W}(\delta_0)$.
\end{lemma}

\begin{proof}
Fix $\delta_0>0$, choose $0<\epsilon\ll\delta$. By Lemma \ref{ids}, we have $$\left|L(\alpha,S_{E+i\epsilon}^{V,J})-L(\alpha,S_{E}^{V,J})\right|\le C_1\epsilon^{\frac{1}{2}}.$$ On the other hand, Lemma \ref{11-27-lemma3.24} states that
\begin{align*}
L(\alpha, S_E^{V,J})=2\int\ln|E-E'|d\mathrm{n}(E')-\E(\ln|E-V_{22}(\cdot)|),
\end{align*}
combining \eqref{dist}, this implies that for every $\epsilon>0$,
\small
\begin{align*}
\left|L(\alpha,S_{E+i\epsilon}^{V,J})-L(\alpha,S_{E}^{V,J})\right|\ge&\int\ln\left(1+\frac{\epsilon^2}{(E-E')^2}\right)d\mathrm{n}(E')+\frac{1}{2}\int\ln\left(1-\frac{\epsilon^2}{(E-V_{22}(\omega))^2}\right)d\omega\\\ge&\int_{E-\epsilon}^{E+\epsilon}\ln\left(1+\frac{\epsilon^2}{(E-E')^2}\right)d\mathrm{n}(E')-\frac{\epsilon^2}{\delta^2}\\\ge&(\mathrm{n}(E+\epsilon)-\mathrm{n}(E-\epsilon))\ln2-\frac{\epsilon^2}{\delta^2},
\end{align*}
\normalsize
which gives$
\mathrm{n}(E+\epsilon)-\mathrm{n}(E-\epsilon)\le C_2(\delta)\epsilon^{\frac{1}{2}}, $
for sufficiently small $\epsilon>0$. Moreover, $\mathrm{n}$ is locally constant in the complement of $\sigma^\mathcal{S}(J,V)$, which means precisely that $\mathrm{n}$ is $\frac{1}{2}$-H\"{o}lder continuous on $\mathcal{W}(\delta_0)$.
\end{proof}

As a consequence, we can establish a lower bound estimate for $\mathrm{n}$:
\begin{lemma}\label{lemma4.10}
For $\delta_{0}>0$ small enough, if $E \in\mathcal{W}(\delta_0)$, then there exists a constant $c(\delta_0)>0$ such that for sufficiently small $\epsilon>0$, $\mathrm{n}(E+\epsilon)-\mathrm{n}(E-\epsilon)\geq c\epsilon^{\frac{3}{2}},$
where $c(\delta_0)$ is a small universal constant.
\end{lemma}

\begin{proof} 
The proof was originally developed by Avila \cite{A01} in the quasiperiodic  Schr\"{o}dinger  operator case. We provide a sketch of the proof here for completeness. Let $\delta=c \epsilon^{3 / 2}$. For any  $E \in \mathcal{W}(\delta_0),$ we have $L(\alpha,S_E^{V,J})=0.$ Then, by Lemma \ref{11-27-lemma3.24}, we obtain
\small
\begin{align*}
L(\alpha,S_{E+i\delta}^{V,J})&=\int\ln\left(1+\frac{\delta^{2}}{|E-E'|^2}\right)d \mathrm{n}(E')-\frac{1}{2}\int\ln\left(1+\frac{\delta^2}{|E-V_{22}(\omega)|}\right)d\omega\\&\le \int\ln\left(1+\frac{\delta^{2}}{|E-E'|^2}\right)d \mathrm{n}(E')-C_1\delta^2,
\end{align*}
\normalsize
where $C_1=\sup_{E\in\sigma^\mathcal{S}(J,V),\omega\in\Omega}|E-V_{22}(\omega)|/2$.
We now split the integral into four parts: $I_{1}=\int_{\left|E-E^{\prime}\right|\ge\frac{\delta_{0}}{2}}$, $ I_{2}=\int_{\epsilon\le\left|E-E^{\prime}\right|<\frac{\delta_{0}}{2}}$, $I_{3}=\int_{\epsilon^{4}\le\left|E-E^{\prime}\right|<\epsilon}$ and $I_{4}=\int_{\left|E-E^{\prime}\right|<\epsilon^{4}}$.

Clearly, we have $I_1 <\frac{4c^2\epsilon^3}{\delta_0^2}$.
For sufficiently small $\epsilon>0$, by Lemma \ref{idsholder} we have 
\begin{equation*}
\begin{split}
I_{4}=\sum_{k \geq 4} \int_{\epsilon^{k}>\left|E-E'\right|\ge\epsilon^{k+1}}  \ln \left(1+\frac{\delta^{2}}{\left|E-E'\right|^{2}}\right)d\mathrm{n}(E') 
\leq \sum_{k \geq 4} \epsilon^{\frac{k}{2}} \ln (1+c^{2} \epsilon^{1-2 k})\leq 2\epsilon^{\frac{7}{4}}.
\end{split}
\end{equation*}
We also have the estimate
$$
\begin{aligned} I_{2}  \leq \sum_{k=0}^{m} \int_{e^{-k-1}\le|E-E'|<e^{-k}} \ln\left(1+\frac{\delta^{2}}{\left|E-E'\right|^{2}}\right)d\mathrm{n}(E')\le\sum_{k=0}^{m}e^{-\frac{k}{2}}\delta^2e^{2k+2}\leq C_2c^2\delta,\end{aligned}
$$
where $m=[-\ln \epsilon]$. 
It follows that
\begin{align*}
I_3\ge L(\alpha,S_{E+i\delta}^{V,J})+C_1\delta^2-2\delta^{\frac{7}{6}}-C_2c^2\delta.
\end{align*}

Since $\sigma^\mathcal{S}(J,V)$ is bounded and $\inf_{E\in\mathcal{W}(\delta_0)}|E-V_{22}(x)|\ge\delta_0$, there exist $l>0$ such that
\begin{align*}
    |\Re \left(z-V_{11}(x)-\frac{|V_{12}(x)|^2}{z-V_{22}(x)}\right)|\le l
\end{align*}
and
\begin{align*}
    l>\Im z(1+\frac{\sup|V_{12}|}{\delta_0})>\Im \left(z-V_{11}(x)-\frac{|V_{12}(x)|^2}{z-V_{22}(x)}\right)=\Im z+\Im z\frac{|V_{12}(x)|^2}{|z-V_{22}(x)|}\ge \delta.
\end{align*}
Applying Lemma \ref{8.8-theorem4.5.}, we conclude that $L(\alpha,S_{E+i \delta}^{V,J})\geq C_3\delta$ for sufficiently small $\epsilon>0$, where $C_3$ is an universal constant. Since the constant $c$ above is consistent with our choice of $\delta$, we can shrink it such that $I_3\ge\frac{C_3}{2}\delta$. Since $I_{3} \leq C(\mathrm{n}(E+\epsilon)-\mathrm{n}(E-\epsilon)) \ln \epsilon^{-1},$
the result follows.
\end{proof}

\textbf{Proof of Theorem \ref{8.4-theorem4.4}}
Let $\mathcal{B}$ be the set of $E\in  \mathcal{W}(\delta_0) $ such that the cocycle $(\alpha, S_{E}^{V,J})$ is bounded. 
 By the subordinary theory for singular Jacobi operator  (Corollary \ref{subordinnacy}), it is enough to prove that
$\mu_{\omega}( \mathcal{W}(\delta_0)\backslash\mathcal{B})=0$ for any $\omega\in\T$.

 Let $\mathcal{R}$ be the set of $E\in  \mathcal{W}(\delta_0) $ such that $(\alpha,S_E^{V,J})$ is reducible, then $\mathcal{R}\backslash\mathcal{B}$ only contains $E$ for which $(\alpha,S_E^{V,J})$ is analytically reducible to a constant parabolic cocycle.  Recall that  for any $E\in  \mathcal{W}(\delta_0)$, by well-known result of Eliasson \cite{E92}, if 
 $\rho(\alpha, S_{E}^{V,J})$ is rational or  Diophantine w.r.t $\alpha$,  then $(\alpha,S_E^{V,J})$ is reducible. It follows that $\mathcal{R}\backslash\mathcal{B}$
  is countable. Moreover, if $E\in\mathcal{R}$, then any non-zero solution of $\mathcal{S}_{J,V,\omega}\boldsymbol{u}=E\boldsymbol{u}$, satisfies $0<\inf_{n\in\mathbb{Z}}|a(n)|^{2}+|a(n+1)|^2\le\inf_{n\in\mathbb{Z}}|\boldsymbol{u}(n)|^{2}+|\boldsymbol{u}(n+1)|^2$, so there are no eigenvalues in $\mathcal{R}$ and $\mu_{\omega}(\mathcal{R}\backslash\mathcal{B})=0$.  Therefore, it is enough to show that for sufficiently small $\delta_{0}>0$, 
$\mu_{\omega}( \mathcal{W}(\delta_0)\backslash\mathcal{R})=0.$
Note that $\mathcal{W}(\delta_0)\backslash\mathcal{R}\subset\limsup K_{m}$, by Borel-Cantelli Lemma, we only need to prove $\sum_m \mu_{\omega}(\overline{K}_{m})<\infty$. 

Let $J_{m}(E)$  be an open $\epsilon_{m-1}^{\frac{2}{45}}$  neighborhood of $E\in K_m$. By Lemma \ref{growth-cocycle} and Corollary \ref{lemma4.7}, we have
\begin{equation*}
\begin{split}
\mu_{\omega}(J_m(E))&\le \sup_{0\le s\le\epsilon_{m-1}^{-\frac{2}{45}}}\|(S_E^{V,J})_s\|_0^2|J_m(E)|\\&\le\sup_{0\le s\le\epsilon_{m-1}^{-\frac{1}{8}}}\|(S_E^{V,J})_s\|_0^2|J_m(E)|\le C|\ln\epsilon_{m-1}|^{16\sigma}\epsilon_{m-1}^{\frac{2}{45}}.
\end{split}
\end{equation*}
Let $\cup_{l=0}^rJ_m(E_l)$ be a finite subcover of $\overline{K}_m$. By refining this subcover, we can assume that every $E\in\mathbb{R}$ is contained in at most two different $J_m(E_l)$. 

On the other hand, by Corollary \ref{lemma4.7}, if $E\in K_m$, then
$\|2\mathrm{n}(E)+\langle n,\alpha\rangle \|_{\T}\le 2\epsilon_{m-1}^{\frac{1}{15}},$
for some $|n|<2N_{m-1}$.
This shows that $2\mathrm{n}({K}_{m})$ can be covered by $2N_{m-1}$  intervals $I_{s}$ of length  $2 \epsilon_{m-1}^{\frac{1}{15}}$.
By Lemma \ref{lemma4.10},
$2\mathrm{n}(J_m(E))\ge c|J_m(E)|^{\frac{3}{2}},$
thus  by our selection $|I_{s}|\leq \frac{1}{c}|2\mathrm{n}(J_{m}(E))|$   for any $s$
and $E\in{K}_{m}$, there are at most $2([\frac{1}{c}]+1)+4$ intervals $J_m(E_l)$ such that $2\mathrm{n}(J_m(E_l))$  intersects $I_{s}$.
We conclude that there are at most $2(2([\frac{1}{c}]+1)+4)N_{m-1}$
intervals $J_m(E_l)$ to cover $K_m$. Then 
\begin{equation*}
\mu_{\omega}(\overline{K}_{m})\leq \sum_{j=0}^{r}\mu_{\omega}(J_{m}(E_{j}))\le CN_{m-1}|\ln\epsilon_{m-1}|^{16\sigma}\epsilon_{m-1}^{\frac{2}{45}}< \epsilon_{m-1}^{\frac{1}{45}},
\end{equation*}
which gives $\sum_m \mu_{\omega}(\overline{K}_{m})<\infty$.

\subsection{Proof of Theorem \ref{11-28-theorem1.2}:}

We first need the following lemma to validate the assumption of Theorem \ref{8.4-theorem4.4}. 

\begin{lemma}\cite{wang}\label{11-27-lemma8.9}
    Let $\lambda>0,\alpha\in\R\backslash\Q,|\tau|<1$. Define
    \begin{align*}
        A(\omega)=\begin{pmatrix}
            E-\frac{2\lambda\cos2\pi\omega}{1-\tau\cos2\pi\omega}&-1\\1&0
        \end{pmatrix}.
    \end{align*}
     Suppose that $(\alpha,A)$ is not uniformly hyperbolic. Then $L(\alpha,A)=0\text{ and }\omega(\alpha,A)=0$ if and only if $sgn(\lambda)\tau E<2(1-\lambda).$
\end{lemma}

Once we have this, note
\begin{align*}
S_E^{V,J}(\omega)=\begin{pmatrix}E-V_{11}(\omega)-\frac{|V_{12}(\omega)|^2}{E-V_{22}(\omega)}& -1\\1&0\end{pmatrix}=\begin{pmatrix}
    \frac{E^2-\lambda^2}{E+\lambda\cos2\pi\omega}&-1\\1&0
\end{pmatrix},
\end{align*}	
and rewrite
\begin{align*}
    \frac{E^2-\lambda^2}{E+\lambda\cos2\pi\omega}=\frac{E^2-\lambda^2}{E}-\frac{\lambda(E^2-\lambda^2)\cos2\pi\omega}{E^2(1+\frac{\lambda}{E}\cos2\pi\omega)},
\end{align*}
and noting  $$|\frac{\lambda}{E}|<1,\quad\forall E\in\sigma^{\mathcal{S}}(J,V)\backslash[-\lambda,\lambda].$$
Since $E\in\sigma^{\mathcal{S}}(J,V)$ implies the cocycle $(2\alpha,S_E^{V,J})$ is not uniformly hyperbolic (Proposition \ref{11-29-proposition3.6}), we can apply Lemma \ref{11-27-lemma8.9} to obtain $\mathcal{W}=\sigma^{\mathcal{S}}(J,V)\backslash[-\lambda,\lambda].$ Then, according to Theorem \ref{8.4-theorem4.4}, the operator $\mathcal{S}_{J,V,\omega}$ has purely absolutely continuous spectrum in $\sigma^{\mathcal{S}}(J,V)\backslash[-\lambda,\lambda]$ for every $\omega\in\T$. Moreover, combining this with Theorem \ref{11-28-theorem8.2}, we complete the proof.

\section{Acknowledgements}
Y. Wang is supported by the NSFC grant (12401208) and the Natural Science Foundation of Jiangsu Province (Grants No BK20241431). Q. Zhou was partially supported by National Key R\&D Program of China (2020YFA0713300) and Nankai Zhide Foundation.

\end{document}